\documentclass[10pt]{amsart}
      \usepackage[mathscr]{eucal}
      \usepackage{amsmath,amsfonts, amssymb}
\def\mapright#1{\smash{
\mathop{\rg}\limits^{#1}}}
\def\mapdown#1{\bigg\downarrow
\rlap{$\vcenter{\hbox{$\scriptstyle#1$}}$}}

\def\rg{\hbox to 30pt{\rightarrowfill}}
\def\lg{\hbox to 30pt{\leftarrowfill}}

      \parskip\smallskipamount
          \newtheorem{theorem}{Theorem}[section]
      
      \newtheorem{proposition}[theorem]{Proposition}
      \newtheorem{corollary}[theorem]{Corollary}
      \newtheorem{lemma}[theorem]{Lemma}
      \newtheorem{example}[theorem]{Example}
      
      \newtheorem{remark}[theorem]{Remark}
      \makeatletter
      \@addtoreset{equation}{section}
      \makeatother

      \newcommand{\BB}{{\mathbb B}}
      \newcommand{\CC}{{\mathbb C}}
      \newcommand{\NN}{{\mathbb N}}
      \newcommand{\HH}{{\mathbb H}}

      \newcommand{\DD}{{\mathbb D}}
      
      \newcommand{\FF}{{\mathbb F}}

      \newcommand{\cA}{{\mathcal A}}
      
      \newcommand{\cC}{{\mathcal C}}
      \newcommand{\cD}{{\mathcal D}}
      \newcommand{\cE}{{\mathcal E}}
      \newcommand{\cF}{{\mathcal F}}
      \newcommand{\cG}{{\mathcal G}}
      \newcommand{\cH}{{\mathcal H}}
      \newcommand{\cI}{{\mathcal I}}
      \newcommand{\cK}{{\mathcal K}}
      
      \newcommand{\cM}{{\mathcal M}}
      \newcommand{\cN}{{\mathcal N}}
      \newcommand{\cP}{{\mathcal P}}
      \newcommand{\cR}{{\mathcal R}}
      \newcommand{\cS}{{\mathcal S}}
      
      \newcommand{\cU}{{\mathcal U}}
      \newcommand{\cV}{{\mathcal V}}

      \newcommand{\rank}{\hbox{\rm{rank}}\,}

      \newdimen\expt
      \def\boxit#1{\setbox0\hbox{$\displaystyle{#1}$}
            \hbox{\lower.4\expt
       \hbox{\lower3\expt\hbox{\lower\dp0
            \hbox{\vbox{\hrule height.4\expt
       \hbox{\vrule width.4\expt\hskip3\expt
            \vbox{\vskip3\expt\box0\vskip2\expt}%
       \hskip3\expt\vrule width.4\expt}\hrule height.4\expt}}}}}}
      
\begin{document}
       \pagestyle{myheadings}
      \markboth{ Gelu Popescu}{    Free biholomorphic functions and operator model theory  }

      \title [ Free biholomorphic functions and operator model theory  ]
{ Free biholomorphic functions and operator model theory  }
        \author{Gelu Popescu}
\date{April 27, 2011}
      \thanks{Research supported in part by an NSF grant}

      \address{Department of Mathematics, The University of Texas
      at San Antonio \\ San Antonio, TX 78249, USA}
      \email{\tt gelu.popescu@utsa.edu}

\bigskip

\begin{abstract}

Let $f=(f_1,\ldots, f_n)$ be  an $n$-tuple of formal power series in
noncommutative indeterminates $Z_1,\ldots, Z_n$ such that $f(0) =0$
and the Jacobian $\det J_f(0)\neq 0$, and let $g=(g_1,\ldots, g_n)$
be  its inverse with respect to the composition. We assume that  $f$
and $g$ have nonzero radius of convergence and   $g$ is a bounded
 free holomorphic
 function on the  open  unit ball
$$
[B(\cH)^n]_1:=\{X=(X_1,\ldots, X_n)\in B(\cH)^n: \ X_1X_1^*+\cdots
+ X_nX_n^*<  I\},
$$
where $B(\cH)$ is the algebra  of bounded linear operators  an a
Hilbert space $\cH$.

In this paper, several results concerning the noncommutative
multivariable operator theory on
 the unit ball  $[B(\cH)^n]_1^-$ (which corresponds to the particular case when $f=(Z_1,\ldots, Z_n)$) are extended to  the
noncommutative domain
$$\BB_f(\cH):=\{X \in B(\cH)^n: \  g(f(X))=X \text{ and }  \|f(X)\| \leq
  1\}
  $$
  for an appropriate evaluation $X\mapsto f(X)$.
We
develop an operator model theory  and dilation theory for $\BB_f(\cH)$,
  where the associated universal model  is  an $n$-tuple $(M_{Z_1},\ldots, M_{Z_n})$ of
   left  multiplication operators  acting on a Hilbert space $\HH^2(f)$
  of formal power series. We obtain  a Beurling type characterization for the joint invariant
   subspaces  under $M_{Z_1},\ldots, M_{Z_n}$ and study the representations of the algebras
   they generate: the domain algebra $\cA(\BB_f)$, the Hardy algebra $H^\infty(\BB_f)$, and
    the $C^*$-algebra $C^*(M_{Z_1},\ldots, M_{Z_n})$. We associate with each $n$-tuple
  $X\in \BB_f(\cH)$ a characteristic function $\Theta_{f,X}$ and  use it to provide a functional
   model and to show that it is a complete unitary invariant  for   the  class of  pure elements in
    $\BB_f(\cH)$. The commutant lifting theorem is extended to our
    setting and used to solve the   Nevanlinna-Pick  interpolation
    problem for the Hardy algebra $H^\infty(\BB_f)$.

   Using a noncommutative Poisson  transform  associated with the domain $\BB_f(\cH)$, we
     introduce a curvature invariant on $\BB_f(\cH)$, which turns out  to be a complete numerical
      invariant for the finite rank submodules of the free Hilbert module $\HH^2(f)\otimes \cK$,
      where $\cK$ is finite dimensional.

  When $f$ and $g$ are $n$-tuples of noncommutative polynomials (or certain free holomorphic functions),
   one can give up the condition $f(0)=0$. All the results of
   this paper have commutative versions.
\end{abstract}

      \maketitle

\bigskip

\section*{Contents}
{\it

\quad Introduction

\begin{enumerate}
   \item[1.]   Inverse mapping theorem for free holomorphic   functions
   \item[2.]  Polynomial automorphisms of $B(\cH)^n$
 \item[3.]  Hilbert spaces of noncommutative formal power series
\item[4.]  Noncommutative domains $\BB_f(\cH)$ and the universal model $(M_{Z_1},\ldots, M_{Z_n})$
\item[5.]  The invariant subspaces under $M_{Z_1},\ldots, M_{Z_n}$
\item[6.]
The Hardy algebra $H^\infty(\BB_f)$ and the eigenvectors of
$M_{Z_1}^*,\ldots, M_{Z_n}^*$
\item[7.]  Characteristic functions  and functional models
\item[8.]  Curvature invariant on  $\BB_f(\cH)$
\item[9.]  Commutant lifting and interpolation
   \end{enumerate}

\quad References

}

\bigskip

\bigskip
\bigskip





\bigskip

\section*{Introduction}

In the last sixty years, the study of the
unit ball of the algebra $B(\cH)$, of all bounded linear operators
on a Hilbert space, has generated the celebrated Sz.-Nagy--Foia\c s
  theory of contractions \cite{SzF-book} and has had profound implications  in
mathematics and applied mathematics. In the last three  decades, a
{\it free analogue} of Sz.-Nagy--Foia\c s theory on the unit ball of
$B(\cH)^n$ has been pursued by the author and others (see \cite{F},
\cite{B}, \cite{Po-isometric}, \cite{Po-charact}, \cite{Po-multi},
\cite{Po-von}, \cite{Po-funct}, \cite{Po-analytic}, \cite{Po-disc},
\cite{Po-interpo}, \cite{ArPo2},   \cite{DP}, \cite{Arv},
\cite{Po-entropy}, \cite{Po-varieties}, \cite{Po-varieties2},
\cite{Po-unitary},
   \cite{BT1}, \cite{BT2}, \cite{BES1}).
 This theory
has  already had remarkable
   applications in complex interpolation on the unit ball of $\CC^n$,  multivariable
   prediction  and entropy optimization, control theory,
 systems theory, scattering theory, and
wavelet theory.
 On the other hand, it has  been a source of inspiration
 for the development of
 several other areas of research  such as tensor algebras over
 $C^*$-correspondences and  free semigroup (resp.\,semigroupoid, graph)
 algebras (see \cite{MuSo1},  \cite{MuSo2}, \cite{MuSo3}).

The present paper is an attempt  to find large classes of noncommutative multivariable functions
 $g:\Omega\subset [B(\cH)^n]_1^-\to B(\cH)^n$ for which  a reasonable
  operator model theory  and dilation theory  can be developed  for
   the   noncommutative domain $g(\Omega)$. In other words, we want to {\it transfer}
   the free analogue of Nagy-Foia\c s theory from the unit ball $[B(\cH)^n]_1$ to other
    noncommutative domains in $B(\cH)^n$, using appropriate  maps.

In Section 1, we obtain inverse mapping theorems for formal power
series in noncommutative indeterminates $Z_1,\ldots, Z_n$, and also
for free holomorphic functions. More precisely, we show that an
$n$-tuple $f=(f_1,\ldots, f_n)$ of formal power series with $f(0)=0$
has an inverse $g=(g_1,\ldots, g_n)$ with respect to the composition
if and only if the Jacobian $\det J_f(0)\neq 0$. If, in addition,
$f$ and $g$ have nonzero radius of convergence, we prove that there
are open neighborhoods $D$ and $G$ of $0$ in $B(\cH)^n$ such that
$f|_D:D\to G$ and $g|_G :G\to D$ are free holomorphic functions
inverses of each other.

Let $f=(f_1,\ldots, f_n)$ be an $n$-tuple of  formal power series in
indeterminates $Z_1,\ldots, Z_n$  such that  $f(0)=0$ and $\det
J_f(0)\neq 0$, and assume that $f$ and its inverse $g=(g_1,\ldots,
g_n)$ have nonzero radius of convergence.  Due to  a re-scaling, we
can assume  without loss of generality that $g$ is a bounded
 free holomorphic
 function on the  open  unit ball
$$
[B(\cH)^n]_1:=\{X=(X_1,\ldots, X_n)\in B(\cH)^n: \ X_1X_1^*+\cdots +
X_nX_n^*<  I\}.
$$
We consider the noncommutative domain
$$\BB_f(\cH):=\{X \in B(\cH)^n: \  g(f(X))=X \text{ and }  \|f(X)\| \leq
  1\}
  $$
  for an appropriate evaluation $X\mapsto f(X)$ and using the
  functional calculus for row contractions to define $g(f(X))$.
  We remark that the domain above makes sense if  we give up the condition $f(0)=0$
  and ask instead that
   $f$ and $g$ be $n$-tuples of noncommutative polynomials
   or certain free holomorphic functions.
   In this paper, several results concerning the  noncommutative multivariable operator theory on
 the unit ball  $[B(\cH)^n]_1^-$   are extended to  the
noncommutative domain
   $\BB_f(\cH)$.

In Section 2 and Section 3, we introduce three classes of $n$-tuples
$f=(f_1,\ldots, f_n)$ for which an operator model theory and
dilation theory  for the domain $\BB_f(\cH)$ will be developed in
the coming sections. These classes consist of noncommutative
polynomials, formal power series with $f(0)=0$, and free holomorphic
functions, respectively. When $f$ belongs to any of these classes,
we say that it has the {\it model property}. In this case, each
domain $\BB_f$ has a universal model $(M_{Z_1},\ldots, M_{Z_n})$ of
multiplication operators acting on a Hilbert space $\HH^2(f)$ of
formal power series.

In Section 4, we show that $T=(T_1,\ldots, T_n)\in B(\cH)^n$ is a
pure $n$-tuple of operators in $\BB_f(\cH)$ if and only if there
exists a Hilbert space $\cD$ and  a co-invariant
 subspace $\cM\subseteq \HH^2(f)\otimes \cD$
 under $M_{Z_1}\otimes I_\cD,\ldots,M_{Z_n}\otimes I_\cD$ such that
 the $n$-tuple
$(T_1,\ldots, T_n)$ is unitarily equivalent  to
$$
(P_\cM(M_{Z_1}\otimes I_\cD)|_\cM,\ldots, P_\cM(M_{Z_n}\otimes
I_\cD)|_\cM).
$$
The $C^*$-algebra $C^*(M_{Z_1},\ldots, M_{Z_n})$ turns out to be
irreducible and
$$
M_{Z_i}^* M_{Z_j}=\left< Z_j,Z_i\right>_{\HH^2(f)}
I_{\HH^2(f)},\qquad i,j\in \{1,\ldots,n\}.
$$
If, in addition, $f$ has the radial approximation property, that is,
there is $\delta\in (0,1)$ such that $rf$ has the model property for
any $r\in (\delta, 1)$, we prove that, for any $T:=(T_1,\ldots,
T_n)\in  \BB_f(\cH)$,
 there is
    a unique unital completely contractive linear map
$$
\Psi_{f,T}: C^*(M_{Z_1},\ldots, M_{Z_n})\to B(\cH)
$$
such that
 $$
\Psi_{f,T}(M_{Z_\alpha} M_{Z_\beta}^*)=T_\alpha T_\beta^*, \qquad
\alpha,\beta\in \FF_n^+,
$$
where $T_\alpha:=T_{i_1}\cdots T_{i_k}$ if $\alpha=g_{i_1}\cdots g_{i_k}$ is a word
in the free semigroup $\FF_n^+$ with generators $g_1,\ldots, g_n$.
 As a consequence we obtain   a minimal dilation of $T$ which is unique up to an
isomorphism.

We introduce the domain algebra $\cA(\BB_f)$ as the norm-closure of
all polynomials in $M_{Z_1},\ldots, M_{Z_n}$ and the identity.  Under natural conditions on $f$, we use
Paulsen's similarity result \cite{Pa} to obtain a characterization
for the completely bounded representations of $\cA(\BB_f)$. We also
show that the set $M_{\cA(\BB_f)}$ of all characters of $\cA(\BB_f)$
is homeomorphic to $g(\overline{\BB}_n)$, where $\overline{\BB}_n$
is the closed unit ball of $\CC^n$.

In Section 5, we provide a Beurling \cite{Be} type characterization
of the invariant subspaces  under the    multiplication operators
$M_{Z_1},\ldots, M_{Z_n}$ associated with the noncommutative domain
$\BB_f$. More precisely, we show that if
 $f=(f_1,\ldots, f_n)$ is an  $n$-tuple of formal power series with
the model property, then   a  subspace $\cN\subseteq \HH^2(f)\otimes
\cH$ is invariant under each operator $M_{Z_1}\otimes I_\cH,\ldots,
M_{Z_n}\otimes I_\cH$  if and only if there exists an inner
multi-analytic operator $\Psi:\HH^2(f)\otimes \cE\to \HH^2(f)\otimes
\cH$ with respect to $M_{Z_1},\ldots, M_{Z_n}$, i.e., $\Psi$ is an
 isometry  and $\Psi(M_{Z_i}\otimes I_\cE)=(M_{Z_i}\otimes
I_\cH)\Psi$ \, for any $i=1,\ldots,n$, such that
$$
\cN=\Psi[\HH^2(f)\otimes \cE].
$$
 Using some of the results of this section and noncommutative Poisson transforms associated with the
 noncommutative domain $\BB_f$, we provide   a minimal dilation theorem for pure $n$-tuples of
operators in $\BB_f(\cH)$, which turns out to be unique up to an
isomorphism.

In Section 6, we show that  the eigenvectors for $M_{Z_1}^*,\ldots,
M_{Z_n}^*$ are precisely the   noncommutative Poisson kernels
associated with the noncommutative domain $\BB_f$  at the points in
the set $$ \BB_{f}^{<}(\CC):=\{\lambda\in \CC^n:\
g(f(\lambda))=\lambda \text{ and } \|f(\lambda)\|< 1 \}, $$
 that is,
the
 formal power series
$$
\Gamma_\lambda:= \left(1-\sum_{i=1}^n
|f_i(\lambda)|^2\right)^{1/2}\sum_{\alpha\in \FF_n^+}
[\overline{f(\lambda)}]_\alpha\
  f_\alpha,\qquad \lambda\in \BB_{f}^{<}(\CC).
$$
Moreover, they satisfy the equations $ M_{Z_i}^*
\Gamma_\lambda=\overline{\lambda}_i \Gamma_\lambda$,
 $i=1,\ldots,n$. We define the noncommutative Hardy algebra $H^\infty(\BB_f)$
 to be the WOT-closure of all
  noncommutative polynomials in $M_{Z_1},\ldots, M_{Z_n}$ and the
  identity, and show that it coincides with the algebra of bounded left multipliers of
  $\HH^2(f)$.
The symmetric Hardy space $\HH_s^2(f)$
  associated with the noncommutative
domain $\BB_f$  is defined as the subspace $\HH^2(f)\ominus
\overline{J_c(1)}$, where $J_c$ is the WOT-closed two-sided ideal of
the Hardy algebra $H^\infty(\BB_f)$ generated by the commutators
$$ M_{Z_i}M_{Z_j}-M_{Z_j}M_{Z_i},\qquad i,j=1,\ldots, n.
$$
We show that $\HH_s^2(f)=\overline{\text{\rm span}}\{\Gamma_\lambda:
\ \lambda\in
 \BB_{f}^<(\CC)\}$ and can be identified with  a Hilbert space $H^2(\BB_{f}^<(\CC))$ of
  holomorphic functions on
$\BB_{f}^<(\CC)$,
 namely, the
 reproducing kernel
 Hilbert space   with
 reproducing
 kernel
$\Lambda_f:\BB_{f}^<(\CC)\times \BB_{f}^<(\CC)\to \CC$ defined by
$$
\Lambda_f(\mu,\lambda):= \frac{1}{1-\sum_{ i=1}^n   f_i(\mu)
\overline{f_i(\lambda)}},\qquad \lambda,\mu\in \BB_{f}^<(\CC).
$$
The algebra $P_{H_s^2(f)} H^\infty(\BB_f)|_{\HH_s^2(f)}$
  coincides with the WOT-closed algebra  generated by
    the operators $L_i:=P_{\HH_s^2(f)} M_{Z_i}|_{\HH_s^2(f)}$, $i=1,\dots, n$, and
 can be identified
 with the algebra of  all multipliers of the  Hilbert space $H^2(\BB_{f}^<(\CC))$.
 Under this identification the  operators $L_1,\ldots, L_n$ become the
  multiplication operators $M_{z_1},\ldots, M_{z_n}$ by the coordinate
   functions.  The $n$-tuple $(L_1,\ldots, L_n)$ turns out to be the
   universal model for the commutative $n$-tuples from $\BB_f(\cH)$.

In Section 7, we introduce
  the characteristic  function  of  an  $n$-tuple $T=(T_1,\ldots, T_n)\in
  \BB_f(\cH)$ to be
  a certain multi-analytic operator
$\Theta_{f,T}:\HH^2(f)\otimes \cD_{f,T^*}\to \HH^2(f)\otimes
\cD_{f,T} $
 with respect to $M_{Z_1},\ldots, M_{Z_n}$, and point out a natural connection with the characteristic function of a row contraction \cite{Po-charact}. We present a model for  pure $n$-tuples of operators  in the
noncommutative domain $\BB_f(\cH)$  in terms of characteristic
functions, and show that the   characteristic function is a complete
unitary invariant for pure $n$-tuples of operators  in $\BB_f(\cH)$.

Using ideas from \cite{Po-curvature}, we introduce, in Section 8,
the  curvature invariant of  $T=(T_1,\ldots, T_n)\in \BB_f(\cH)$
 by setting
$$
\text{\rm curv}_f(T):=\lim_{m\to \infty} \frac{\text{\rm
trace}~[K_{f,T}^*(Q_{\leq m}\otimes I_{\cD_{f,T}})K_{f,T}]} {
\text{\rm trace}~[K_{f,M_Z}^*(Q_{\leq m})K_{f,M_Z}]},
$$
where $K_{f,T}$ is the noncommutative Poisson kernel associated with $T$,
and $Q_{\leq m}$, $m=0,1,\ldots,$ is  the orthogonal projection of
$\HH^2(f)$ on the linear span of the formal power series $f_\alpha$,
$\alpha\in \FF_n^+$ with $|\alpha|\leq m$.  We show that the limit
exists and  provide an index type formula for the curvature in terms
of the characteristic function. One of the main goals of this
section is  to show that  the curvature is  a complete numerical
      invariant for the finite rank submodules of the free Hilbert module $\HH^2(f)\otimes \cK$,
      where $\cK$ is finite dimensional. Here, the
        Hilbert module structure of $\HH^2(f)$ over $\CC[Z_1,\ldots,
Z_n]$ is  defined by the universal model $(M_{Z_1},\ldots, M_{Z_n})$
by setting
$$
p\cdot h:=p(M_{Z_1},\ldots, M_{Z_n})h,\qquad p\in \CC[Z_1,\ldots, Z_n] \
\text{ and } \  h\in \HH^2(f).
$$
In our setting, the Hilbert module  $\HH^2(f)$ occupies the position
of the rank-one free module in the algebraic theory \cite{K}.

In Section 9,  we use the commutant lifting theorem for row
contractions \cite{Po-isometric}, to deduce  an analogue for  the
pure $n$-tuples of operators in the noncommutative domain
$\BB_f(\cH)$. As a consequence,  and using the results from Section
6,
  we solve the  Nevanlinna Pick  interpolation problem  for the noncommutative
  Hardy algebra $H^\infty(\BB_f)$.
   We show that if    $\lambda_1,\ldots, \lambda_m$ are  $m$
distinct points in    $\BB_{f}^<(\CC)$ and   $A_1,\ldots, A_m\in
B(\cK)$, then there exists $\Phi \in H^\infty(\BB_f)\bar\otimes
B(\cK)$ such that
$$\|\Phi \|\leq 1\quad \text{and }\quad
\Phi(\lambda_j)=A_j,\quad j=1,\ldots, m,
$$
if and only if the operator matrix
\begin{equation*}
\left[  \frac{I_\cK-A_i A_j^*}{1-\sum_{ k=1}^n   f_k(\lambda_i)
\overline{f_k(\lambda_j)}}\right]_{m\times m}
\end{equation*}
is positive semidefinite.

We remark that, using the results from Section 6, we can provide
commutative versions for all the results of the present paper.
Moreover,   a model theory and   dilation theory  for not
necessarily pure $n$-tuples of operators in the noncommutative
domain $\BB_f(\cH)$  (resp.\,varieties in  $\BB_f(\cH)$) is
developed in a sequel    to the present paper.

\bigskip

\section{  Inverse mapping theorem for free holomorphic functions }

Initiated in \cite{Po-holomorphic}, the theory of free holomorphic
(resp.~pluriharmonic) functions on the unit ball of $B(\cH)^n$,
where $B(\cH)$ is the algebra of all bounded linear operators on  a
Hilbert space $\cH$,
   has been developed   very
recently (see   \cite{Po-holomorphic} \cite{Po-hyperbolic},
\cite{Po-hyperbolic3}, \cite{Po-unitary},  \cite{Po-pluriharmonic},
  \cite{Po-automorphism}, \cite{Po-holomorphic2}) in the attempt  to
provide a framework for the study of arbitrary
 $n$-tuples of operators on a Hilbert space.
  Several classical
 results from complex analysis and  hyperbolic geometry   have
 free analogues in
 this  noncommutative multivariable setting. Related to our work, we
 mention  the papers \cite{HKMS},   \cite{MuSo2},
 \cite{MuSo3}, and  \cite{V}, where several aspects of the theory
 of noncommutative analytic functions are considered in various
 settings.

In this section, we obtain inverse mapping theorems for formal power
series in noncommutative indeterminates   and   for free holomorphic
functions. We recall \cite{Po-holomorphic} that  a free holomorphic
functions on the open operatorial $n$-ball of radius $\gamma>0$ (or
$\gamma=\infty$) is defined
 as a  formal power series $f=\sum_{\alpha\in
\FF_n^+}a_\alpha Z_\alpha$ in noncommutative indeterminates $Z_1,\ldots, Z_n$
with radius of convergence $r(f)\geq \gamma$,
i.e.,
 $\{a_\alpha\}_{\alpha\in \FF_n^+}$ are complex numbers  with
$r(f)^{-1}:=\limsup_{k\to\infty} \left(\sum_{|\alpha|=k}
|a_\alpha|^2\right)^{1/2k}\leq 1/\gamma,$
 where $\FF_n^+$ is the
free semigroup with $n$ generators  $g_1,\ldots, g_n$ and the identity $g_0$.  The length
of $\alpha\in \FF_n^+$ is defined by $|\alpha|:=0$ if $\alpha=g_0$
and $|\alpha|:=k$ if
 $\alpha=g_{i_1}\cdots g_{i_k}$, where $i_1,\ldots, i_k\in \{1,\ldots, n\}$.
If $(X_1,\ldots, X_n)\in B(\cH)^n$, we denote $X_\alpha:=
X_{i_1}\cdots X_{i_k}$ and $X_{g_0}:=I_\cH$.  A free holomorphic
function $f$ on  the open ball
$$
[B(\cH)^n]_\gamma:=\left\{ (X_1,\ldots, X_n)\in B(\cH)^n: \
\|X_1X_n^*+\cdots + X_nX_n^*\|^{1/2}<\gamma\right\},
$$
 is
the representation of   $f $~ on the Hilbert space $\cH$, that is,
the mapping
$$
[B(\cH)^n]_\gamma\ni (X_1,\ldots, X_n)\mapsto f(X_1,\ldots,
X_n):=\sum_{k=0}^\infty \sum_{|\alpha|=k}
 a_\alpha X_\alpha\in
B(\cH),
$$
where  the convergence is in the operator norm topology.
   Due to the fact that a free holomorphic function is
uniquely determined by its representation on an infinite dimensional
Hilbert space,  throughout this paper, we  identify  a free
holomorphic function   with its representation on a   separable
infinite dimensional Hilbert space.

 A free holomorphic function $f$ on $[B(\cH)^n]_\gamma$ is
bounded if $
 \|f\|_\infty:=\sup  \|f(X)\|<\infty,
  $
where the supremum is taken over all $X\in [B(\cH)^n]_\gamma$ and
$\cH$ is an infinite dimensional Hilbert space. Let $H^\infty_{{\bf
ball}_\gamma}$ be the set of all bounded free holomorphic functions
and let $A_{{\bf ball}_\gamma}$ be the set of all elements $f\in
H^\infty_{{\bf ball}_\gamma }$ such that the mapping
$$[B(\cH)^n]_\gamma\ni (X_1,\ldots, X_n)\mapsto f(X_1,\ldots, X_n)\in B(\cH)$$
 has a continuous extension to the closed  ball $[B(\cH)^n]^-_\gamma$.
We  showed in \cite{Po-holomorphic} that $H^\infty_{{\bf
ball}_\gamma}$ and $A_{{\bf ball}_\gamma}$ are Banach algebras under
pointwise multiplication and the norm $\|\cdot \|_\infty$.

 For each $i=1,\ldots, n$,  we define the free  partial derivation  $\frac{\partial } {\partial Z_i}$  on $\CC[Z_1,\ldots, Z_n]$, the algebra of noncommutative polynomials with complex coefficients and indeterminats $Z_1,\ldots, Z_n$, as the unique linear operator  on this algebra, satisfying the conditions
$$
\frac{\partial I} {\partial Z_i}=0, \quad  \frac{\partial Z_i}
{\partial Z_i}=I, \quad  \frac{\partial Z_j} {\partial Z_i}=0\
\text{ if }  \ i\neq j,
$$
and
$$
\frac{\partial (\varphi \psi )} {\partial Z_i}=\frac{\partial
\varphi} {\partial Z_i} \psi +\varphi\frac{\partial \psi} {\partial
Z_i}
$$
for any  $\varphi,\psi\in \CC[Z_1,\ldots, Z_n]$ and $i,j=1,\ldots
n$. Note that if $\alpha=g_{i_1}\cdots g_{i_p}$, $|\alpha|=p$, and
$q$ of the $g_{i_1},\ldots, g_{i_p}$ are equal to $g_j$, then
$\frac{\partial Z_\alpha} {\partial Z_j}$ is the sum of the $q$
words obtained by deleting each occurrence of $Z_j$ in
$Z_\alpha:=Z_{i_1}\cdots Z_{i_p}$ and replacing it by  the identity
$I$. The directional derivative of $Z_\alpha$ at $Z_i$ in the
direction $Y$, denoted by $\left(\frac{\partial Z_\alpha} {\partial
Z_i}\right)[Y]$,    is defined similarly by replacing each
occurrence of $Z_j$ in $Z_\alpha:=Z_{i_1}\cdots Z_{i_p}$  by  $Y$
(see \cite{He-C-V}). Note that $\frac{\partial Z_\alpha} {\partial
Z_i} =\left(\frac{\partial Z_\alpha} {\partial Z_i}\right)[I]$.
These   definitions extend to formal power series in the
noncommuting indeterminates $Z_1,\ldots, Z_n$.
  If $F:=\sum\limits_{\alpha\in \FF_n^+} a_\alpha Z_\alpha $ is a
  power series, then   the    free partial derivative
    of $F$ with respect to $Z_{i} $ is the power series
$
\frac {\partial F}{\partial Z_{i}} := \sum_{\alpha\in
\FF_n^+}a_\alpha
 \frac {\partial Z_\alpha}{\partial Z_{i} } .
$

We denote by ${\bold S}[Z_1,\ldots,Z_n]$ the algebra of all formal
power series in noncommuting indeterminates $Z_1,\ldots, Z_n$ and
complex coefficients. We remark that, for any power series $G\in
{\bold S}[Z_1,\ldots,Z_n]$,
$$
\left(\frac {\partial F}{\partial Z_{i}}\right)[G] :=
\sum_{\alpha\in \FF_n^+}a_\alpha
 \left(\frac {\partial Z_\alpha}{\partial Z_{i} }\right)[G].
$$
is a power series in ${\bold S}[Z_1,\ldots,Z_n]$. Indeed, it is
enough to notice that all the  monomials of degree $m\geq 1$ in
$Z_1,\ldots, Z_n$ occur   in the sum $\sum_{k=0}^{m+1}
\sum_{|\alpha|=k} \left(\frac {\partial Z_\alpha}{\partial Z_{i}
}\right)[G]$. Consequently, we can use the directional derivative of
$F$ at $Z_i$ to define the mapping
$$\left(\frac {\partial F}{\partial Z_{i}}\right): {\bold
S}[Z_1,\ldots,Z_n]\to {\bold S}[Z_1,\ldots,Z_n],\qquad
 g\mapsto\left(\frac
{\partial F}{\partial Z_{i}}\right)[G].
$$
Let $H$ be a   formal power series in indeterminates  $W_1,\ldots,
W_n$ and let $G=(G_1,\ldots, G_n)$ be an $n$-tuple of formal power
series in indeterminates  $Z_1,\ldots, Z_n$ with $G(0)=0$.  Then we
have the following chain rule
$$
\frac {\partial (H\circ G)}{\partial Z_{i}} =\sum_{k=1}^n
\left\{\left(\frac {\partial H}{\partial W_{k}}\right)\left[ \frac
{\partial G_k }{\partial Z_{i}}\right]\right\}_{W=G(Z)},
$$
where $Z=(Z_1,\ldots, Z_n)$ and $W=(W_1,\ldots,W_n)$. Indeed, it is
enough to prove this rule when $H=W_\alpha$ and
$\alpha:=g_{i_1}\cdots g_{i_k}\in \FF_n^+$. Note that, for each
$i=1,\ldots,n$, we have
$$
\frac {\partial G_\alpha}{\partial Z_{i}} =\sum_{k=1}^n
\left\{\left(\frac {\partial W_\alpha}{\partial W_{k}}\right)\left[
\frac {\partial G_k }{\partial Z_{i}}\right]\right\}_{W=G(Z)}.
$$
Let $F:=(F_1,\ldots, F_n)$ be an $n$-tuple of formal power
series in indeterminates  $W_1,\ldots, W_n$. We define the Jacobian
matrix  of $F$ to be $J_F:=\left[\frac {\partial F_i}{\partial
W_{j}}\right]_{n\times n}$ with entries in ${\bold
S}[W_1,\ldots,W_n]$. Note that
$$
J_{F\circ G}=\left[\sum_{k=1}^n \left\{\left(\frac {\partial
F_i}{\partial W_{k}}\right)\left[ \frac {\partial G_k }{\partial
Z_{j}}\right]\right\}_{W=G(Z)}\right]_{n\times n},
$$
which,  symbolically, can be written   as
$$(J_F[\cdot])\, \lozenge\,
J_G=\left[\left(\frac {\partial F_i}{\partial W_{k}}\right)[ \cdot
]\right]_{n\times n}\lozenge \left[\frac {\partial G_k}{\partial
Z_{j}}\right]_{n\times n}, $$ which is the substitute for the matrix
multiplication from the commutative case.
In particular, we can easily deduce the following result.

\begin{lemma}\label{Jacobian}
Let $F:=(F_1,\ldots, F_n)$ and $G:=(G_1,\ldots, G_n)$ be formal
power series in $n$-indeterminates and such that $G(0)=0$. Then
$$
J_{F\circ G}(0)=J_{F}(0)J_{G}(0).
$$
If  $F$ is an $n$-tuple  of noncommutative polynomials, the
condition $G(0)=0$ is not necessary.
\end{lemma}

\begin{theorem}
\label{Schroder} Let $f=(f_1,\ldots, f_n)$ be  an $n$-tuple of
formal power series in indeterminates  $Z_1,\ldots, Z_n$  and with
the property that $$\det J_f(0):=\det\left[\left.\frac {\partial
f_i}{\partial Z_{j}}\right|_{Z=0}\right]\neq 0. $$
 Then  the set
$\{f_\alpha\}_{\alpha\in \FF_n^+}$   is  linearly independent
   in
${\bold S}[Z_1,\ldots,Z_n]$.
\end{theorem}
\begin{proof}  First, we consider the case when $f(0)=0$.
 Let $A:=J_f(0)^{t}$, where $t$ stands for the transpose,
and let $f=G=[G_1,\ldots G_n]$ be an $n$-tuple of power series in
noncommuting indeterminates $Z_1,\ldots, Z_n$,  of the form
$$
G=[Z_1,\ldots, Z_n]A+[Q_1,\ldots, Q_n],
$$
where $Q_1, \ldots, Q_n$ are  noncommutative power series containing
only monomials  of degree greater than or equal to $2$. In what
follows, we prove that the composition  map $C_G:{\bold
S}[Z_1,\ldots,Z_n]\to {\bold S}[Z_1,\ldots,Z_n]$ defined by $C_G
\Psi:=\Psi\circ G$ is an injective homomorphism.  Let $F$ be a
formal power series   such that $ F\circ G=0$. Since $A\in
M_{n\times n}$ there is a unitary matrix $U\in M_{n\times n}$ such
that
  $U^{-1} A U$ is an upper triangular matrix. Setting
$\Phi_U:=[Z_1,\ldots, Z_n] U$, the equation $F\circ G=0$ is
equivalent to $F'\circ G'=0$, where $F':=\Phi_U\circ F \circ
\Phi_{U^{-1}}$ and
$$
G':=\Phi_U\circ G \circ \Phi_{U^{-1}} =[Z_1,\ldots,
Z_n]U^{-1}AU+U^{-1}[Q_1,\ldots, Q_n] U.
$$
 Therefore,  we can assume that $A=[a_{ij}]\in M_{n\times n}$ is  an invertible upper triangular matrix
  and, therefore $a_{ii}\neq 0$ for any $i=1,\ldots,n$..
We introduce a total order $\leq $ on the free semigroup $\FF_n^+$
as follows. If $\alpha, \beta\in \FF_n^+$ with $|\alpha|< |\beta|$
we say that $\alpha<\beta$. If $\alpha,\beta\in \FF_n^+$ are such
that $|\alpha|=|\beta|$, then $\alpha=g_{i_1}\cdots g_{i_k}$ and
$\beta=g_{j_1}\cdots g_{j_k}$ for some $i_1,\ldots, i_k,
j_1,\ldots,j_k\in \{1,\ldots,k\}$. We say that $\alpha<\beta$ if
either $i_1<j_1$ or there exists $p\in \{2,\ldots,k\}$ such that $
i_1=j_1,\ldots, i_{p-1}=j_{p-1}$ and $i_p<j_p$.  The  relation
$\leq$ is a total order on $\FF_n^+$.
According to the hypothesis and due to the fact that  $A$ is an
upper triangular matrix, we have
\begin{equation} \label{La}
G_j=\sum_{i=1}^j a_{ij} Z_i+ Q_j,\qquad j=1,\ldots,n.
\end{equation}
Consequently, if $\alpha=g_{i_1}\cdots g_{i_k}\in \FF_n^+$,
$i_1,\ldots i_k\in \{1,\ldots,n\}$, then
\begin{equation} \label{La-al}
G_\alpha:=G_{i_1}\cdots G_{i_k}=L^{<\alpha}+a_{i_1 i_1}\cdots a_{i_k
i_k} Z_\alpha +\Psi^{(\alpha)},
\end{equation}
where $L^{<\alpha}$ is a power series containing only monomials
$Z_\beta$ such that $|\beta|=|\alpha|$ and  $\beta<\alpha$, and
$\Psi^{(\alpha)}$ is a power series containing only monomials
$Z_\gamma$ with $|\gamma|\geq |\alpha|+1$.

Now, assume that $F$ has the representation   $F=\sum_{p=0}^\infty
\sum_{|\alpha|=p} c_\alpha Z_\alpha$, $c_\alpha\in \CC$,   and
satisfies the equation $ F\circ G=0$. We will show by induction over
$p$, that
 $\sum_{|\alpha|=p} c_\alpha Z_\alpha=0$ for any $p=0,1,\ldots$. Note that the above-mentioned equation is equivalent to
 \begin{equation}\label{Sch2}
 \sum_{p=0}^\infty \sum_{|\alpha|=p} c_\alpha G_\alpha=0.
 \end{equation}
Due to relation \eqref{La}, we have $c_0=0$. Assume that
$c_\alpha=0$ for any $\alpha\in \FF_n^+$ with $|\alpha|< k$.
  According to  equations \eqref{La-al} and \eqref{Sch2}, we have

\begin{equation*}
\sum_{|\alpha|=k} c_\alpha \left(L^{<\alpha}+ d_A(\alpha) Z_\alpha
+\Psi^{(\alpha)}\right)+ \sum_{p=k+1}^\infty \sum_{|\alpha|=p}
c_\alpha G_\alpha=0,
 \end{equation*}
where $d_A(\alpha):=a_{i_1 i_1}\cdots a_{i_k i_k}$ if
$\alpha=g_{i_1}\cdots g_{i_k}\in \FF_n^+$ and $i_1,\ldots i_k\in
\{1,\ldots,n\}$. Since $\Psi^{(\alpha)}$ is a power series
containing only monomials $Z_\gamma$ with $|\gamma|\geq |\alpha|+1$,
and  the power series $G_\alpha$, $|\alpha|\geq k+1$, contains only
monomials $Z_\sigma$ with $|\sigma|\geq k+1$, we deduce that

\begin{equation} \label{sum-red}
\sum_{|\alpha|=k} c_\alpha \left(L^{<\alpha}+ d_A(\alpha) Z_\alpha
 \right) =0.
 \end{equation}
We arrange  the elements of  the set $\{\alpha\in \FF_n^+:
|\alpha|=k\}$ increasingly with respect to the total  order, i.e.,
$\beta_1<\beta_2<\cdots <\beta_{n^k}$.  Note that $\beta_1= g_1^k$
and $\beta_{n^k}=g_n^k$. The relation \eqref{sum-red} becomes

\begin{equation} \label{sum-red2}
\sum_{j=1}^{n^k}\left(c_{\beta_j} L^{< \beta_j}+ c_{\beta_j}
d(\beta_j) Z_{ {\beta_j}}\right) =0.
\end{equation}
Taking into account that $L^{<\alpha}$ is a power series containing
only monomials $Z_\beta$ such that $|\beta|=|\alpha|$ and
$\beta<\alpha$, one can see that the
 monomial $Z_{\beta_{n^k}}$ occurs just once in   relation \eqref{sum-red2}.
 Consequently, we  must  have
$ c_{\beta_{n^k}} d(\beta_{n^k})=0.
 $
 Since $0\neq a_{nn}^k=d(\beta_{n^k})$, we must have $c_{\beta_{n^k}}=0$.
 Then, equation \eqref{sum-red2} becomes
 \begin{equation*}
\sum_{j=1}^{n^k-1}\left(c_{\beta_j} \Psi^{< \beta_j}+ c_{\beta_j}
d(\beta_j) Z_{ {\beta_j}}\right) =0.
\end{equation*}
Continuing the process, we deduce that $c_{\beta_j}=0$ for
$j=1,\ldots, n^k$. Therefore $c_\alpha=0$ for any $\alpha\in
\FF_n^+$ with $|\alpha|=k$, which completes our induction. This shows that $F=0$.

Now, we consider the  case when $f(0)\neq 0$. Then $f_i=f_i(0)I+
G_i$, $i=1,\ldots,n$, for some $n$-tuple $G=[G_1,\ldots G_n]$ of
formal power series in ${\bold S}[Z_1,\ldots,Z_n]$ with $G(0)=0$.
According to the first part of the proof, the set
 $\{G_\alpha\}_{\alpha\in \FF_n^+}$ is linearly independent in ${\bold S}[Z_1,\ldots,Z_n]$.
  Consequently,
 setting $\cM_k:=\text{\rm span}\{G_\alpha\}_{|\alpha|\leq k }$, $k\geq 0$, we have
 $\dim \cM_k=1+n+n^2+\cdots n^k$. Now, assume that $\{f_\alpha\}_{\alpha\in \FF_n^+}$ is not
 linearly independent in ${\bold S}[Z_1,\ldots,Z_n]$. Then there exists $m\geq 1$ such that
 $\{f_\alpha\}_{|\alpha|\leq m}$ is not
 linearly independent. This shows that the space
   $\cN_m:=\text{\rm span}\{f_\alpha\}_{|\alpha|\leq m }$ has
    $\dim \cN_m<1+n+n^2+\cdots n^m=\dim \cM_m.
   $
   On the other hand, note that for each $\alpha\in \FF_n^+$, $f_\alpha$ is a linear
   combination of $G_\beta$ with $\beta\in \FF_n^+$, $|\beta|\leq |\alpha|$, and each $G_\alpha$
is a linear
   combination of $f_\beta$ with $\beta\in \FF_n^+$, $|\beta|\leq |\alpha|$. Consequently,
    $\cN_m=\cM_m$ and, therefore,
$\dim \cN_m=\dim \cM_m$, which  is in contradiction with the strict
inequality above. The proof is complete.
\end{proof}

Now we prove an inverse mapping theorem for formal power series in noncommuting indeterminates.

\begin{theorem}\label{inv-series}
Let $F=(F_1,\ldots, F_n)$ be an $n$-tuple of formal power series in indeterminates
$Z_1,\ldots, Z_n$. Then the following statements are equivalent.
\begin{enumerate}
\item[(i)] There is an $n$-tuple of formal power series $G=(G_1,\ldots, G_n)$ such that
    $$G(0)=0\quad \text{ and }\quad F\circ G=id.
    $$
    \item[(ii)] $F(0)=0$ and the Jacobian $\det J_F(0)\neq 0$.
\end{enumerate}
In this case, $G$ is unique and $G\circ F=id$.
\end{theorem}

\begin{proof}
Assume that item (i) holds. For each $i=1,\ldots,n$, let
$$
F_i:=\sum_{k=0}^\infty \sum_{|\alpha|=k} a_\alpha^{(i)} Z_\alpha \quad \text{ and} \quad
G_i:=\sum_{k=1}^\infty \sum_{|\alpha|=k} b_\alpha^{(i)} Z_\alpha
$$
be such that $G_i(0)=0$ and $F\circ G=id$. Hence, we deduce that
\begin{equation*}
a_g^{(i)} +\sum_{j=1}^n a_{g_j}^{(i)} G_j+ \sum_{|\alpha|\geq
2}a_\alpha^{(i)} G_\alpha =Z_i,\qquad i=1,\ldots,n.
\end{equation*}
Since $G_i(0)=0$, if $|\alpha|\geq 2$, then each monomial in $G_\alpha$ has degree $\geq 2$. Consequently, we have $a_g^{(i)}=0$ for $i=1,\ldots,n$, i.e., $F(0)=0$, and
$\sum_{j=1}^n a_{g_j}^{(i)} b_{g_p}^{(j)}=\delta_{ip}$ for any $i,p\in \{1,\ldots, n\}$. The latter condition is equivalent to $J_F(0) J_G(0)=I_n$, which implies
$\det J_F(0)\neq 0$ and $\det J_G(0)\neq 0$. Therefore, item (ii) holds.

Now, we prove the implication $(ii)\implies (i)$. Assume that condition (ii) is satisfied  and let $F_i:=\sum_{k=1}^\infty \sum_{|\alpha|=k} a_\alpha^{(i)} Z_\alpha$.
We need to find and $n$-tuple  $G=(G_1,\ldots, G_n)$ with $G_i:=\sum_{k=1}^\infty \sum_{|\alpha|=k} b_\alpha^{(i)} Z_\alpha$
such that
    $G(0)=0$ and  $F\circ G=id$.  Therefore, we should have
\begin{equation}
\label{Zi2}
 \sum_{|\alpha|\geq 1}a_\alpha^{(i)} G_\alpha =Z_i,\qquad i=1,\ldots,n.
\end{equation}
We denote by $\text{\rm Coef}_{Z_\alpha}(H)$ the coefficient of
 the monomial $Z_\alpha$, $\alpha\in \FF_n^+$, in the formal power series $H$. Due to
 relation \eqref{Zi2}, we have
$$
\delta_{ip}=\text{\rm Coef}_{Z_p}(Z_i)=\sum_{j=1}a_{g_j}^{(i)} \text{\rm Coef}_{Z_p}(G_j)=\sum_{j=1}a_{g_j}^{(i)} b_{g_p}^{(j)}
$$
for any $i,p\in \{1,\ldots, n\}$. Hence, we deduce that
$J_F(0) J_G(0)=I_n$, where $J_F(0)=[a_{g_j}^{(i)}]_{i,j=1,\ldots,n}$  and
$J_G(0)=[b_{g_j}^{(i)}]_{i,j=1,\ldots,n}$. This implies that $J_G(0)$ is the inverse of $J_F(0)$ and, therefore, the coefficients $\{b_\alpha\}_{|\alpha|=1, i=1,\ldots,n}$ are uniquely determined and $\det J_G(0)\neq 0$.
Now, we prove by induction over $m$ that the coefficients $\{b^{(i)}_\alpha\}_{|\alpha|\leq m, i=1,\ldots,n}$ are uniquely determined by condition \eqref{Zi2}. Assume that the coefficients $\{b^{(i)}_\alpha\}_{|\alpha|\leq m-1, i=1,\ldots,n}$, $m\geq 2$,  are uniquely determined by \eqref{Zi2}.
Let $\beta=g_{p_1}\cdots g_{p_m}\in \FF_n^+$ with $p_1,\ldots, p_m\in \{1,\ldots, n\}$ and $m\geq 2$.  Note that  condition \eqref{Zi2} implies
\begin{equation*}
\begin{split}
\text{\rm Coef}_{Z_\beta}\left( \sum_{|\alpha|\geq 1}a_\alpha^{(i)} G_\alpha \right)
&=\text{\rm Coef}_{Z_\beta} \left(\sum_{j_1=1}^n a^{(i)}_{g_{j_1}} G_{j_1}+
\sum_{j_1,j_2=1}^n a^{(i)}_{g_{j_1} g_{j_2}} G_{j_1}G_{j_2}+\cdots +
\sum_{j_1,\ldots, j_m=1}^n a^{(i)}_{g_{j_1}\cdots g_{j_m}} G_{j_1}\cdots G_{j_m}
\right)\\
&=
\sum_{j_1=1}^n a^{(i)}_{g_{j_1}} b_\beta^{(j_1)}
+
\sum_{j_1,j_2=1}^n a^{(i)}_{g_{j_1} g_{j_2}}\left( \sum_{\sigma_1 \sigma_2=\beta, \sigma_1,\sigma_2\in \FF_n^+\backslash \{g_0\}} b_{\sigma_1}^{(j_1)} b_{\sigma_2}^{(j_2)}\right)
\\
&\qquad +\cdots +
\sum_{j_1,\ldots, j_m=1}^n a^{(i)}_{g_{j_1}\cdots g_{j_m}}b_{g_{p_1}}^{(j_1)}\cdots b_{g_{p_m}}^{(j_1)}=0
\end{split}
\end{equation*}
for each $i=1,\ldots,n$. We consider the matrices
$$
J_F(0)=\left[a^{(i)}_{g_{j_1}}\right]_{i,j_1=1,\ldots,n},\quad  \quad
 A_{n\times n^k}:=\left[a^{(i)}_{g_{j_1} g_{j_2}\cdots g_{j_k}}\right]_{i,j_1,\ldots,j_k=1,\ldots,n}
$$
for $ 2\leq k\leq m$,
and the column matrices
$$
B_{n\times 1}^{(\beta)}:=\left[b_\beta^{(i)}\right]_{i=1,\ldots,n},\quad \quad
B_{n^k\times 1}^{(\beta)} := \left[ \sum_{\sigma_1\cdots \sigma_k=\beta, \sigma_1,\ldots,\sigma_k\in \FF_n^+\backslash \{g_0\}} b_{\sigma_1}^{(j_1)} \cdots b_{\sigma_k}^{(j_k)}\right]_{j_1,\ldots,j_k=1,\ldots,n}
$$
for $2\leq k\leq m$.
The equality above is equivalent to
$$
J_F(0)B_{n\times 1}^{(\beta)}+A_{n\times n^2} B_{n^2\times 1}^{(\beta)} +\cdots + A_{n\times n^m} B_{n^m\times 1}^{(\beta)}=0_{n\times 1},
$$
where $0_{n\times 1}$ is the column zero matrix. Since the entries
of the matrices  $B_{n^2\times 1}^{(\beta)}, \ldots, B_{n^m\times
1}^{(\beta)}$ contain only  coefficients $b_\omega^{(j)}$, where
$|\omega|\leq m-1$ and $j=1,\ldots,n$, the relation
\begin{equation*}
B_{n\times 1}^{(\beta)}=-J_F(0)^{-1}A_{n\times n^2} B_{n^2\times 1}^{(\beta)} -\cdots - J_F(0)^{-1}A_{n\times n^m} B_{n^m\times 1}^{(\beta)}
\end{equation*}
shows that the coefficients $\{b_\beta^{(i)}\}_{|\beta|=m,
i=1,\ldots,n}$ are uniquely determined. This completes our proof by
induction. Therefore, the  item  (i) holds. Since $G(0)=0$ and $\det
J_G(0)\neq 0$, the result we proved above implies the existence of
an $n$-tuple of formal power series $H=(H_1,\ldots, H_n)$ such that
$H(0)=0$, $\det J_H(0)\neq 0$, and $G\circ H=id$. Hence, and using
item (i), we deduce that
$$
H=id\circ H=(F\circ G)\circ H=F\circ (G\circ H) =F\circ id =F
$$
and $G\circ F=id$. The uniqueness of $G$ is now obvious.  The proof is complete.
\end{proof}

The $n$-tuple $G=(G_1,\ldots, G_n)$ of Theorem \ref{inv-series} is called the inverse
of $F=(F_1,\ldots, F_n)$ with respect to the composition of power series.
We remark that under the conditions of Theorem \ref{inv-series}, the
composition  map $C_F:{\bold S}[Z_1,\ldots,Z_n]\to {\bold
S}[Z_1,\ldots,Z_n]$ defined by $C_F \Lambda:=\Lambda\circ F$ is an algebra
isomomorphism.

Let $H_n$ be an $n$-dimensional complex  Hilbert space with
orthonormal
      basis
      $e_1$, $e_2$, $\dots,e_n$, where $n\in\{1,2,\dots\}$.
       We consider the full Fock space  of $H_n$ defined by
      $$F^2(H_n):=\CC 1\oplus \bigoplus_{k\geq 1} H_n^{\otimes k},$$
      where   $H_n^{\otimes k}$ is the (Hilbert)
      tensor product of $k$ copies of $H_n$.
       We denote $e_\alpha:=
e_{i_1}\otimes\cdots \otimes  e_{i_k}$  if $\alpha=g_{i_1}\cdots
g_{i_k}$, where $i_1,\ldots, i_k\in \{1,\ldots,n\}$, and
$e_{g_0}:=1$. Note that $\{e_\alpha\}_{\alpha\in \FF_n^+}$ is an
orthonormal basis for $F^2(H_n)$. Define the left  (resp.~right)
creation
      operators  $S_i$ (resp.~$R_i$), $i=1,\ldots,n$, acting on $F^2(H_n)$  by
      setting
      $$
       S_i\varphi:=e_i\otimes\varphi, \qquad  \varphi\in F^2(H_n),
      $$
       (resp.~$
       R_i\varphi:=\varphi\otimes e_i$). Note that $S_iR_j=R_jS_i$ for $i,j\in \{1,\ldots, n\}$.
The noncommutative disc algebra $\cA_n$ (resp.~$\cR_n$) is the norm
closed algebra generated by the left (resp.~right) creation
operators and the identity. The   noncommutative analytic Toeplitz
algebra $F_n^\infty$ (resp.~$R_n^\infty$)
 is the the weakly
closed version of $\cA_n$ (resp.~$\cR_n$). These algebras were
introduced in \cite{Po-von} in connection with a noncommutative
version of the classical  von Neumann inequality (\cite{von}), and
have  been studied
    in several papers
\cite{Po-charact},  \cite{Po-multi},  \cite{Po-funct},
\cite{Po-analytic}, \cite{Po-disc}, \cite{Po-poisson},
 \cite{ArPo},
  \cite{DP1}, \cite{DP2},   \cite{DP},
  \cite{ArPo2},  and \cite{Po-curvature}.

Let $\Omega\subset B(\cH)^n$ be a set containing a ball
$[B(\cH)^n]_r$ for some $r>0$. We say that $f:\Omega\to B(\cH)$ is a
free holomorphic function  on $\Omega$ if there are some complex
numbers $a_\alpha$, $\alpha\in \FF_n^+$, such that
$$
f(X)=\sum_{k=0}^\infty \sum_{|\alpha|=k} a_\alpha X_\alpha, \qquad X=(X_1,\ldots, X_n)\in \Omega,
$$
where  the convergence is in the operator norm topology.  As in \cite{Po-holomorphic},
 one can show that  any free holomorphic function on $\Omega$ has a  unique representation as above.

If $f=(f_1,\ldots, f_n)$  is an    $n$-tuple  of
 formal power series, we define the radius of convergence  of $f$ by setting $r(f)=\min_{i=1,\ldots,n} r(f_i)$.
 According to \cite{Po-holomorphic}, $f_i$ is a free holomorphic function on the open ball
  $[B(\cH)^n]_{r(f)}$ for any $i=1,\ldots,n$.
The next result can be viewed as an inverse function theorem for free holomorphic functions.

\begin{theorem} \label{inv-holo} Let  $f=(f_1,\ldots, f_n)$  be  an   $n$-tuple  of
 formal power series with nonzero radius of convergence such that  $f(0)=0$ and $\det J_f(0)\neq 0$.
  Let  $g=(g_1,\ldots, g_n)$ be  the inverse power series  of $f$ with respect to the composition.

  If $g$ has a non-zero radius of convergence, then
there are open neighborhoods  $D$ and $G$ of $0$ in $B(\cH)^n$ such
that $f|_D:D\to G$  is a bijective free holomorphic function whose
inverse is a free holomorphic on $G$ which  coincides with
$g|_G:G\to D$.
\end{theorem}
\begin{proof} First, note that according to Theorem \ref{inv-series}, since $f(0)=0$ and $\det J_f(0)\neq 0$,  there is an $n$-tuple $g=(g_1,\ldots, g_n)$ of formal  power series   such that $g(0)=0$ and
$g\circ f=f\circ g=id$.
Assume that  $f$ and $g$
have  nonzero radius of convergence $r(f)>0$ and $r(g)>0$, respectively.
Fix $\epsilon_0>0$ such that
 $\epsilon_0<r(g)$.
Since $r(f)>0$  and $f(0)=0$, the Schwartz lemma for free holomorphic functions
 (see \cite{Po-holomorphic}) implies that there is $\delta_0\in (0, r(f))$ such that
$\|f(Y)\|<r(g)-\epsilon_0$ for any $Y\in [B(\cH)^n]_{\delta_0}^-$. On the other hand, using Theorem 1.2 from \cite{Po-automorphism}, the composition $Y\mapsto g(f(Y))$ is a free holomorphic function on $[B(\cH)^n]_{\delta_0}^-$. Due to the uniqueness theorem for free holomorphic functions  and the fact that $g\circ f=id$ as formal power series, we deduce that $g(f(Y))=Y$
for any $Y\in [B(\cH)^n]_{\delta_0}^-$. Hence, $f|_{[B(\cH)^n]_{\delta_0}^-}$ is a one-to-one free holomorphic function.

Now, fix $c_0\in (0, \delta_0)$.
 Since $r(g)>0$  and $g(0)=0$, using again the Schwartz lemma for free holomorphic functions, we find  $\gamma\in (0, r(g))$ such that
$\|g(X)\|<\delta_0-c_0$ for any $X\in [B(\cH)^n]_\gamma^-$. As above, the composition $X\mapsto f(g(X))$ is a free holomorphic function on $[B(\cH)^n]_\gamma^-$. Due to the uniqueness theorem for free holomorphic functions  and that $f\circ g=id$ as formal power series, we deduce that $f(g(X))=X$
for any $X\in [B(\cH)^n]_\gamma$. Consequently, $g|_{[B(\cH)^n]_{\gamma}}$ is a one-to-one free holomorphic function.

Set $G:=[B(\cH)^n]_\gamma$ and $D:=g([B(\cH)^n]_\gamma)$. Note that $g$ and $f$ are free holomorphic (and, therefore, continuous) on $[B(\cH)^n]_{r(g)}\supset G$ and $[B(\cH)^n]_{\delta_0}\supset [B(\cH)^n]_{\delta_0-c_0}\supset D$, respectively.
Due to the fact that $f|_{[B(\cH)^n]_{\delta_0}}:[B(\cH)^n]_{\delta_0}\to B(\cH)^n$ is a one-to-one continuous  function and $f(g(X))=X$
for any $X\in [B(\cH)^n]_\gamma$,  we deduce that  the pre-image
$\left((f|_{[B(\cH)^n]_{\delta_0}}\right)^{-1}\left([B(\cH)^n]_\gamma\right)$ is an open set in $[B(\cH)^n]_{\delta_0}$ which coincides with
$$[B(\cH)^n]_{\delta_0}\cap g([B(\cH)^n]_\gamma)=[B(\cH)^n]_{\delta_0-c_0}\cap g([B(\cH)^n]_\gamma)=g([B(\cH)^n]_\gamma=D.$$
Consequently,  since $D\subset [B(\cH)^n]_{\delta_0}$ is an open set in $[B(\cH)^n]_{\delta_0}$, we deduce that   $D$ is an open set in $B(\cH)^n$.
The proof is complete.
\end{proof}
In Theorem \ref{inv-holo},  we conjecture  that the condition that $g$ has a non-zero radius of convergence is a consequence of the fact that   $f=(f_1,\ldots, f_n)$ has  nonzero radius of convergence such that  $f(0)=0$ and $\det J_f(0)\neq 0$. However, this remains an open problem.

We  also remark that there is a converse for Theorem \ref{inv-holo}.
Let $D,G$ be open   neighborhoods of $0$ in $B(\cH)^n$ and let $\varphi:D\to G$ and $\psi:G\to D$ be free holomorphic functions such that $\psi=(\psi_1,\ldots, \psi_n)$ is the inverse of $\varphi=(\varphi_1,\ldots, \varphi_n)$. Then the associated formal power series are inverses of each other with respect to the composition.
Indeed, assume that $\varphi_i$ has the representation $\sum_{k=1}^\infty \sum_{|\alpha|=k} a^{(i)}_\alpha X_\alpha$ on $D$, and $\psi_i$ has the representation $\sum_{k=1}^\infty \sum_{|\alpha|=k} c^{(i)}_\alpha X_\alpha$ on $G$.
Then, we ca find $0<r<1$ such that  $[B(\cH)^n]_r^-\subset G$ and
$\varphi(\psi(X))=X$ for any $X\in [B(\cH)^n]_r^-$, where the convergence of the series defining $\psi(X)$ and $\varphi(\psi(X))$ are in the operator norm topology.  Hence, we deduce that
$\varphi(\psi(rS_1,\ldots, rS_n))=(rS_1,\ldots, rS_n)$. Since $\varphi_i(\psi(rS_1,\ldots, rS_n)) $ is in the noncommutative disc algebra $\cA_n$, it has a unique Fourier representation
$\sum_{k=1}^\infty \sum_{|\alpha|=k} c_\alpha^{(i)} r^{|\alpha|} S_\alpha$, where the coefficients $c_\alpha^{(i)}$ are exactly those of the formal power series $\varphi_i\circ \psi$. The equality above shows that $c_\alpha^{(i)}=0$ if $|\alpha|\geq 2$ and  $c_{g_j}^{(i)}=\delta_{ij}$. Therefore, $\varphi\circ \psi=id$. Due to Theorem \ref{inv-series}, we also deduce that $\psi\circ \varphi=id$, which proves our assertion.

\section{Polynomial automorphisms of $B(\cH)^n$}

In this section   we introduce the set of $n$-tuples of polynomials
  with property $(\cA)$, which is
one of the three classes of $n$-tuples $f=(f_1,\ldots, f_n)$ of
formal power series for which an operator model theory and dilation
theory  for the domain $\BB_f(\cH)$ will be developed in the coming
sections.

Let $\CC[Z_1,\ldots, Z_n]$ be the algebra of noncommutative
polynomials over $\CC$ (complex numbers) and noncommuting
indeterminates $Z_1,\ldots, Z_n$. We  say that an $n$-tuple
$p=(p_1,\ldots, p_n)$ of
   polynomials
 is invertible in $\CC[Z_1,\ldots, Z_n]^n$  with respect to the
composition if
  there exists an
 $n$-tuple $q=(q_1,\ldots, q_n)$ of
  polynomials such that
 $p\circ q=q\circ p=id.
 $
We remark that such an $n$-tuple of polynomials induces a free
holomorphic automorphism of $B(\cH)^n$, i.e., the map $\Phi_p:
B(\cH)^n\to B(\cH)^n$ defined by
$$\Phi_p(X):=(p_1(X),\ldots, p_n(X)),\qquad X=(X_1,\ldots, X_n)\in B(\cH)^n.
$$
We say that  $\Phi_p$  is a polynomial automorphism  of $B(\cH)^n$
and write $\Phi_p\in Aut(B(\cH)^n)$. Note that if $p,p'$ are
$n$-tuples of polynomials and  $\Phi_p, \Phi_{p'}$ are in  $
Aut(B(\cH)^n)$, then so is $\Phi_{p\circ p'}$ and $\Phi_{p\circ
p'}=\Phi_p \Phi_{p'}$.

\begin{theorem} \label{poly-equi} If $p=(p_1,\ldots, p_n)$ is  an $n$-tuple of
 noncommutative polynomials in $Z_1,\ldots, Z_n$, then the following statements are equivalent.
 \begin{enumerate}
\item[(i)]  $p$ is invertible in $\CC[Z_1,\ldots, Z_n]^n$  with respect to the
composition.
\item[(ii)]  There exists an
 $n$-tuple $q=(q_1,\ldots, q_n)$ of
 noncommutative polynomials in $Z_1,\ldots, Z_n$ such that
 $ q\circ p=id$.
 \item[(iii)]   $Z_1,\ldots, Z_n$  are contained in the linear span of
$ \{p_\alpha\}_{\alpha\in \FF_n^+}$ (where $p_0:=I$).
 \item[(iv)] The set
 $\{p_\alpha\}_{\alpha\in \FF_n^+}$  is  a linear  basis in
 $\CC[Z_1,\ldots,Z_n]$.

 \end{enumerate}
 \end{theorem}
\begin{proof} First we consider the case when $p_i(0)=0$,
$i=1,\ldots,n$. The implications $(i)\implies (ii)$ and
$(ii)\implies (iii)$ are obvious. To prove that  $(iii)\implies
(iv)$, assume that condition (iii) holds.  Since $Z_1,\ldots, Z_n$
are contained in the linear span of $ \{p_\alpha\}_{\alpha\in
\FF_n^+}$, there are some complex numbers
$\{a_\alpha^{(i)}\}_{\alpha\in \FF_n^+, |\alpha|\leq m}$ such that
$Z_i=\sum_{|\alpha|\leq m}a_\alpha^{(i)} p_\alpha(Z)$,
$i=1,\ldots,n$. Setting $q=(q_1,\ldots, q_n)$ with
$q_i(Z):=\sum_{|\alpha|\leq m}a_\alpha^{(i)}Z_\alpha$, we have
$q(0)=0$ and $q\circ p=id$.     Due to Lemma \ref{Jacobian}, we
obtain
$
\det J_p(0) \det J_q(0)=1,
$
which implies  $\det J_p(0)\neq 0$. Using now Theorem
\ref{Schroder}, we deduce that  the set
 $\{p_\alpha\}_{\alpha\in \FF_n^+}$  is  a linearly    independent  in
 $\CC[Z_1,\ldots,Z_n]$. On the other hand, condition (iii) also
 implies that $\CC[Z_1,\ldots,Z_n]$ is spanned by $\{p_\alpha\}_{\alpha\in
 \FF_n^+}$. Therefore, condition (iv) holds.

 Since $(iv)\implies (iii)$ is obvious, it remains to prove that
 $(iii)\implies (i)$. As above, if (iii) holds, then there is an
 $n$-tuple $q=(q_1,\ldots, q_n)$ of polynomials with $q_i(0)=0$ such
 that $q\circ p=id$ and  the set
 $\{p_\alpha\}_{\alpha\in \FF_n^+}$  is  a linearly    independent  in
 $\CC[Z_1,\ldots,Z_n]$.  The latter property shows that $p$ is not a
 right zero divisor with respect to the composition of polynomials, that
 is, if $\psi\in \CC[Z_1,\ldots,Z_n]^n$ and $\psi\circ p=0$, then
 $\psi=0$. Due to relation $q\circ p=id$, we obtain $(p\circ
 q-id)\circ p=0$. Since $p$ is not a
 right zero divisor, we deduce that $p\circ q=id$, which completes the
 proof.

Now, we consider the  case when $p(0)\neq 0$.   Assume that (iii)
holds. Then $p_i':=p_i-p_i(0)I $, $i=1,\ldots,n$, has the property
that $p_i'(0)=0$ and  $Z_1,\ldots, Z_n$  are contained in the linear
span of $ \{p_\alpha'\}_{\alpha\in \FF_n^+}$. Applying  the first
part of the proof to $p':=(p_1',\ldots, p_n')$, we deduce that  the
set
 $\{p'_\alpha\}_{\alpha\in \FF_n^+}$ is  a linear basis for  $\CC[Z_1,\ldots,Z_n]$.

 Consequently,
 setting $\cM_k:=\text{\rm span}\{p_\alpha'\}_{|\alpha|\leq k }$, $k\geq 0$, we have
 $\dim \cM_k=1+n+n^2+\cdots n^k$. Now, assume that $\{p_\alpha\}_{\alpha\in \FF_n^+}$ is not
 linearly independent in $\CC[Z_1,\ldots,Z_n]$. Then there exists $m\geq 1$ such that
 $\{p_\alpha\}_{|\alpha|\leq m}$ is not
 linearly independent. This shows that the space
   $\cN_m:=\text{\rm span}\{p_\alpha\}_{|\alpha|\leq m }$ has
    $\dim \cN_m$ strictly less than $\dim \cM_m=1+n+n^2+\cdots n^m.
   $
   On the other hand, note that for each $\alpha\in \FF_n^+$, $p_\alpha$ is a linear
   combination of $p'_\beta$ with $\beta\in \FF_n^+$, $|\beta|\leq |\alpha|$, and each $p'_\alpha$
is a linear
   combination of $p_\beta$ with $\beta\in \FF_n^+$, $|\beta|\leq |\alpha|$. Consequently,
    $\cN_m=\cM_m$ and, therefore,
$\dim \cN_m=\dim \cM_m$, which  is in contradiction with the strict
inequality above.  Therefore, the set
 $\{p_\alpha\}_{\alpha\in \FF_n^+}$  is  a linearly    independent  in
 $\CC[Z_1,\ldots,Z_n]$.  Since $\CC[Z_1,\ldots,Z_n]$ is spanned by $\{p_\alpha\}_{\alpha\in
 \FF_n^+}$, we deduce that
 $\{p_\alpha\}_{\alpha\in \FF_n^+}$  is  a linear  basis in
 $\CC[Z_1,\ldots,Z_n]$, which shows that  condition (iv) holds.
 Moreover, it shows
 that $p$ is not a
 right zero divisor with respect to the composition of polynomials.
   Since $Z_1,\ldots, Z_n$  are contained in the linear span of
$ \{p_\alpha\}_{\alpha\in \FF_n^+}$, we find $q\in \CC[Z_1,\ldots,
Z_n]$ such that
 $q\circ p=id$.  Hence,  we obtain $(p\circ
 q-id)\circ p=0$. Since $p$ is not a
 right zero divisor, we deduce that $p\circ q=id$, which implies (i).
The proof is complete.
\end{proof}

We say that $p=(p_1,\ldots, p_n)$ has property $(\cA)$ if any of the equivalences of Theorem \ref{poly-equi} holds.

\begin{example}\label{Ex2} If
\begin{equation*}
\begin{split}
p_1&= a_0I+a_1Z_1+a_2 Z_2 + a_3Z_3Z_2\\
p_2&=  b_0I+b_1Z_2+ b_2Z_3 + b_3Z_3^2 \\
p_3&= c_0I + c_1Z_3
\end{split}
\end{equation*}
are polynomials with complex coefficients such that $a_1 b_1c_1\neq
0$ then $p=(p_1,p_2, p_3)$ has property  $(\cA)$.
\end{example}

In what  follows we present a large class  of  polynomial
automorphisms of $B(\cH)^n$.
\begin{proposition}\label{Ex} Let $p=(p_1,\ldots, p_n)$ be an $n$-tuple
of noncommutative polynomials in $Z_1,\ldots, Z_n$ of the form
\begin{equation*}
\begin{split}
[p_1,\ldots, p_n]&=[a_1I,\ldots, a_n I]+[Z_1,\ldots, Z_n]A\\
&\qquad \quad +[q_1(Z_2,\ldots,Z_n), q_2(Z_3,\ldots, Z_n),\ldots,
q_{n-1}(Z_n), 0]A,
\end{split}
\end{equation*}
where $a_i\in \CC$, $A\in M_{n\times n}$ is an invertible scalar
matrix, and $q_1,\ldots, q_{n-1}$ are arbitrary noncommutative
polynomials in the specified indeterminates.
 Then   $p$  has property $(\cA)$.
\end{proposition}
\begin{proof}
 According to Theorem \ref{poly-equi}, it is enough to show that  $Z_1,\ldots, Z_n$  are contained in the linear span of
$ \{p_\alpha\}_{\alpha\in \FF_n^+}$. To solve the formal  system, multiply (to the right)
both sides  of the equality by $A^{-1}$ and solve for the indeterminates $Z_n,Z_{n-1},\ldots, Z_1$ in this order.
\end{proof}

   As we saw in the proof of Theorem \ref{poly-equi}, if  $p=(p_1,\ldots, p_n)$ is an $n$-tuple
of noncommutative polynomials  with property  $(\cA)$,
then the Jacobian matrix
$$
J_p(0):=\left[\left.\frac {\partial p_i}{\partial
Z_{j}}\right|_{Z=0}\right]_{1\leq i,j\leq n}
  $$
is invertible.  Moreover,   for the class of noncommutative
polynomials considered in Proposition \ref{Ex}, we have that $
 J_p(X)$ is an invertible operator  for any $X\in B(\cH)^n$.   This leads
to the following question. Is the Jacobian conjecture true in our
noncommutative setting? In other words, assuming that
$p=(p_1,\ldots, p_n)$ is an $n$-tuple of noncommutative polynomials
such that  the Jacobian matrix $
 J_p(X)$ is invertible for any $X\in
B(\cH)^n$ (or only for $X=0$), does this imply that $\Phi_p$ is a
polynomial automorphism of $B(\cH)^n$ ? Of course, this is true if  each polynomial $p_i$ has  degree  $1$.

Let  $p=(p_1,\ldots, p_n)$  be  be an $n$-tuple of noncommutative
polynomials in $Z_1,\ldots, Z_n$ with property $(\cA)$. We introduce
an inner product on $\CC[Z_1,\ldots,Z_n]$ by setting $\left< p_\alpha, p_\beta\right>:=\delta_{\alpha \beta}$, $\alpha,\beta\in \FF_n^+$.
Let $\HH^2(p)$ be the completion  of the linear space $\bigvee
\{p_\alpha\}_{\alpha\in \FF_n^+}$ with respect to this inner
product.
  It is easy to see that, due to Theorem \ref{poly-equi},
the noncommutative polynomials $\CC[Z_1,\ldots, Z_n]$ are dense in
$\HH^2(p)$. We define the  noncommutative domain
$$
\BB_p(\cH):=\{(X_1,\ldots, X_n)\in B(\cH)^n: \  \sum_{j=1}^n
p_j(X)p_j(X)^*\leq I\},
$$
which will be studied in the next sections.

\bigskip

\section{ Hilbert spaces of noncommutative formal power series}

In this section,   we introduce the class of $n$-tuples
$f=(f_1,\ldots, f_n)$ of
  of
formal power series (resp. free holomorphic functions) with property
$(\cS)$ (resp. $(\cF)$) and the  Hilbert space $\HH^2(f)$.
The associated domain $\BB_f(\cH)$ has a universal model
$(M_{Z_1},\ldots, M_{Z_n})$ of multiplication operators acting on
$\HH^2(f)$, which plays a crucial role   in the dilation theory on
the noncommutative domain $\BB_f(\cH)$.

We recall (see \cite{Po-holomorphic}) that  the algebra $H_{{\bf ball}}$~ of free holomorphic
functions on the open operatorial  $n$-ball of radius one is defined
 as the set of all power series $f=\sum_{\alpha\in
\FF_n^+}a_\alpha Z_\alpha$ with radius of convergence $\geq 1$,
i.e.,
 $\{a_\alpha\}_{\alpha\in \FF_n^+}$ are complex numbers  with
$\limsup_{k\to\infty} \left(\sum_{|\alpha|=k}
|a_\alpha|^2\right)^{1/2k}\leq 1.$
In this case, the mapping
$$
[B(\cH)^n]_1\ni (X_1,\ldots, X_n)\mapsto f(X_1,\ldots,
X_n):=\sum_{k=0}^\infty \sum_{|\alpha|=k}
 a_\alpha X_\alpha\in
B(\cH)
$$
is well defined, where  the convergence is in the operator norm topology. Moreover, the series converges absolutely, i.e., $\sum_{k=0}^\infty \left\|\sum_{|\alpha|=k} a_\alpha X_\alpha\right\|<\infty$, and uniformly  on any ball $[B(\cH)^n]_\gamma$ with $0\leq \gamma<1$.

Another case when  the  evaluation of $f$  can be defined is the following.
Assume that there exists an $n$-tuple $\rho=(\rho_1,\ldots, \rho_n)$ of strictly positive numbers such that
$$\limsup_{k\to\infty} \left(\sum_{|\alpha|=k}
|a_\alpha|\rho_\alpha\right)^{1/k}\leq 1.$$ Then the series
$f(X_1,\ldots, X_n):=\sum_{k=0}^\infty \sum_{|\alpha|=k}a_\alpha
X_\alpha$ converges absolutely and uniformly on any noncommutative
polydisc

 $$P({\bf r}):=\left\{(X_1,\ldots, X_n)\in B(\cH)^n:\ \|X_j\|\leq r_j, j=1,\ldots,n\right\}$$
of multiradius ${\bf{r}}=(r_1,\ldots, r_n)$ with $r_j<\rho_j$, $j=1,\ldots, n$.

We should also remark that, when $(X_1,\ldots, X_n)\in B(\cH)^n$ is a nilpotent $n$-tuple of operators, i.e.,
there is $m\geq 1$ such that $X_\alpha=0$ for all $\alpha\in \FF_n^+$ with $|\alpha|=m$, then  $f(X_1,\ldots,X_n)$ makes sense  since the series defining it  has only finitely many nonzero terms.

We need a few more definitions. Let $g=\sum_{k=0}^\infty
\sum_{|\alpha|=k} a_\alpha Z_\alpha$ be a formal power series in
indeterminates  $Z_1,\ldots, Z_n$.
 We denote by $\cC_g(\cH)$ (resp. $\cC_g^a(\cH)$, $\cC_g^{SOT}(\cH)$) the set of all
 $Y:=(Y_1,\ldots, Y_n)\in B(\cH)^n$ such that the series
$$
g(Y_1,\ldots, Y_n):=\sum_{k=0}^\infty \sum_{|\alpha|=k} a_\alpha Y_\alpha
$$
 is norm (resp. absolutely, SOT) convergent. These sets are called  sets of norm (resp. absolutely, SOT) convergence for the power series $g$.
We also introduce the set $\cC_g^{rad}(\cH)$  of all $Y:=(Y_1,\ldots, Y_n)\in B(\cH)^n$ such that there exists $\delta\in (0,1)$ with the property that
$rY\in \cC_g(\cH)$ for any $r\in (\delta, 1)$ and
$$
\widehat{g}(Y_1,\ldots, Y_n):=\text{\rm SOT-}\lim_{r\to 1}\sum_{k=0}^\infty \sum_{|\alpha|=k} a_\alpha r^{|\alpha|}Y_\alpha
$$
  exists.
Note that
$$ \cC_g^a(\cH)\subseteq \cC_g(\cH)\subseteq \cC_g^{SOT}(\cH)\quad \text{ and  } \quad
 \cC_g^{rad}(\cH)\subseteq \overline{\cC_g(\cH)}^{SOT}.
 $$

Now, consider    an $n$-tuple of
 formal power series $f=(f_1,\ldots, f_n)$ in indeterminates  $Z_1,\ldots, Z_n$  with the property that  the Jacobian
$$
\det J_f(0):=\det \left[\left. \frac{\partial f_i} {\partial
Z_j}\right|_{Z=0}\right]_{i,j=1}^n\neq 0.
$$  Due to Theorem \ref{Schroder},
    the set
 $\{f_\alpha\}_{\alpha\in \FF_n^+}$ (where $f_0:=I$)  is  linearly independent
in ${\bold S}[Z_1,\ldots,Z_n]$.   We  introduce an  inner product on
the linear span of $\{f_\alpha\}_{\alpha\in \FF_n^+}$ by setting
$
\left< f_\alpha, f_\beta\right>:=\delta_{\alpha \beta},
$
$\alpha,\beta\in \FF_n^+$.
 Let $\HH^2(f)$ be the
completion of the linear space  $\bigvee \{f_\alpha\}_{\alpha\in
\FF_n^+}$ with respect to this inner product. Assume now that
$f(0)=0$. Theorem \ref{inv-series} shows that $f$ is not  a {\it right
zero divisor}  with respect to the composition of power series,
i.e., there is no  non-zero power series $G$ in
$\mathbf{S}[Z_1,\ldots, Z_n]$ such that
 $G\circ f = 0$. Consequently,
the elements of  $\HH^2(f)$  can be seen as formal power series in
${\bold S}[Z_1,\ldots,Z_n]$  of  the form $\sum_{\alpha\in \FF_n^+}
a_\alpha f_\alpha$, where $\sum_{\alpha\in \FF_n^+}
|a_\alpha|^2<\infty$.

    Let $f=(f_1,\ldots, f_n)$  be  an $n$-tuple of
 formal power series  in $Z_1,\ldots, Z_n$ such that $f(0)=0$. We  say that $f$ has  property $(\mathcal{S})$
 if the following conditions
  hold.
\begin{enumerate}
\item[($\mathcal{S}_1$)] The $n$-tuple $f$ has nonzero radius of convergence  and
$ \det  J_f(0)\neq 0. $
\item[($\mathcal{S}_2$)]  The indeterminates   $Z_1,\ldots, Z_n$\ are in the Hilbert space
$\HH^2(f)$
 and  each left multiplication operator
  $M_{Z_i}:\HH^2(f)\to \HH^2(f)$  defined by
   $$
   M_{Z_i}\psi:=Z_i\psi, \qquad \psi\in \HH^2(f),
  $$
is a bounded multiplier of $\HH^2(f)$.
\item[($\mathcal{S}_3$)]
The left multiplication operators $M_{f_j}:\HH^2(f)\to \HH^2(f)$,
$M_{f_j}\psi=f_j \psi$,  satisfy the equations
\begin{equation*}
  M_{f_j}=f_j(M_{Z_1},\ldots, M_{Z_n}),\quad j=1,\ldots,n,
\end{equation*}
where  $(M_{Z_1},\ldots, M_{Z_n})$ is either in the convergence set
$\cC_f^{SOT}(\HH^2(f))$ or $\cC_f^{rad}(\HH^2(f))$.
\end{enumerate}
We remark that if $f$ is an $n$-tuple of noncommutative polynomials, then the condition $(\cS_3)$ is automatically satisfied.
We should also mention that, in case $(M_{Z_1},\ldots, M_{Z_n})$ is in
the set $\cC_f^{rad}(\HH^2(f))$, then  the condition
$(\mathcal{S}_3)$ should be understood as
$$
M_{f_j}=\widehat{f}_j(M_{Z_1},\ldots, M_{Z_n}):=\text{\rm
SOT-}\lim_{r\to 1} f_j(rM_{Z_1},\ldots, rM_{Z_n}),\quad
j=1,\ldots,n.
$$

\begin{remark} If $p=(p_1,\ldots, p_n)$ is an $n$-tuple of noncommutative polynomials with property $(\cA)$, then it has property  $(\cS)$.

\end{remark}

 \begin{proposition}\label{dense} If  $f=(f_1,\ldots, f_n)$ is an $n$-tuple of
 formal power series  with $f(0)=0$ and  property $(\mathcal{S})$, then $\CC[Z_1,\ldots, Z_n]$
  is dense in the Hilbert space $\HH^2(f)$.

 \end{proposition}
 \begin{proof}
Since $Z_i\in \HH^2(f)$ and $M_{Z_i}$ are bounded  multipliers of $\HH^2(f)$,
 we deduce that $Z_\alpha\in \HH^2(f)$ for any $\alpha\in \FF_n^+$ and,
 therefore, $\CC[Z_1,\ldots, Z_n]\subset \HH^2(f)$. Let $f_j$, $j=1,\ldots,n$, have the representation
 $f_j(Z_1,\ldots, Z_n)=\sum_{k=0}^\infty\sum_{\alpha\in \FF_n^+, |\alpha|=k}
c_\alpha^{(j)} Z_\alpha$. First,  we  assume that $(M_{Z_1},\ldots, M_{Z_n})$ is  in  the set $\cC_f^{SOT}(\HH^2(f))$
and
 $$f_j(M_{Z_1},\ldots, M_{Z_n})=\sum_{k=0}^\infty\sum_{\alpha\in \FF_n^+, |\alpha|=k}
c_\alpha^{(j)} M_{Z_\alpha},$$
where the convergence of the series is in the strong operator topology.
Then, for any $\epsilon>0$ and any polynomial $\psi\in \CC[Z_1,\ldots, Z_n]$, there exists $N_j\geq 1$
 such that
\begin{equation}
\label{approx1}
\left\|f_j(M_{Z_1},\ldots, M_{Z_n})\psi-\sum_{k=0}^{N_j}\sum_{\alpha\in \FF_n^+, |\alpha|=k}
c_\alpha^{(j)} M_{Z_\alpha}\psi\right\|_{\HH^2(f)}<\epsilon,\qquad j=1,\ldots,n.
\end{equation}
Fix $i,j\in \{1,\ldots, n\}$. Due to  relation \eqref{approx1}, we can find  polynomials $p$ and $q$ such that
$$\|f_j(M_{Z_1},\ldots, M_{Z_n})1-p\|_{\HH^2(f)}<\frac{\epsilon}{2\|M_{f_j}\|}$$ and
$$\|f_j(M_{Z_1},\ldots, M_{Z_n})p-qp\|_{\HH^2(f)}<\frac{\epsilon}{2}.
$$
Hence, and using condition $(\mathcal{S}_2)$,  we deduce that
\begin{equation*}
\begin{split}
\|f_jf_i-qp\|_{\HH^2(f)}&\leq \|f_j(M_{Z_1},\ldots, M_{Z_n})f_i(M_{Z_1},\ldots, M_{Z_n})1
-f_j(M_{Z_1},\ldots, M_{Z_n})p\|\\
&\qquad +\|f_j(M_{Z_1},\ldots, M_{Z_n})p-qp\|\\
&\leq \|f_j(M_{Z_1},\ldots, M_{Z_n})\|\frac{\epsilon}{2\|M_{f_j}\|}+\frac{\epsilon}{2}=\epsilon.
\end{split}
\end{equation*}
An inductive argument shows that each power series $f_\alpha$, $\alpha\in \FF_n^+$, can be approximated
in $\HH^2(f)$ by polynomials in $\CC[Z_1,\ldots, Z_n]$.
Taking into account that $\text{\rm span} \{f_\alpha\}_{\alpha\in \FF_n^+}$ is dense in
  $\HH^2(f)$, we deduce that $\CC[Z_1,\ldots, Z_n]$ is dense in $\HH^2(f)$.

Now, we consider the case when  $(M_{Z_1},\ldots, M_{Z_n})$ is  in  the set $\cC_f^{rad}(\HH^2(f))$ and
\begin{equation*}
\widehat{f}_j(M_{Z_1},\ldots, M_{Z_n})=\text{\rm SOT-}\lim_{r\to 1}
\sum_{k=0}^\infty\sum_{\alpha\in \FF_n^+, |\alpha|=k}
c_\alpha^{(j)}r^{|\alpha|} M_{Z_\alpha},
\end{equation*}
where the convergence of the series is in the   operator norm topology for each $0\leq r<1$.
Hence, we deduce that,   for any $\epsilon>0$ and any polynomial $\psi\in \CC[Z_1,\ldots, Z_n]$, there exists
$r_0\in (0,1)$ such that

\begin{equation*}
\left\|\widehat{f}_j(M_{Z_1},\ldots, M_{Z_n})\psi-\sum_{k=0}^\infty\sum_{\alpha\in \FF_n^+, |\alpha|=k}
c_\alpha^{(j)}r_0^{|\alpha|} M_{Z_\alpha}\psi\right\|_{\HH^2(f)}<\epsilon,\qquad j=1,\ldots,n.
\end{equation*}
 Using the convergence of the series in the operator norm topology, we find $N_j\geq 1$ such that
\begin{equation*}
\left\|\widehat{f}_j(M_{Z_1},\ldots, M_{Z_n})\psi-\sum_{k=0}^{N_j}\sum_{\alpha\in \FF_n^+, |\alpha|=k}
c_\alpha^{(j)}r_0^{|\alpha|} M_{Z_\alpha}\psi\right\|_{\HH^2(f)}<\epsilon,\qquad j=1,\ldots,n.
\end{equation*}
Now, one can proceed as  in the first part of the proof to show that $\CC[Z_1,\ldots, Z_n]$
  is dense in the Hilbert space $\HH^2(f)$.
The proof is complete.
 \end{proof}

According to \cite{Po-holomorphic} and \cite{Po-pluriharmonic}, the
noncommutative Hardy space
  $H_{\text{\bf ball}}^\infty (B(\cE,\cG))$   can be identified to the operator
  space
$  F_n^\infty\bar\otimes B(\cE,\cG)$ (the weakly closed operator
space generated by the spatial tensor product), where $F_n^\infty$
is the noncommutative analytic Toeplitz algebra. More precisely, a
bounded free holomorphic function $F$ is uniquely determined by its
{\it (model) boundary function} $\widetilde F\in
F_n^\infty\bar \otimes B(\cE, \cG)$ defined by
$\widetilde F:=\text{\rm SOT-}\lim_{r\to 1}
F(rS_1,\ldots, rS_n). $ Moreover, $F$ is  the noncommutative
Poisson transform  \cite{Po-poisson} of $\widetilde F$ at $X \in [B(\cH)^n]_1$, i.e.,
$
F(X )=(P_X\otimes I)[\widetilde F].
$
Similar results hold for bounded free holomorphic functions on the
noncommutative ball  $[B(\cH)^n]_\gamma$, $\gamma>0$.

The next result provides a characterization for  the $n$-tuples of
formal power series with property $(\cS)$.

\begin{lemma} \label{M2} Let  $f=(f_1,\ldots, f_n)$ be  an $n$-tuple  of
 formal power series with   $f(0)=0$.  Then $f$  has property $(\mathcal{S})$ if and only if
     the following conditions hold:
 \begin{enumerate}
 \item[(i)] the $n$-tuple  $f$ has nonzero radius of convergence and $\det J_f(0)\neq
 0$;
 \item[(ii)]  the inverse  of $f$, say $g=(g_1,\ldots,
 g_n)$,
 is a  bounded free holomorphic function  on
 $[B(\cH)^n]_1$;
 \item[(iii)]  the model boundary  function $\widetilde{g}=(\widetilde{g}_1,\ldots, \widetilde{g}_n)$ satisfies either one of the following conditions:
     \begin{enumerate}
     \item[(a)] $\widetilde{g}$ is  in $\cC_f^{SOT}(\HH^2(f))$  and
 $
 S_i=f_i(\widetilde{g}_1,\ldots, \widetilde{g}_n)$, $ i=1,\ldots,n;
 $
 \item[(b)] $\widetilde{g}$ is  in $\cC_f^{rad}(\HH^2(f))$  and
 $
 S_i= \text{\rm SOT-}\lim_{r\to 1} f_j(r\widetilde{g}_1,\ldots, r\widetilde{g}_n)$, $i=1,\ldots,n,
 $
 where $(S_1,\ldots, S_n)$ are the left creation operators on the full Fock space $F^2(H_n)$.
 \end{enumerate}
 If  $f$ is an $n$-tuple of noncommutative polynomials, then the condition $(iii)$ is automatically satisfied.
\end{enumerate}
\end{lemma}
\begin{proof} Since the condition $(\cS_1)$ coincides with $(i)$,
  we show  that the condition
   $(\mathcal{S}_2)$ holds if and
 only if $f$ satisfies the condition (ii). To prove the direct implication note that,
  due to Theorem \ref {inv-series},  the
composition  map $C_f:{\bold S}[Z_1,\ldots,Z_n]\to {\bold
S}[Z_1,\ldots,Z_n]$ defined by $C_f \psi:=\psi\circ f$ is an
isomorphism. Therefore, there is an $n$-tuple $g=(g_1,\ldots, g_n)$
of power series such that $f\circ g=g\circ f=id$.
 On the other hand, condition
$(\cS_2)$ implies the existence of an $n$-tuple
$\chi=(\chi_1,\ldots, \chi_n)$ of formal power series with
$\chi(0)=0$ and $\chi_i\in H^2_{\bf ball}$, i.e.,
$\chi_i=\sum_{\alpha\in \FF_n^+} a_\alpha^{(i)} Z_\alpha$ for some
$a_\alpha^{(i)}\in \CC$
  with $\sum_{\alpha\in \FF_n^+} |a_\alpha^{(i)}|^2<\infty$,
 and such that $\chi\circ f=id$.
Consequently, $(f\circ \chi-id)\circ f=0$ and, using the injectivity
of $C_f$, we deduce that $f\circ \chi=id$. Since the inverse of $f$
is unique, we must have $g=\chi$.

   Due to condition $(\mathcal{S}_2)$, the left multiplication operator
  $M_{Z_i}:\HH^2(f)\to \HH^2(f)$  defined by
   $$
   M_{Z_i}\psi:=Z_i\psi, \qquad \psi\in \HH^2(f),
  $$
  is a bounded multiplier of $\HH^2(f)$.
Let $U:\HH^2(f)\to F^2(H_n)$ be the unitary operator defined by
$U(f_\alpha):=e_\alpha$, $\alpha\in \FF_n^+$. Note that
$Z_i=\sum_{\alpha\in \FF_n^+} a_\alpha^{(i)}
f_\alpha=U^{-1}(\varphi_i)$, where $\varphi:=\sum_{\alpha\in
\FF_n^+}a_\alpha^{(i)} e_\alpha\in F^2(H_n)$. One can easily see
  that $M_{Z_i}$ is a bounded multiplier of $\HH^2(f)$ if and only if $\varphi_i$ is a bounded
   multiplier of $F^2(H_n)$.
 Moreover, $M_{Z_i}=U^{-1} \varphi_i(S_1,\ldots,S_n) U$, where
   $\varphi_i(S_1,\ldots,S_n)$ is   in  the noncommutative Hardy algebra $F_n^\infty$ and has the
    Fourier
    representation  $\sum_{\alpha\in \FF_n^+} a_\alpha^{(i)} S_\alpha$.
     According to Theorem 3.1 from \cite{Po-holomorphic}, we deduce that
      $g_i=\sum_{\alpha\in \FF_n^+} a_\alpha^{(i)} Z_\alpha$ is a bounded free
      holomorphic function on the unit ball $[B(\cH)^n]_1$ and  has its model boundary
       function $\widetilde{g_i}=\varphi_i(S_1,\ldots, S_n)$.
Therefore, condition  $(\cS_2)$ is equivalent to item (ii).
Since each $g\in \HH^2(f)$ has a unique representation
$g=\sum_{\alpha\in \FF_n^+} a_\alpha f_\alpha$ with $\sum_{\alpha\in
\FF_n^+} |a_\alpha|^2<\infty$,
  the multiplication operator
$M_{f_j}:\HH^2(f)\to \HH^2(f)$ defined by
$$
M_{f_j}\left(\sum_{\alpha\in \FF_n^+} a_\alpha
f_\alpha\right)=\sum_{\alpha\in \FF_n^+}
 a_\alpha f_jf_\alpha
$$
satisfies the equation
\begin{equation}
\label{simi} M_{f_j}=U^{-1} S_j U,\qquad j=1,\ldots,n,
\end{equation}
where $S_1,\ldots, S_n$ are the left creation operators on
$F^2(H_n)$. Consequently, $M_{f_\alpha}=U^{-1} S_\alpha U$,
$\alpha\in \FF_n^+$. Since   $M_{Z_i}=U^{-1} \widetilde{g}_i U$,
where $\widetilde{g}_i$ is the model boundary function of $g_i\in
H^\infty_{\bf ball}$, it is easy to see that the equality $
M_{f_j}=f_j(M_{Z_1},\ldots, M_{Z_n})$, $ j=1,\ldots,n, $ of
$(\mathcal{S}_3)$ is equivalent to condition (iii). This completes
the proof.
\end{proof}

Let $g=(g_1,\ldots, g_n)$ be the $n$-tuple of power series, as in
Lemma \ref{M2}, having the representations
$$
g_i:=\sum_{k=0}^\infty\sum_{\alpha\in \FF_n^+, |\alpha|=k}
a_\alpha^{(i)}Z_\alpha,  \qquad i=1,\ldots,n,
$$
where the  sequence $\{a_\alpha^{(i)}\}_{\alpha\in \FF_n^+}$  is
uniquely determined   by the condition    $g\circ f=id$. We say that
an $n$-tuple of operators $X=(X_1,\ldots, X_n)\in B(\cH)^n$
satisfies the equation $g(f(X))=X$   if
 either
one of the following conditions hold:
\begin{enumerate}
     \item[(a)] $X\in \cC_f^{SOT}(\cH)$  and either $
X_i=\sum_{k=1}^\infty\sum_{\alpha\in \FF_n^+, |\alpha|=k}
a_\alpha^{(i)}[f(X)]_\alpha$,  $i=1,\ldots,n,$
  where the  convergence of the  series  is in the strong operator topology, or
  $$
 X_i= \text{\rm SOT-}\lim_{r\to 1} \sum_{k=1}^\infty\sum_{\alpha\in \FF_n^+, |\alpha|=k}
a_\alpha^{(i)}r^{|\alpha|}[{f}(X)]_\alpha ,\qquad i=1,\ldots,n;
 $$
 \item[(b)] $X\in \cC_f^{rad}(\cH)$  and either
 $
X_i=\sum_{k=1}^\infty\sum_{\alpha\in \FF_n^+, |\alpha|=k}
a_\alpha^{(i)}[\widehat{f}(X)]_\alpha$,  $i=1,\ldots,n,
$
  where the  convergence of the  series  is in the strong operator topology, or
 $$
 X_i= \text{\rm SOT-}\lim_{r\to 1} \sum_{k=1}^\infty\sum_{\alpha\in \FF_n^+, |\alpha|=k}
a_\alpha^{(i)}r^{|\alpha|}[\widehat{f}(X)]_\alpha ,\qquad
i=1,\ldots,n.
 $$
  \end{enumerate}
We consider  the noncommutative domains
$$
\BB_f(\cH):=\{X=(X_1,\ldots, X_n)\in B(\cH)^n:\   g(f(X))=X \text{ and }   \|f(X)\|\leq 1 \}
$$
and
$$
\BB_{f}^{<}(\cH):=\{X=(X_1,\ldots, X_n)\in B(\cH)^n:\   g(f(X))=X \text{ and }   \|f(X)\|< 1 \}.
$$
We say that $(T_1,\ldots, T_n)\in B(\cH)^n$ is  a pure $n$-tuple  of
operators in  $\BB_f(\cH)$ if
$$
\text{\rm SOT-}\lim_{k\to \infty} \sum_{\alpha\in \FF_n,\,
|\alpha|=k} [f(T)]_\alpha [f(T)]_\alpha^*=0.
$$
The set of all pure elements of $\BB_f(\cH)$ is denoted by
$\BB^{pure}_f(\cH)$. Note that
$$
\BB_{f}^{<}(\cH)\subseteq \BB^{pure}_f(\cH) \subseteq \BB_f(\cH).
$$
An $n$-tuple of operators $X:=(X_1,\ldots, X_n)\in B(\cH)^n$ is called nilpotent
if there is $m\geq 1$ such that $X_\alpha=0$ for any $\alpha\in \FF_n^+$ with $|\alpha|=m$.
 We denote by $\BB^{nil}_f(\cH)$ the set of all nilpotent $n$-tuples in
$\BB_f(\cH)$.

\begin{proposition} \label{single}
Let $g\in H^\infty(\DD)$ be such that $g(0)=0$ and $g'(0)\neq 0$,
and let $f$ be its inverse power series with respect to the
composition. If $S$ is the unilateral shift on the Hardy space
$H^2(\DD)$ and
$$
f(g(S))=S
$$
for an appropriate evaluation of $f$ at $g(S)$ (where  $g(S)$ is
defined using the Nagy-Foias functional calculus), then $f$ has the
property $(\mathcal{S})$.
\end{proposition}
\begin{proof}   According to \cite{Ca}, the power series associated with $g$ has an inverse
 $f$, with respect to the composition, with nonzero radius of convergence.  Using the fact that
 $S=f(g(S))$ and applying Lemma \ref{M2} when $n=1$, we deduce that  $f$ has the
property $(\mathcal{S})$.
\end{proof}

In what follows, we present several examples  of $n$-tuples of
formal power series with property $(\cS)$. First, we consider the
single variable case.

\begin{example} The power series defined by
$$
f=Z\left(I+\frac{1}{a} Z\right)^{-1},\qquad a>2,
$$
has property $(\mathcal{S})$  and
$$
[B(\cH)]_1^-\subsetneq [B(\cH)]_{\frac{a}{a-1}}\subset \BB_f(\cH).
$$
\end{example}
\begin{proof} A straightforward computation shows that the inverse
power series of $f$ is $g= Z\left(I-\frac{1}{a} Z\right)^{-1}$. The
corresponding function $z\mapsto g(z)$  is analytic and bounded on
$\DD$. Moreover,
$$
g(S)=S-\frac{1}{a}S^2+\frac{1}{a^2}S^3+\cdots
$$ is a bounded operator, where the convergence is in the operator
norm topology, and $\|g(S)\|<2$.  Taking into account that
$\left\|\frac{1}{a} g(S)\right\|<1$, we deduce that
\begin{equation*}
\begin{split}
f(g(S))&= g(S)\left( I+\frac{1}{a} g(S)\right)^{-1}\\
&=S\left(I-\frac{1}{a} S\right)^{-1}\left( I+\frac{1}{a}
S\left(I-\frac{1}{a} S\right)^{-1}\right)^{-1}=S.
\end{split}
\end{equation*}
 Therefore, $f$ has property $(\mathcal{S})$.
Consider the noncommutative domain
$$
\BB_f(\cH):=\{X\in B(\cH): \ X=g(f(X)) \text{ and }  \|f(X)\|\leq
1\}.
$$
Note that if $\|X\|<a$, then $f(X):=X\left(I+\frac{1}{a}
X\right)^{-1}$ is well-defined. If, in addition, $\|f(X)\|\leq 1$,
then one can easily see that
$$
g(f(X))=f(X)\left(I-\frac{1}{a} f(X)\right)^{-1}=X.
$$
Hence
$$
\{X\in B(\cH): \|X\|<a \text{ and } \|f(X)\|\leq 1\} \subset
\BB_f(\cH).
$$
Note also that if $\|X\|\leq \frac{a}{a-1}$, then
$$
\|f(X)\|\leq \|X\|\frac{1}{1+\frac{\|X\|}{a}}\leq 1.
$$
Since $a>2$, we have $1<\frac{a}{a-1}<a$ and
$$
[B(\cH)]_1^-\subsetneq [B(\cH)]_{\frac{a}{a-1}}\subset \BB_f(\cH).
$$
 This completes the proof.
\end{proof}

Now we consider some  tuples of noncommutative polynomials  with the
property $(\mathcal{S})$.

\begin{example}\label{Ex22} If
$$
\begin{cases}
p_1 &= Z_1-Z_2-\frac{1}{2} Z_1Z_2 \\
 p_2 &=  Z_2
 \end{cases}
\quad \text{ and } \quad
\begin{cases}
q_1 &= Z_1-\frac{1}{3} Z_1Z_2\\
 q_2 &=  Z_2-\frac{1}{2} Z_3Z_2\\
 q_3 &= Z_3,
 \end{cases}
 $$
then $p=(p_1,p_2)$ and $q=(q_1,q_2, q_3)$  have  property
$(\mathcal{S})$.
\end{example}

\begin{proof} Note that
$$
\begin{cases}
Z_1 &= (p_1+p_2)\left( I+\frac{p_2}{2}  + \left(\frac{p_2}{2}\right)^2+\cdots\right)\\
 Z_2 &=  p_2
 \end{cases}
 $$
 Setting $g_1:=(Z_1+ Z_2)\left( I+\frac{Z_2}{2}  + \left(\frac{Z_2}{2}\right)^2+\cdots\right)$ and
 $g_2=Z_2$, it is easy to see that $p\circ g=g\circ p=id$. On the other hand, $g=(g_1, g_2)$ is a bounded free holomorphic function on $[B(\cH)^2]_1$ and the model boundary function $\widetilde{g}=(\widetilde{g}_1, \widetilde{g}_2)$  is given by
$\widetilde{g}_1:=(S_1+ S_2)\left( I+\frac{1}{2} S_2 +
\left(\frac{1}{2} S_2\right)^2+\cdots\right)$ and
 $\widetilde{g}_2=S_2$. According to Lemma \ref{M2}, $p=(p_1,p_2)$   has  property
$(\mathcal{S})$. The second example can be treated similarly.
Setting $r=(r_1,r_2,r_3)$, where
$$
\begin{cases}
{r}_1 &=  Z_1\left[I-\frac{1}{3} Z_2\left( I-\frac{1}{2} Z_3\right)^{-1}\right]^{-1}\\
  {r}_2&=  Z_2\left(I-\frac{1}{2} Z_3\right)^{-1}\\
  {r}_3&= Z_3,
 \end{cases}
 $$
 one can check that  $r\circ q=q\circ r$ and the model boundary functions
 $\widetilde{r}_1 =  S_1\left[I-\frac{1}{3} S_2\left( I-\frac{1}{2} S_3\right)^{-1}\right]^{-1}$,
 $\widetilde{r}_2=  S_2\left( I-\frac{1}{2} S_3\right)^{-1}$ and
  $\widetilde{r}_3= S_3$ are in the noncommutative disc algebra $\cA_3$. Applying again Lemma
   \ref{M2},  we deduce that $q=(q_1,q_2)$   has  property
$(\mathcal{S})$.
\end{proof}

\begin{example}\label{Ex3} Let $\gamma>0$ and  $a \in \CC$ with   $|a|>1$ and let
\begin{equation*}
\begin{split}
f_1 &= \frac{1}{\gamma}Z_1-\frac{1}{\gamma}Z_2-\left(\frac{1}{\gamma}Z_2\right)^2-\cdots\\
f_2 &=  \frac{a}{\gamma}Z_2.
\end{split}
\end{equation*}
Then $f=(f_1,f_2)$ has property  $(\mathcal{S})$.
\end{example}
\begin{proof}
First note that $f=(f_1,f_2)$   satisfies condition $(\cS_1)$. Since
\begin{equation*}
\begin{split}
Z_1&=\gamma f_1 +\gamma\sum_{j=1}^\infty\left(\frac{1}{a}f_2 \right)^j\\
Z_2&=\frac{\gamma}{a} f_2
\end{split}
\end{equation*}
and $\sum_{j=1}^\infty\frac{1}{|a |^{2j}}<\infty$, we deduce that
$Z_1,\ldots, Z_n$\ are in $\HH^2(f)$. Let $U:\HH^2(f)\to F^2(H_2)$
be the unitary operator defined by $U(f_\alpha):=e_\alpha$,
$\alpha\in \FF_2^+$. Note that the multiplication operator
$M_{Z_1}\in B(\HH^2(f))$ is unitarily equivalent to the operator
$\varphi_1(S_1,S_2)\in B(F^2(H_2))$ defined by
$$
\varphi_1(S_1,S_2):=\gamma S_1+\gamma\sum_{j=1}^\infty\left(\frac{1}{a}S_2\right)^j,
$$
which is in the noncommutative disc algebra $\cA_2$. Similarly, the
operator $M_{Z_2}\in B(\HH^2(\psi))$ is unitarily equivalent to
$\varphi_2(S_1,S_2):=\frac{\gamma}{a} S_2\in \cA_2$. Therefore,
condition $(\mathcal{S}_2)$ is satisfied.
It remains to check  condition $(\mathcal{S}_3)$. Since $|a|>1$, we
have $\|M_{Z_2}\|<\gamma$ and, therefore,
$$
f(M_{Z_1}, M_{Z_2})= \frac{1}{\gamma} M_{Z_1}-\sum_{j=1}^\infty
\left(\frac{1}{\gamma} M_{Z_2}\right)^j,
$$
where the convergence is in the operator norm topology. On the other hand, since
 the operator $M_{f_1}\in B(\HH^2(f))$ is unitarily equivalent to the left creation operator $S_1$
 on $F^2(H_2)$, the condition
$M_{f_1}=\lim_{m\to \infty} \left[\frac{1}{\gamma}
M_{Z_1}-\sum_{j=1}^m
 \left(\frac{1}{\gamma} M_{Z_2}\right)^j\right]$ is equivalent to
 $$
 S_1=\lim_{m\to\infty}\left[S_1+\sum_{j=1}^\infty \left(\frac{1}{a}S_2\right)^j
 -\sum_{j=1}^m \left(\frac{1}{a}S_2\right)^j\right],
 $$
 which is obviously true. This completes the proof.
\end{proof}
 Similarly, one can treat the following
\begin{example} If
\begin{equation*}
\begin{split}
f_1&= Z_1-Z_2-Z_2Z_1 -Z_2^2-Z_2^3\cdots\\
f_2&=2Z_2,
\end{split}
\end{equation*}
then $f=(f_1,f_2)$ has property $(\mathcal{S})$.
\end{example}

\bigskip
\bigskip

{\bf Hilbert spaces of  of free holomorphic functions.}
  Let $\varphi=(\varphi_1,\ldots,
\varphi_n)$ be an $n$-tuple of free holomorphic functions on
$[B(\cH)^n]_\gamma$, $\gamma>0$,  with range in $[B(\cH)^n]_1$.
 We say that $\varphi$
is  not a {\it  right zero divisor}  with respect to the  composition
 with free holomorphic functions on $[B(\cH)^n]_1$
if   for any non-zero free holomorphic function $G$ on $[B(\cH)^n]_1$,  the composition
 $G\circ \varphi$ is not identically zero. We recall (see
 \cite{Po-automorphism}) that $G\circ \varphi$ is a free holomorphic
 function on $[B(\cH)^n]_\gamma$.
Consider the vector space of free holomorphic functions
  $$
  \HH^2(\varphi):=\{G\circ \varphi: \ G\in  H_{\bf ball}^2\},
  $$
  where the noncommutative Hardy space
$H^2_{\bf ball}$ is the Hilbert space of all free holomorphic
functions on $ [B(\cH)^n]_1 $
 of the
form
$$f(X_1,\ldots, X_n)=\sum_{k=0}^\infty \sum_{|\alpha|=k}
a_\alpha  X_\alpha, \qquad   \sum_{\alpha\in \FF_n^+}|a_\alpha|^2<\infty,
$$ with the inner product
$ \left< f,g\right>:=\sum_{k=0}^\infty \sum_{|\alpha|=k}a_\alpha
{\overline b}_\alpha,$ where  $g=\sum_{k=0}^\infty \sum_{|\alpha|=k}
b_\alpha  X_\alpha$ is
 another free holomorphic function  in $H^2_{\bf ball}$.
  Note that each element $\psi\in \HH^2(\varphi)$ is a free holomorphic function on $[B(\cH)^n]_\gamma$ which  has a unique
  representation of the form  $\psi=G\circ \varphi$ for some
   $G\in  H_{\bf ball}^2$.
We introduce an inner product on $\HH^2(\varphi)$ by setting
\begin{equation*}
\label{inner} \left<F\circ \varphi, G\circ \varphi\right>_{\HH^2(\varphi)}:=  \left<F, G\right>_{H_{\bf ball}^2}.
\end{equation*}
It is easy to see that $\HH^2(\varphi)$ is a Hilbert space with respect to this inner product.
 We make the
following assumptions:
\begin{enumerate}
\item[($\mathcal{F}_1$)]  the $n$-tuple  $\varphi=(\varphi_1,\ldots, \varphi_n)$   of  free holomorphic
functions on $[B(\cH)^n]_\gamma$  has  range in $[B(\cH)^n]_1$  and it is
 not a   right zero divisor with respect to the  composition
  with free holomorphic functions on $[B(\cH)^n]_1$.
\item[($\mathcal{F}_2$)] The coordinate functions $X_1,\ldots, X_n$ on $[B(\cH)^n]_\gamma$
 are contained in $\HH^2(\varphi)$ and
  the left multiplication
 by $X_i$
is a bounded multiplier of $\HH^2(\varphi)$, for each $i=1,\ldots,n$.
\item[($\mathcal{F}_3$)] For each $i=1,\ldots,n$,
the left multiplication operator $M_{\varphi_i}:\HH^2(\varphi)\to
\HH^2(\varphi)$ satisfies the equation
$$
M_{\varphi_i}=\varphi_i(M_{Z_1},\ldots, M_{Z_n}),
$$
 where  $(M_{Z_1},\ldots, M_{Z_n})$  is either in the convergence set
$\cC_\varphi^{SOT}(\HH^2(\varphi))$ or
$\cC_\varphi^{rad}(\HH^2(\varphi))$.
\end{enumerate}
If $\varphi$ is an $n$-tuple of noncommutative polynomials, then the
condition $(\cF_3)$ is automatically satisfied. Under the
above-mentioned  conditions, the free holomorphic function $\varphi$
is said to have property $(\mathcal{F})$. We remark that, unlike the power series with property $(\cS)$,  $\varphi(0)$
could be different from $0$.

Using Theorem 2.1 from \cite{Po-automorphism}, we can show that
$\varphi$  has  property $(\mathcal{F})$ if and only if there exists
$g=(g_1,\ldots, g_n)$ a bounded free holomorphic function on
$[B(\cH)^n]_1$ such that
\begin{equation} \label{gfi}
g(\varphi(X))=X,\qquad X\in [B(\cH)^n]_\gamma,
\end{equation}
where $\varphi(X)$ is in the set of norm-convergence of $g$, and the
model boundary   function $\widetilde{g}=(\widetilde{g}_1,\ldots,
\widetilde{g}_n)$ satisfies either one of the following conditions:
     \begin{enumerate}
     \item[(a)] $\widetilde{g}$ is  in $\cC_\varphi^{SOT}(\HH^2(\varphi))$  and
 $
 S_i=\varphi_i(\widetilde{g}_1,\ldots, \widetilde{g}_n)$, $i=1,\ldots,n;
 $
 \item[(b)] $\widetilde{g}$ is  in $\cC_\varphi^{rad}(\HH^2(\varphi))$  and
 $
 S_i= \text{\rm SOT-}\lim_{r\to 1} \varphi_j(r\widetilde{g}_1,\ldots, r\widetilde{g}_n)$, $ i=1,\ldots,n,
 $
 where $(S_1,\ldots, S_n)$ are the left creation operators on the full Fock space $F^2(H_n)$.
 \end{enumerate}

\begin{example}\label{Ex4} If
\begin{equation*}
\begin{split}
\varphi_1 &= \frac{1}{6}Z_1-\frac{1}{8}Z_2-\left(\frac{1}{8}Z_2\right)^2-\cdots\\
\varphi_2 &=  \frac{1}{3}Z_2.
\end{split}
\end{equation*}
Then $\varphi=(\varphi_1,\varphi_2)$ is a free holomorphic function  on $[B(\cH)^2]_2$ and
 has property  $(\mathcal{F})$. In this case,  $\HH^2(\varphi)$ is a Hilbert space of free holomorphic functions on $[B(\cH)^2]_2$.
\end{example}

The theory of noncommutative characteristic functions for row
contractions \cite{Po-charact} was used  in \cite{Po-automorphism}
to determine the group $Aut(B(\cH)^n_1)$  of all  free holomorphic automorphisms of the
noncommutative ball $[B(\cH)^n]_1$. We showed that any $\Psi\in
Aut(B(\cH)^n_1)$ has the form
$$
\Psi=\Phi_U \circ \Psi_\lambda,
$$
where $\Phi_U$ is an automorphism implemented by a unitary operator
$U$ on $\CC^n$, i.e.,
\begin{equation*}
 \Phi_U(X_1,\ldots,
X_n):=[X_1,\ldots, X_n]U , \qquad (X_1,\ldots, X_n)\in [B(\cH)^n]_1,
\end{equation*}
and $\Psi_\lambda$ is an involutive free holomorphic automorphism
associated with $\lambda:=\Psi^{-1} (0)\in \BB_n$. The  automorphism
$\Psi_\lambda:[B(\cH)^n]_1\to [B(\cH)^n]_1$  is   given by
\begin{equation*}
 \Psi_\lambda(X_1,\ldots, X_n):={
\lambda}-\Delta_{ \lambda}\left(I_\cH-\sum_{i=1}^n \bar{{
\lambda}}_i X_i\right)^{-1} [X_1,\ldots, X_n]
\Delta_{{\lambda}^*},\quad (X_1,\ldots, X_n)\in [B(\cH)^n]_1,
\end{equation*}
where $\Delta_\lambda$ and $\Delta_{\lambda^*}$ are the defect
operators associated with  the row contraction
$\lambda:=(\lambda_1,\ldots, \lambda_n)$.
   Note that, when
$\lambda=0$, we have $\Psi_0(X)=-X$. We recall that if $\lambda \in
\BB_n\backslash \{0\}$ and   $\gamma:=\frac{1}{\|\lambda\|_2}$, then
$\Psi_\lambda$ is a free holomorphic function on $[B(\cH)^n]_\gamma$
which has the following properties:
\begin{enumerate}
\item[(i)]
$\Psi_\lambda (0)=\lambda$ and $\Psi_\lambda(\lambda)=0$;

\item[(ii)] $\Psi_\lambda$ is an involution, i.e., $\Psi_\lambda(\Psi_\lambda(X))=X$
for any $X\in [B(\cH)^n]_\gamma$;
\item[(iii)] $\Psi_\lambda$ is a free holomorphic automorphism of the
noncommutative unit ball $[B(\cH)^n]_1$;
\item[(iv)] $\Psi_\lambda$ is a homeomorphism of $[B(\cH)^n]_1^-$
onto $[B(\cH)^n]_1^-$;
\item[(v)]  the model boundary function $\tilde \Psi_\lambda$ is unitarily equivalent  to the row contraction
$[S_1,\ldots, S_n]$.
\end{enumerate}

\begin{proposition}\label{ex-auto}  Any  free holomorphic automorphism
of $[B(\cH)^n]_1$ has property  $(\mathcal{F})$.
\end{proposition}
\begin{proof} Let $\varphi\in Aut(B(\cH)^n_1)$. Since the
composition of free holomorphic functions is a free holomorphic
function, one can easily show, by contradiction, that condition
$(\cF_1)$ is satisfied by $\varphi$. Now, taking into account the
properties of the free holomorphic automorphisms of $[B(\cH)^n]_1$
and the remarks above, we have $\varphi\in H^\infty_{\bf ball}$ and
$\varphi(\varphi(X))=X$ for all $X\in [B(\cH)^n]_1$, which shows
that condition $(\cF_2)$ holds. Moreover, since the multiplication
$M_{X_i}:\HH^2(\varphi)\to \HH^2(\varphi)$ is unitarily equivalent
to  the model boundary function $\widetilde \varphi$ acting on
$F^2(H_n)$, and $M_{\varphi_i}:\HH^2(\varphi)\to \HH^2(\varphi)$ is
unitarily equivalent to $S_i\in B(F^2(H_n))$, the equation
$M_{\varphi_i}=\varphi_i(M_{Z_1},\ldots, M_{Z_n})$ is equivalent to
the equation $S_i=\varphi_i(\widetilde \varphi_1,\ldots,
\widetilde \varphi_n)$, where $(\widetilde \varphi_1,\ldots,
\widetilde \varphi_n)$ is in the convergence set
$\cC_\varphi^{rad}(\HH^2(\varphi))$. Due to the functional calculus
for row contractions \cite{Po-funct},   the latter equality holds for any
$\varphi\in Aut(B(\cH)^n_1)$.  Therefore, $\varphi$ satisfies
condition $(\cF_3)$, which proves our assertion.
\end{proof}

We saw  above that, due to condition $(\cF_2)$, there is  a bounded free
holomorphic function $g:[B(\cH)^n]_1\to B(\cH)^n$ such that
$X=g(\varphi(X))$ for any $X\in [B(\cH)^n]_\gamma$.    We consider
the noncommutative domain
$$
\BB_\varphi(\cH):=\{Y=(Y_1,\ldots, Y_n)\in B(\cH)^n:\
g(\varphi(Y))=Y \text{ and }  \|\varphi(Y)\|\leq 1  \}
$$
which will be studied  in the next sections.
  Note that the ball $[B(\cH)^n]_\gamma$
is included in  $\BB_\varphi(\cH)$.

\bigskip

\section{Noncommutative domains  and the universal model $(M_{Z_1},\ldots, M_{Z_n})$}

Throughout this section, unless otherwise specified,  we assume that
 $f=(f_1,\ldots, f_n)$
   is either one of the following:
 \begin{enumerate}
 \item[(i)]   an $n$-tuple of polynomials with  property $(\cA)$;

 \item[(ii)] an $n$-tuple of  formal power series  with  $f(0)=0$ and    property
 $(\mathcal{S})$;

 \item[(iii)]  an $n$-tuple of  free holomorphic functions  with property $(\mathcal{F})$.
\end{enumerate}
In this case, we say that $f$ has the {\it model property}. We denote by $\cM$ the set of all $n$-tuples $f$ with the model property.  The
noncommutative domain associated with $f$ is
$$
\BB_f(\cH):=\{X=(X_1,\ldots, X_n)\in B(\cH)^n: \ g(f(X))=X \text{
and } \|f(X)\|\leq 1\},
$$
where $g:=(g_1,\ldots, g_n)$ is the inverse power series of $f$ with
respect to the composition, and the evaluations are well-defined
(see previous section).  We recall that  the condition $ g(f(X))=X$
is
 automatically satisfied when $f$ is an $n$-tuple of polynomials with  property
 $(\cA)$.

In this section, we present some of the basic properties of the
universal model $(M_{Z_1},\ldots, M_{Z_n})$ associated with the
noncommutative domain $\BB_f$.

Two $n$-tuples $(A_1,\ldots, A_n)\in B(\cH)$ and $(B_1,\ldots,
B_n)\in B(\cK)$ are said to be unitarily equivalent if there is a
unitary operator $U:\cH\to \cK$ such that $A_i=U^*B_i U$ for all
$i=1,\ldots,n$.

\begin{theorem} \label{model}
Let $T:=(T_1,\ldots, T_n)$ be an $n$-tuple of operators  in
$B(\cH)^n$ and let $f$ have the model property. Then the following
statements are equivalent:
\begin{enumerate}
\item[(i)] $T=(T_1,\ldots, T_n)$ is  a pure $n$-tuple of operators in $\BB_f(\cH)$;
\item[(ii)] there exists a Hilbert space $\cD$ and  a co-invariant
 subspace $\cM\subseteq \HH^2(f)\otimes \cD$
 under each operator $M_{Z_1}\otimes I_\cD,\ldots,M_{Z_n}\otimes I_\cD$ such that
 the $n$-tuple
$(T_1,\ldots, T_n)$ is unitarily equivalent  to
$$
(P_\cM(M_{Z_1}\otimes I_\cD)|_\cM,\ldots, P_\cM(M_{Z_n}\otimes I_\cD)|_\cM).
$$
\end{enumerate}

\end{theorem}
\begin{proof} We shall prove the theorem when $f$ is an $n$-tuple of
formal power series with $f(0)=0$ and has property $(\cS)$. The
other two cases can be treated similarly. Let $g=(g_1,\ldots, g_n)$
be the inverse of $f$ with respect to the composition. Note that
condition $(\mathcal{S}_3)$ implies
$$
\sum_{j=1}^n f_j(M_{Z_1},\ldots, M_{Z_n}) f_j(M_{Z_1},\ldots,
M_{Z_n})^* =\sum_{j=1}^n M_{f_j} M_{f_j}^*=U^{-1}\left(\sum_{j=1}^n
S_jS_j^*\right)U\leq I,
$$
where $U:\HH^2(f)\to F^2(H_n)$ is the unitary operator defined by
$U(f_\alpha):=e_\alpha$, $\alpha\in \FF_n^+$. Since $M_{f_\alpha}=[f
(M_{Z_1},\ldots, M_{Z_n})]_\alpha$, $M_{f_\alpha}=U^{-1} S_\alpha
U$, $\alpha\in \FF_n^+$,  and SOT-$\lim_{p\to \infty}
\sum_{|\alpha|=p}S_\alpha S_\alpha^*=0$, we deduce that the
$n$-tuple $M_Z:=(M_{Z_1},\ldots, M_{Z_n})$ is a pure element with
$\|f(M_Z)\|\leq 1$. Now let us show that  $M_Z $ is in the
noncommutative domain $\BB_f(\HH^2(f))$. It remains to prove that
$g(f(M_Z))=M_Z$ which, due to condition $(\mathcal{S}_3)$, is
equivalent to
\begin{equation}\label{gi}
g_i(M_{f_1},\ldots, M_{f_n})=M_{Z_i},\qquad i=1,\ldots,n.
\end{equation}
According to Lemma \ref{M2}, if $g_i=\sum_{\alpha\in \FF_n^+}
a_\alpha^{(i)} Z_\alpha$,  then $M_{Z_i}=U^{-1}
\varphi_i(S_1,\ldots,S_n) U$, where $\varphi_i(S_1,\ldots,S_n)\in
F_n^\infty$   has the Fourier representation $\sum_{\alpha\in
\FF_n^+} a_\alpha^{(i)} S_\alpha$. Proving the equality above is
equivalent to showing that
$$
\text{\rm SOT-}\lim_{r\to 1}\sum_{k=0}^\infty \sum_{|\alpha|=k}
a_\alpha^{(i)} r^{|\alpha|}
S_\alpha=\varphi_i(S_1,\ldots,S_n),\qquad i=1,\ldots,n.
$$
The latter relation is well-known (see \cite{Po-funct}). Therefore,
$M_Z\in \BB_f(\HH^2(f))$. If $\cD$ is a Hilbert space and
$\cM\subseteq \HH^2(f)\otimes \cD$
 is a co-invariant  subspace under $M_{Z_1}\otimes I_\cD,\ldots,M_{Z_n}\otimes I_\cD$, then
$$
[f(P_\cM(M_{Z_1}\otimes I_\cD)|_\cM,\ldots, P_\cM(M_{Z_n}\otimes
I_\cD)|_\cM)]_\alpha = P_\cM\left\{[f(M_{Z_1},\ldots,
M_{Z_n})]_\alpha\otimes I_\cD\right\}|_\cM
$$
for any $\alpha\in \FF_n^+$.  Due  to relation
$$
g_i(f_1(M_{Z_1},\ldots, M_{Z_n}),\ldots, f_n(M_{Z_1},\ldots,
M_{Z_n}))=M_{Z_i},\qquad i=1,\ldots,n,
$$
we deduce that
$$
\text{\rm SOT-}\lim_{r\to 1} \sum_{k=0}^\infty \sum_{|\alpha|=k} a_\alpha^{(i)} r^{|\alpha|}
 [f(M_{Z_1},\ldots, M_{Z_n})]_\alpha=M_{Z_i},\qquad i=1,\ldots,n.
$$
Taking the compression to the  subspace  $\cM\subseteq \HH^2(f)\otimes \cD$, we deduce that
  $$
  g_i(f(P_\cM(M_{Z_1}\otimes I_\cD)|_\cM,\ldots, P_\cM(M_{Z_n}\otimes
  I_\cD)|_\cM))
  =P_\cM(M_{Z_i}\otimes I_\cD)|_\cM
  $$
  for  each $i=1,\ldots,n$.
  Since $(M_{Z_1},\ldots, M_{Z_n})$ is a
pure element in $\BB_f(\HH^2(f))$, we deduce   that the
$n$-tuple $ (P_\cM(M_{Z_1}\otimes I_\cD)|_\cM,\ldots,
P_\cM(M_{Z_n}\otimes I_\cD)|_\cM) $
 is a pure element in $\BB_f(\cM)$. Therefore, the implication
 $(ii)\implies (i)$ holds.

Now, we prove the implication $(i)\implies (ii)$. Assume that
condition (i) holds. Let $T=(T_1,\ldots,T_n)\in \BB_f(\cH)$ be a
pure $n$-tuple of operators. Consider the defect operator
$$
\Delta_{f,T}:=\left(I-\sum_{j=1}^nf_j(T)f_j(T)^*\right)^{1/2}
$$
and the defect space $\cD_{f,T}:=\overline{\Delta_{f}(T)\cH}$. Define
the noncommutative Poisson kernel $K_{f,T}:\cH\to \HH^2(f)\otimes
\cD_{f,T}$ by setting
\begin{equation}
\label{KfT} K_{f,T}h:=\sum_{\alpha\in \FF_n^+} f_\alpha\otimes
\Delta_{f,T} [f (T)]_\alpha^*h,\qquad h\in \cH.
\end{equation}
We need to prove that $K_{f,T}$ is an isometry and
$K_{f,T}T_i^*=(M_{Z_i}\otimes I_{\cD_{f,T}})K_{f,T}$ for any
$i=1,\ldots,n$. Indeed, a straightforward calculation reveals that
\begin{equation*}
\begin{split}
\left\|\sum_{\alpha\in \FF_n^+, |\alpha|\leq q} f_\alpha\otimes
\Delta_{f,T} [f (T)]_\alpha^*h\right\|^2_{\HH^2(f)\otimes
\cH}&=\sum_{\alpha\in \FF_n^+,
 |\alpha|\leq q} \left\| \Delta_{f,T} [f (T)]_\alpha^*h\right\|^2_\cH\\
&=\sum_{\alpha\in \FF_n^+, |\alpha|\leq q} \left<[f (T)]_\alpha\Delta_{f,T}^2 [f (T)]_\alpha^*h, h\right>\\
&=\|h\|-\left<\left(\sum_{\alpha\in \FF_n^+,
 |\alpha|=q} [f (T)]_\alpha [f (T)]_\alpha^*\right)h, h\right>
\end{split}
\end{equation*}
for any $q\in \NN$. Since $T=(T_1,\ldots, T_n)$ is  a pure $n$-tuple
in $\BB_f(\cH)$ we have
$$
\text{\rm SOT-}\lim_{q\to \infty} \sum_{\alpha\in \FF_n,\, |\alpha|=q} [f (T)]_\alpha [f (T)]_\alpha^*=0.
$$
Consequently, we obtain  $\|K_{f,T}h\|=\|h\|$ for any $h\in \cH$. On
the other hand,  for any $h,h'\in \cH$ and $\alpha\in \FF_n^+$, we
have
\begin{equation*}
\begin{split}
\left<K_{f,T}^*(f_\alpha\otimes h), h'\right>&=\left<f_\alpha\otimes h, K_{f,T} h'\right>\\
&=\left< h, \Delta_{f,T} [f (T)]_\alpha^*h'\right>\\
&=\left< [f (T)]_\alpha\Delta_{f,T}h, h'\right>.
\end{split}
\end{equation*}
Therefore,
\begin{equation}\label{KT}
K_{f,T}^*(f_\alpha\otimes h)=[f (T)]_\alpha\Delta_{f,T}h,\qquad h\in
\cH.
\end{equation}
Since the $n$-tuple  $f$ has the property $(\mathcal{S})$, for each
$i=1,\ldots,n$,  $Z_i\in \HH^2(f)$, i.e., there is a sequence
$\{a_\alpha^{(i)}\}_{\alpha\in \FF_n^+}$ with $\sum_{\alpha\in
\FF_n^+} |a_\alpha^{(i)}|^2<\infty$  such that
$$
Z_i=\sum_{\alpha\in \FF_n^+} a_\alpha^{(i)}[f(Z)]_\alpha.
$$
 Taking into account that  $T:=(T_1,\ldots, T_n)\in \BB_f(\cH)$,
we have either $T\in \cC_f^{SOT}(\cH)$ or  $T\in
  \cC_f^{rad}(\cH)$.
Let us consider first the case when $T\in \cC_f^{SOT}(\cH)$. The
equation $T=g(f(T))$ shows that  either
     \begin{equation}\label{one}
T_i=\sum_{k=0}^\infty\sum_{\alpha\in \FF_n^+, |\alpha|=k}
a_\alpha^{(i)}[f(T)]_\alpha,  \qquad i=1,\ldots,n,
\end{equation}
  where the  convergence of the  series  is in the strong operator topology, or
  \begin{equation} \label{two}
 T_i= \text{\rm SOT-}\lim_{r\to 1} \sum_{k=0}^\infty\sum_{\alpha\in \FF_n^+, |\alpha|=k}
a_\alpha^{(i)}r^{|\alpha|}[{f}(T)]_\alpha ,\qquad i=1,\ldots,n.
 \end{equation}
When relation \eqref{one} holds, then we have
\begin{equation}
\label{conv2} T_i[f(T)]_\beta=\sum_{k=0}^\infty\sum_{\alpha\in
\FF_n^+, |\alpha|=k} a_\alpha^{(i)}[f(T)]_\alpha[f(T)]_\beta, \qquad
i=1,\ldots,n,\ \beta\in \FF_n^+,
\end{equation}
where the  convergence of the series  is in the strong operator
topology. Using relation \eqref{KT}, we deduce that
$$
K_{f,T}^*\left(\sum_{k=0}^p \sum_{|\alpha|=k} a_\alpha^{(i)}
f_\alpha f_\beta\otimes h\right) =\left(\sum_{k=0}^p
\sum_{|\alpha|=k} a_\alpha^{(i)}[f (T)]_\alpha
[f(T)]_\beta\right)\Delta_{f,T}h,\qquad h\in \cH,
$$
for any $p\in \NN$.  Hence, due to relation \eqref{conv2} and the
fact that $M_{Z_i}f_\beta=\sum_{\alpha\in \FF_n^+} a_\alpha^{(i)}
f_\alpha f_\beta$ in $\HH^2(f)$, we obtain
$$
K_{f,T}^*(M_{Z_i}f_\beta\otimes h)=T_i
[f(T)]_\beta\Delta_{f,T}h,\qquad h\in \cH,
$$
which, combined with relation \eqref{KT}, implies
\begin{equation*}
\begin{split}
K_{f,T}^*(M_{Z_i}\otimes I)(f_\beta\otimes h)&=
T_iK_{f,T}^*(f_\beta\otimes h)
\end{split}
\end{equation*}
for any $\beta\in \FF_n^+$ and $i=1,\ldots,n$. Consequently
$$K_{f,T}T_i^*=(M^*_{Z_i}\otimes I)K_{f,T}$$
 for any $i=1,\ldots,n$.
Now, we assume that relation \eqref{two} holds. Then, using relation
\eqref{KT}, we deduce that
$$
K_{f,T}^*\left(\sum_{k=0}^p \sum_{|\alpha|=k}
a_\alpha^{(i)}r^{|\alpha|} f_\alpha f_\beta\otimes h\right)
=\left(\sum_{k=0}^p \sum_{|\alpha|=k} a_\alpha^{(i)}r^{|\alpha|}[f
(T)]_\alpha [f(T)]_\beta\right)\Delta_{f,T}h,\qquad h\in \cH,
$$
for any $p\in \NN$ and $r\in [0,1)$. Taking  first $p\to \infty$ and
then $r\to 1$, we obtain $ K_{f,T}^*(M_{Z_i}f_\beta\otimes h)=T_i
[f(T)]_\beta\Delta_{f,T}h$, $h\in \cH$. This implies
$K_{f,T}T_i^*=(M^*_{Z_i}\otimes I)K_{f,T}$
 for any $i=1,\ldots,n$. The case when $T\in
  \cC_f^{rad}(\cH)$ can be treated similarly.
The proof is complete.
\end{proof}

Any $n$-tuple $(T_1,\ldots, T_n)\in \BB_f(\cH)$ gives
 rise to a Hilbert module over $\CC[Z_1,\ldots, Z_n]$ by setting
$$
p\cdot h:=p(T_1,\ldots, T_n)h,\qquad p\in \CC[Z_1,\ldots, Z_n] \ \text{ and } \  h\in \cH,
$$
which we call $\BB_f$-Hilbert module.
The homomorphisms in this category are the contractive operators intertwining the module action.
 If $\cK\subseteq \cH$ is a closed subspace  of $\cH$ which is invariant under the action
 of the associated operators with $\cH$, i.e., $T_1,\ldots, T_n $, then $\cK$ and
  the quotient $\cH/\cK$ have natural $\BB_f$-Hilbert module structure coming from that of $\cH$.
  More precisely, the canonical  $n$-tuples associated with
$\cK$  and $\cH/\cK$  are  $(T_1|_\cK,\ldots, T_n|_\cK)\in \BB_f(\cK)$
and $(P_{\cK^\perp}T_1|_{\cK^\perp},\ldots, P_{\cK^\perp}T_n|_{\cK^\perp})\in \BB_f(\cK^\perp)$,
 respectively, where $P_{\cK^\perp}$ is the orthogonal projection of $\cH$ onto $\cK^\perp:=\cH\ominus \cK$.

Each noncommutative domain $\BB_f$  has  a universal model
$(M_{Z_1},\ldots, M_{Z_n})\in \BB_f(\HH^2(f))$. The module structure
defined by $M_{Z_1},\ldots, M_{Z_n}$ on the Hilbert space $\HH^2(f)$
occupies the position of the rank-one free module in the algebraic
theory \cite{K}. More precisely, the free $\BB_f$-Hilbert module  of
rank one $\HH^2(f)$ has a universal property in the category of pure
$\BB_f$-Hilbert modules of finite rank. Indeed, it is a consequence
of Theorem \ref{model}  that if $\cH$ is a pure finite rank
$\BB_f$-Hilbert module over $\CC[Z_1,\ldots, Z_n]$,   then there
exist $m\in \NN$ and a closed submodule $\cM$ of $\HH^2(f)\otimes
I_{\CC^m}$ such that $(\HH^2(f)\otimes I_{\CC^m})/\cM$ is isomorphic
to $\cH$. To clarify our terminology,  we mention that the rank of a
$\BB_f$-Hilbert module $\cH$  is  the rank of the defect operator
$\Delta_{f,T}$, while $\cH$ is called
 pure if   $T$ is  a pure  $n$-tuple in $\BB_f(\cH)$.

We introduce the {\it dilation index} of $T=(T_1,\ldots, T_n)\in
\BB_f(\cH)$,  denoted by $\text{\rm dil-ind}\,(T)$, to be the
minimum dimension of the Hilbert space $\cD$  in Theorem
\ref{model}. According to the proof of  the latter  theorem, we
deduce that $\text{\rm dil-ind}\,(T)\leq \dim \cD_{f,T}=\rank
\Delta_{f,T}$. On the other hand, let $\cG$ be a Hilbert space  such
that $\cH$ can be identified with a co-invariant subspace of
$\HH^2(f)\otimes \cG$ under $M_{Z_i}\otimes I_\cG$, \ $i=1,\ldots,
n$, and  such that
 $T_i=P_\cH(M_{Z_i}\otimes I_\cG)|\cH$ for  $i=1,\ldots, n$.
Then
\begin{equation*}
\begin{split}
I_\cH-\sum_{i=1}^n f_i(T)f_i(T)^*&=
P_\cH\left[ \left(I_{\HH^2(f)}-\sum_{i=1}^n f_i(M_{Z_1},\ldots, M_{Z_n})f_i(M_{Z_1},\ldots, M_{Z_n})^*\right)\otimes I_\cG
\right]|_\cH\\
&=P_\cH (\Delta_{f,M_Z}^2\otimes I_\cG)|_\cH=P_\cH (P_\CC\otimes I_\cG)|_\cH.
\end{split}
\end{equation*}
Hence,  we obtain $\rank \Delta_{f,T}\leq \dim \cG$. Therefore, we
have proved that $\text{\rm dil-ind}\,(T)=\rank \Delta_{f,T}$.

\begin{corollary} If $(T_1,\ldots, T_n)$ is  a pure $n$-tuple of operators in $\BB_f(\cH)$, then
$$
 T_\alpha T_\beta^*=K_{f,T}^*[ (M_{Z_\alpha} M_{Z_\beta}^* )\otimes
 I]K_{f,T},\qquad \alpha,\beta\in \FF_n^+,
$$
and
$$
 \left\|\sum_{i=1}^m q_i(T_1,\ldots, T_n)q_i(T_1,\ldots,
T_n)^*\right\| \leq \left\|\sum_{i=1}^m q_i(M_{Z_1},\ldots,
M_{Z_n})q_i(M_{Z_1},\ldots, M_{Z_n})\right\|
$$
for any     $q_i\in \CC[Z_1,\ldots,Z_n]$ and $m\in\NN$.
\end{corollary}

\begin{theorem}\label{irreducible} Let $f=(f_1,\ldots, f_n)$ be  an  $n$-tuple of formal power series
 with the model
 property and let    $(M_{Z_1},\ldots,
 M_{Z_n})$  be  the universal model  associated with the noncommutative domain $\BB_f$.
 Then the   $C^*$-algebra $C^*(M_{Z_1},\ldots,
 M_{Z_n})$  is irreducible  and  coincides with
 $$
 \overline{\text{\rm span}} \{M_{Z_\alpha} M_{Z_\beta}^*:\ \alpha,\beta\in
 \FF_n^+\}.
 $$
\end{theorem}
\begin{proof}
Let $\cM\subset \HH^2(f)$ be a nonzero subspace which is jointly
reducing for $M_{Z_1},\ldots, M_{Z_n}$, and let $y=\sum_{\alpha\in
\FF_n^+} a_\alpha f_\alpha$ be a nonzero power series in  $\cM$.
Then there is $\beta\in \FF_n^+$ such that $a_\beta\neq 0$. Since
$f=(f_1,\ldots, f_n)$ is  an  $n$-tuple of formal power series
 with the model
 property, we have $M_{f_i}=f_i(M_{Z})$,
where  $M_Z:=(M_{Z_1},\ldots, M_{Z_n})$ is either in the convergence
set $\cC_f^{SOT}(\HH^2(f))$ or $\cC_f^{rad}(\HH^2(f))$.
Consequently, we obtain
$$
a_\beta=P_\CC M_{f_\beta}^*y=\left(I-\sum_{i=1}^n
f_i(M_{Z})f_i(M_{Z})^*\right)[f(M_{Z})]_\beta^* y.
$$
Taking into account that $\cM$ is reducing for $M_{Z_1},\ldots,
M_{Z_n}$ and $a_\beta\neq 0$, we deduce that $1\in \cM$. Using again
that $\cM$ is invariant under  $M_{Z_1},\ldots, M_{Z_n}$, we obtain
$\CC[Z_1,\ldots, Z_n]\subset \cM$. Since, according to Proposition
\ref{dense}, $\CC[Z_1,\ldots, Z_n]$ is dense in $\HH^2(f)$, we
conclude that $\cM=\HH^2(f)$, which shows that $C^*(M_{Z_1},\ldots,
 M_{Z_n})$  is irreducible.

Since $f$ has the model property, we have $Z_i=\sum_{\alpha\in
\FF^+} a_\alpha^{(i)} f_\alpha\in \HH^2(f)$ and the multiplication
$M_{Z_i}$ is a bounded multiplier of $\HH^2(f)$ which satisfies the
equation
$$
M_{Z_i}=\text{SOT-}\lim_{r\to 1} \sum_{k=0}^\infty\sum_{|\alpha|=k}
a_\alpha^{(i)}r^{|\alpha|} M_{f_\alpha},\qquad i=1,\ldots,n.
$$
Hence, and taking into account that
$$
f_i(M_Z)^* f_j(M_Z)=M_{f_i}^*M_{f_j}=\delta_{ij}I,\quad i,j\in
\{1,\ldots,n\},
$$
we deduce that, for any $x,y\in \HH^2(f)$,
 \begin{equation*}
\begin{split}
\left< M_{Z_i}^* M_{Z_j}x,y\right>&=\lim_{r\to 1}
\left<\sum_{k=0}^\infty\sum_{|\beta|=k}
a_\beta^{(j)}r^{|\beta|}[f(M_Z)]_\beta
x,\sum_{k=0}^\infty\sum_{|\alpha|=k} a_\alpha^{(i)} r^{|\alpha|}
[f(M_Z)]_\alpha
y\right>\\
&= \lim_{r\to 1}\lim_{m\to\infty} \left<\sum_{|\alpha|\leq
m}\sum_{k=0}^\infty\sum_{|\beta|=k} \overline{a_\alpha^{(i)}}
a_\beta^{(j)}r^{|\alpha|+|\beta|}[f(M_Z)]_\alpha^*[f(M_Z)]_\beta x,
y\right>\\
&= \lim_{r\to 1}\lim_{m\to\infty} \left<\sum_{|\alpha|\leq
m}\sum_{k=0}^\infty\sum_{|\beta|=k} \overline{a_\alpha^{(i)}}
a_\beta^{(j)}r^{|\alpha|+|\beta|} \delta_{\alpha \beta} x,
y\right>\\
&=\lim_{r\to 1}\lim_{m\to\infty}\sum_{|\alpha|\leq
m}\overline{a_\alpha^{(i)}}
a_\alpha^{(j)}r^{2|\alpha|}\left<x,y\right>\\
&=\lim_{r\to 1}\sum_{k=0}^\infty
\sum_{|\alpha|=k}\overline{a_\alpha^{(i)}}
a_\alpha^{(j)}r^{2|\alpha|}\left<x,y\right>\\
&=\left< Z_j,Z_i\right>_{\HH^2(f)}\left<x, y\right>_{\HH^2(f)}.
\end{split}
\end{equation*}
Hence, we deduce that
$$
M_{Z_i}^* M_{Z_j}=\left< Z_j,Z_i\right>_{\HH^2(f)}
I_{\HH^2(f)},\qquad i,j\in \{1,\ldots,n\},
$$
and, therefore, $C^*(M_{Z_1},\ldots,
 M_{Z_n})$   coincides with
 $$
 \overline{\text{\rm span}} \{M_{Z_\alpha} M_{Z_\beta}^*:\ \alpha,\beta\in
 \FF_n^+\}.
 $$
The proof is complete.
\end{proof}

Let $f=(f_1,\ldots, f_n)$ be   an $n$-tuple of formal power series
with the model property.
      We say that  $f$ has the radial
approximation property, and write  $f\in \cM_{rad}$,   if there is
$\delta\in (0,1)$ such that $(rf_1,\ldots, rf_n)$ has the model
property  for any $r\in (\delta, 1]$.
Denote by  $\cM^{||}$  the set of all  formal power series
$f=(f_1,\ldots, f_n)$ having  the model property  and such that the
universal model $(M_{Z_1},\ldots, M_{Z_n})$ associated with the
noncommutative domain $\BB_f$ is in the set of norm-convergence (or radial
norm-convergence) of $f$. We also introduce  the class
$\cM_{rad}^{||}$ of all formal power series $f=(f_1,\ldots, f_n)$
with the property that there is $\delta\in (0,1)$ such that $rf\in
\cM^{||}$ for any $r\in (\delta, 1]$.

\begin{lemma}\label{rad}
 Let $f=(f_1,\ldots, f_n)$ be   an $n$-tuple of formal power series
 with the model
  property  and let $g=(g_1,\ldots,g_n)$ be its inverse with respect
  to the composition. Setting $g_i=\sum_{\alpha\in \FF_n^+}
a_\alpha^{(i)} Z_\alpha$,  the following statements are equivalent.
\begin{enumerate}
\item[(i)] The $n$-tuple $ f$ has
the radial approximation property.
\item[(ii)] There is \ $\delta\in (0,1)$ with the property
that  $g_i(\frac{1}{r} S):=\sum_{\alpha\in \FF_n^+}
\frac{a_\alpha^{(i)}}{r^{|\alpha|}} S_\alpha$
 is the Fourier representation of an element
    in $F_n^\infty$
and
  $$
  \frac{1}{r} S_j=f_j\left(g_1\left(\frac{1}{r} S\right),\ldots,
  g_n\left(\frac{1}{r} S\right)\right),\qquad i,j\in\{1,\ldots, n\}, \ r\in (\delta,
  1],
  $$
where  $g(\frac{1}{r} S)$ is either in the convergence set \
$\cC_f^{SOT}(F^2(H_n))$ or $\cC_f^{rad}(F^2(H_n))$, and
$S=(S_1,\ldots, S_n)$ is the $n$-tuple of
      left creation operators on $F^2(H_n)$.
If $f$ is an $n$-tuple of noncommutative polynomials, then the later
condition is automatically satisfied.
      \end{enumerate}
      Moreover, $f\in \cM_{rad}^{||}$ if and only if item (ii) holds and
       $g(\frac{1}{r} S)$ is in the set of norm-convergence (or radial
norm-convergence) of $f$.
\end{lemma}
\begin{proof} The proof is straightforward if one uses Lemma
\ref{M2} (and the proof) and its analogues when $f$ is an $n$-tuple
of polynomials with property $(\cA)$  or a free holomorphic function
with property $(\cF)$.
\end{proof}

\begin{remark}  In all
the examples presented in this paper, the corresponding  $n$-tuple $f=(f_1,\ldots, f_n)$ is in the class
 $\cM_{rad}^{||}$. Moreover, any   $n$-tuple   of polynomials
with property $(\cA)$ is also in  the class $\cM_{rad}^{||}$.
\end{remark}

\begin{proposition}\label{lots}
Let  $f=(f_1,\ldots, f_n)$   be an  $n$-tuple  of
 formal power series with $f(0)=0$ and $\det J_f(0)\neq 0$, and let
$g=(g_1,\ldots, g_n)$ be its inverse. Assume that $f$ and $g$ have
 nonzero radius of convergence.   Then
 \begin{enumerate}
 \item[(i)] $f(g(X))=X$ for any $X\in [B(\cH)^n]_{\gamma_1}$, where
 $0<\gamma_1<r(g)$ and  $g\left([B(\cH)^n]_{\gamma_1}\right)\subset [B(\cH)^n]_{r(f)}$.

  \item[(ii)]
$ g(f(X))=X$ for any $X\in [B(\cH)^n]_{\gamma_2}$, where
$0<\gamma_2<r(f)$ and  $f([B(\cH)^n]_{\gamma_2})\subset [B(\cH)^n]_{r(g)}$.

 \end{enumerate}
 If \ $\gamma_1>1$,  then $ f\in \cM_{rad}^{||}$, and,
  if $0<\gamma<\gamma_1\leq 1$, then $\frac{1}{\gamma} f$ has the same  property.
\end{proposition}
\begin{proof}
Since $g$ has nonzero radius of convergence and $g(0)=0$, the
Schwartz lemma for free holomorphic functions implies that there is
$\gamma_1\in (0, r(g))$ such that $\|g(X)\|<r(f)$ for any $X\in
[B(\cH)^n]_{\gamma_1}$. On the other hand, using Theorem 1.2 from
\cite{Po-automorphism}, the composition $f\circ g$ is a free
holomorphic function on $[B(\cH)^n]_{\gamma_1}$. Due to the
uniqueness theorem for free holomorphic functions  and the fact that
$f\circ g=id$, as formal power series, we deduce that $f(g(X))=X$
for any $X\in [B(\cH)^n]_{\gamma_1}$. Item (ii) can be proved
similarly. Now, using Lemma \ref{rad}, we can deduce the last part
of the proposition.
\end{proof}

We remark that  Proposition \ref{lots}  does not imply the existence
of a free biholomorphic   function from $[B(\cH)^n]_{\gamma_1}$ to
$[B(\cH)^n]_{\gamma_2}$ (see the examples presented in this paper).

Let $f=(f_1,\ldots, f_n)$ be an $n$-tuple with the model property
and let $T:=(T_1,\ldots, T_n)\in \BB_f(\cH)$. We say that an
$n$-tuple $V:=(V_1,\ldots, V_n)$ of operators on a Hilbert space
$\cK\supseteq \cH$ is a minimal dilation of $T$ if the following
properties are satisfied:
\begin{enumerate}
\item[(i)] $(V_1,\ldots, V_n)\in \BB_f(\cK)$;
\item[(ii)] there  is a $*$-representation $\pi:C^*(M_{Z_1},\ldots, M_{Z_n})\to B(\cK)$ such that $\pi(M_{Z_i})=V_i$, $i=1,\ldots,n$;
\item[(iii)]$V_i^*|_\cH=T_i^*$ for  $i=1,\ldots,n$;
\item[(iv)] $\cK=\bigvee_{\alpha\in \FF_n^+} V_\alpha \cH$.
\end{enumerate}
Without the condition (iv), the $n$-tuple $V$ is called dilation of $T$.
 For  information  on completely bounded (resp. positive) maps,  we refer
 to Paulsen's book \cite{Pa-book}.

\begin{theorem}\label{Poisson-C*} Let $f=(f_1,\ldots, f_n)$ be  an $n$-tuple of formal power series
with the radial approximation property and let
 $T:=(T_1,\ldots, T_n)$ be an $n$-tuple of operators in the
 noncommutative domain $ \BB_f(\cH)$. Then  the following statements hold.
 \begin{enumerate}
 \item[(i)]
 There is
    a unique unital completely contractive linear map
$$
\Psi_{f,T}: C^*(M_{Z_1},\ldots, M_{Z_n})\to B(\cH)
$$
such that
 $$
\Psi_{f,T}(M_{Z_\alpha} M_{Z_\beta}^*)=T_\alpha T_\beta^*, \quad
\alpha,\beta\in \FF_n^+.
$$
\item[(ii)] If $f\in \cM_{rad}\cap \cM^{||}$, then there is a
 minimal dilation of $T$ which is unique up to an isomorphism.
\end{enumerate}

\end{theorem}

\begin{proof} According to Lemma \ref{rad}, there is $\delta\in (0,1)$ such that, for each $r\in (\delta,1]$ and $i=1,\ldots,n$, the multiplication operator
$M_{Z_i}^{(r)}:\HH^2(rf)\to \HH^2(rf)$, defined by $M_{Z_i}^{(r)}\psi:=Z_i \psi$, is unitarily  equivalent to an operator $\varphi_i(\frac{1}{r}S)\in F_n^\infty$ having the Fourier representation $\sum_{\alpha\in \FF_n^+}a_\alpha^{(i)}\frac{1}{r^{|\alpha|}}S_\alpha$.
Therefore,  for any $a_{\alpha,\beta}\in \CC$,
\begin{equation}
\label{equ}
\left\|\sum_{|\alpha|,|\beta|\leq m} a_{\alpha,\beta} M_{Z_\alpha}^{(r)}{M_{Z_\beta}^{(r)}}^*\right\|
=\left\|\sum_{|\alpha|,|\beta|\leq m} a_{\alpha,\beta} \varphi_\alpha\left(\frac{1}{r}S\right)\varphi_\beta\left(\frac{1}{r}S\right)^*\right\|.
\end{equation}
Note that $(T_1,\ldots, T_n)$ is a pure $n$-tuple in $\BB_{rf}(\cH)$ for any $r\in(\delta, 1)$. Applying Theorem \ref{model}, we deduce that
\begin{equation}
\label{equ2}
\left\|\sum_{|\alpha|,|\beta|\leq m} a_{\alpha,\beta} T_\alpha T_\beta^*\right\|
\leq \left\|\sum_{|\alpha|,|\beta|\leq m} a_{\alpha,\beta} M_{Z_\alpha}^{(r)}{M_{Z_\beta}^{(r)}}^*\right\|.
\end{equation}
On the other hand, according to \cite{Po-holomorphic},
$\varphi_i(\frac{t}{r}S)$ is in the noncommutative disc algebra
$\cA_n$   for any $t\in (0,1)$, and the map $(0,1)\ni t\to
\varphi_i(\frac{t}{r}S)$ is continuous in the operator norm
topology. Consequently,
$$
\lim_{r\to 1}\left\|\sum_{|\alpha|,|\beta|\leq m} a_{\alpha,\beta}
\varphi_\alpha\left(\frac{1}{r}S\right)\varphi_\beta\left(\frac{1}{r}S\right)^*\right\|
=\left\|\sum_{|\alpha|,|\beta|\leq m} a_{\alpha,\beta}
\varphi_\alpha\left(S\right)\varphi_\beta\left(S\right)^*\right\|
=\left\|\sum_{|\alpha|,|\beta|\leq m} a_{\alpha,\beta}
M_{Z_\alpha}M_{Z_\beta}^*\right\|.
$$
Combining this with relations \eqref{equ} and \eqref{equ2}, we have
$$
\left\|\sum_{|\alpha|,|\beta|\leq m} a_{\alpha,\beta} T_\alpha T_\beta^*\right\|
\leq \left\|\sum_{|\alpha|,|\beta|\leq m} a_{\alpha,\beta} M_{Z_\alpha}M^*_{Z_\beta}\right\|.
$$
A similar inequality can be obtained if we pass to matrices with entries in
 $C^*(M_{Z_1},\ldots, M_{Z_n})$. Now, an approximation  argument shows that
 the map
 $$\sum_{|\alpha|,|\beta|\leq m} a_{\alpha,\beta} M_{Z_\alpha}M^*_{Z_\beta}\mapsto
 \sum_{|\alpha|,|\beta|\leq m} a_{\alpha,\beta} T_\alpha T_\beta^*
 $$
 can be extended to a unique unital completely contractive map on $
 \overline{\text{\rm span}} \{M_{Z_\alpha} M_{Z_\beta}^*:\ \alpha,\beta\in
 \FF_n^+\}.
 $
 Since, due to Theorem \ref{irreducible},  the latter span  coincides with
  $C^*(M_{Z_1},\ldots, M_{Z_n})$,  item (i) follows.
 Now, we assume that $f\in \cM_{rad}\cap \cM^{||}$. Applying Stinespring's dilation \cite{St} to the unital  completely positive linear map $\Psi_{f,T}$ and taking into account that
$C^*(M_{Z_1},\ldots, M_{Z_n})=\overline{\text{\rm span}} \{M_{Z_\alpha} M_{Z_\beta}^*:\ \alpha,\beta\in
 \FF_n^+\},
 $
we  find a unique representation $\pi:C^*(M_{Z_1},\ldots, M_{Z_n}) \to B(\cK)$, where $\cK\supseteq \cH$, such that $\pi(M_{Z_i})^*|_{\cH}=T_i^*$, $i=1,\ldots,n$, and $\cK=\bigvee_{\alpha\in \FF_n^+} \pi(M_{Z_\alpha})\cH$. Setting
$V_i:=\pi(M_{Z_i})$, $i=1,\ldots,n$, it remains to prove that $(V_1,\ldots, V_n)\in \BB_f(\cK)$. To this end, note that since $(M_{Z_1},\ldots, M_{Z_n})\in  \cM^{||}$, we have $f_i(M_{Z_1},\ldots, M_{Z_n})\in C^*(M_{Z_1},\ldots, M_{Z_n})$ for $i=1,\ldots,n$. Consequently, the inequality
$\sum_{i=1}^n f_i(M_{Z_1},\ldots, M_{Z_n})f_i(M_{Z_1},\ldots, M_{Z_n})^*\leq I_{\HH^2(f)}$ implies
$$\sum_{i=1}^n f_i(\pi(M_{Z_1}),\ldots, \pi(M_{Z_n}))f_i(\pi(M_{Z_1}),\ldots, \pi(M_{Z_n}))^*\leq I_{\cK}
$$
On the other hand, since $g(f(M_{Z_1},\ldots, M_{Z_n}))=M_{Z_i}$, where the convergence is in the operator norm topology, we deduce that $g_i(f(\pi(M_{Z_1}),\ldots, \pi(M_{Z_n})))=\pi(M_{Z_i})$, $i=1,\ldots,n$. Therefore, the $n$-tuple
$(\pi(M_{Z_1}),\ldots, \pi(M_{Z_n}))$ is in the noncommutative domain $\BB_f(\cK)$.
The proof is complete.
\end{proof}

Let $\cS\subset C^*(M_{Z_1},\ldots, M_{Z_n})$ be the operator system
defined by
$$
\cS:=\{p(M_{Z_1},\ldots, M_{Z_n})+q(M_{Z_1},\ldots, M_{Z_n})^*: \
p,q \in \CC[Z_1,\ldots,Z_n]   \}.
$$

\begin{theorem}\label{cc}
Let $f=(f_1,\ldots, f_n)\in \cM_{rad}\cap \cM^{||}$  and $ (T_1,\ldots,
T_n)\in B(\cH)^n$.  Then the following  statements are equivalent:
\begin{enumerate}
\item[(i)] $(T_1,\ldots, T_n)\in \BB_f(\cH)$;
\item[(ii)] the map \ $q(M_{Z_1},\ldots, M_{Z_n})\mapsto q(T_1,\ldots,
T_n)$ is completely contractive;
\item[(iii)]
The map $\Psi:\cS\to B(\cH)$ defined by
$$
\Psi\left(p(M_{Z_1},\ldots, M_{Z_n})+q(M_{Z_1},\ldots,
M_{Z_n})^*\right):= p(T_1,\ldots, T_n) +q(T_1,\ldots, T_n)^*
$$
is completely positive.

\end{enumerate}
\end{theorem}
\begin{proof} The implication $(i)\implies (ii)$ and  $(i)\implies
(iii)$  are due to  Theorem  \ref{Poisson-C*}. Since the implication
$(iii)\implies (ii)$ follows from the theory of completely positive
 (resp. contractive) maps, it remains to prove the implication
$(ii)\implies (i)$. To this end, assume that the map
$q(M_{Z_1},\ldots, M_{Z_n})\mapsto q(T_1,\ldots, T_n)$ is completely
contractive.  For each $j=1,\ldots,n$, assume that $f_j$ has the
representation $\sum_{\alpha\in \FF_n^+} c_\alpha^{(j)} Z_\alpha$
and let $q_m^{(j)}:=\sum_{k=0}^m\sum_{|\alpha|=k}c_\alpha^{(j)}
Z_\alpha$, $m\in \NN$. Since the universal  model
$M_Z:=(M_{Z_1},\ldots, M_{Z_n})$ is in the set  of norm
convergence for the $n$-tuple $f$, we have $f_j(M_{Z }
)=\lim_{m\to\infty} q_m^{(j)}(M_{Z } )$ with the convergence in the
operator norm topology. On the other hand, due to Theorem
\ref{Poisson-C*}, we have
$$
\left\|q_m^{(j)}(T_1,\ldots, T_n)-q_k^{(j)}(T_1,\ldots,
T_n)\right\|\leq \left\|q_m^{(j)}(M_{Z_1},\ldots,
M_{Z_n})-q_k^{(j)}(M_{Z_1},\ldots, M_{Z_n})\right\|
$$
for any $m,k\in \NN$. Consequently, $\{q_m^{(j)}(T_1,\ldots,
T_n)\}_{m=1}^\infty$ is a Cauchy sequence in $B(\cH)$ and,
therefore,
$f_j(T_1,\ldots,T_n):=\lim_{m\to\infty}q_m^{(j)}(T_1,\ldots, T_n)$
exists in the operator norm. Now, since
$$
\left\|[q_m^{(1)}(T_1,\ldots, T_n), \ldots, q_m^{(n)}(T_1,\ldots,
T_n)]\right\|\leq \left\|[q_m^{(1)}(M_{Z_1},\ldots, M_{Z_n}),
\ldots, q_m^{(n)}(M_{Z_1},\ldots, M_{Z_n})]\right\|,
$$
taking the limit as $m\to\infty$, we obtain $\|f(T)\|\leq
\|f(M_Z)\|\leq 1$. Since $f=(f_1,\ldots, f_n)$ has  the radial
approximation property, relation \eqref{gi}  and Lemma \ref{rad}
show that  the  sequence
$p_m^{(i)}:=\sum_{k=0}^m\sum_{|\alpha|=k}a_\alpha^{(j)} Z_\alpha$ of
noncommutative polynomials satisfies the relation
$$M_{Z_i}=g_i(f(M_Z))=\lim_{m\to\infty}p_m^{(i)}(f(M_Z)),
$$
where the limit is in the operator norm.  Therefore, we have
$\|p_m^{(i)}(f(M_Z))-M_{Z_i}\|\to
 0$ as $m\to\infty$. Using  the von Neumann type inequality
 $$
\|p_m^{(i)}(f(T))-T_i\|\leq \|p_m^{(i)}(f(M_Z))-M_{Z_i}\|, \quad
m\in \NN,
$$
we deduce that $T_i=\lim_{m\to\infty}p_m^{(i)}(f(T)) $ in the
operator norm and, therefore, $g_i(f(T))=T_i$ for  all  $i=1,\ldots,
n$.
 This shows that $(T_1,\ldots, T_n)\in \BB_f(\cH)$
and completes the proof.
\end{proof}

We introduce the  noncommutative  domain algebra $\cA(\BB_f)$ as the
norm closure of all polynomials in $M_{Z_1},\ldots, M_{Z_n}$ and the
identity.

\begin{theorem} Let $f=(f_1,\ldots, f_n)\in \cM_{rad}\cap \cM^{||}$ and $(A_1,\dots,A_n)\in B(\cH)^n$. Then
there is an $n$-tuple of operators $(T_1,\dots,T_n)\in \BB_f(\cH)$
and an invertible operator $X$ such that
$$
A_i=X^{-1}T_i X,\text{\quad for any\ }i=1,\dots,n,
$$
 if and only if
the $n$-tuple $(A_1,\dots,A_n)$ is  completely polynomially bounded
with respect to the noncommutative  domain algebra $\cA(\BB_f)$.
\end{theorem}
\begin{proof}
Using Theorem \ref{cc} and Paulsen's similarity result \cite{Pa},
the result follows.
\end{proof}

\begin{lemma} \label{ball} Let  $f=(f_1,\ldots, f_n)$ be   an $n$-tuple  of
 formal power series in the class  $\cM^{||}$, and let $g=(g_1,\ldots, g_n)$ be
  the inverse of $f$.
   Then the following statements hold.
  \begin{enumerate}
  \item[(i)] The set $\BB_{f}^{<}(\cH)$ coincides with
  $g\left([B(\cH)^n]_1\right)$. When $\cH=\CC$, the result holds true  when $f$ has only the model property.
  \item[(ii)] The set $\BB^{pure}_f(\cH)$  coincides with the image of all pure row contractions under
  $g$.
  \item[(iii)]  If  $f(0)=0$, then $\BB_f^<(\cH)$ contains an open ball in $B(\cH)^n$ centered at $0$, and
 $$
 \{X\in B(\cH)^n: \  X \text{ is nilpotent   and } \|f(X)\|\leq 1\}= \BB^{nil}_f(\cH)\subset \BB^{pure}_f(\cH).
$$
      \end{enumerate}

\end{lemma}
\begin{proof} We shall prove items (i) and (ii) when $f$ is an $n$-tuple of
formal power series  with property $(\cS)$. The other two cases
(when $f$ has property $(\cA)$ or property $(\cF)$) can be treated
similarly.
 First, note that  $\BB_{f}^<(\cH)\subseteq g([B(\cH)^n]_1)$. To prove the reversed
 inclusion let $Y=g(X)$, where $X\in [B(\cH)^n]_1$. According to Lemma \ref{M2} part (iii), we have
  either
 $\widetilde{g}\in \cC_f^{SOT}(\HH^2(f))$  and
 \begin{equation}
 \label{sifi}
 S_i=f_i(\widetilde{g}_1,\ldots, \widetilde{g}_n),\qquad i=1,\ldots,n,
 \end{equation}
 or
 $\widetilde{g}\in \cC_f^{rad}(\HH^2(f))$  and
 \begin{equation}
 \label{sifi2}
 S_i= \text{\rm SOT-}\lim_{r\to 1} f_j(r\widetilde{g}_1,\ldots, r\widetilde{g}_n),\qquad i=1,\ldots,n.
 \end{equation}
 Since $f\in \cM^{||}$, the $n$-tuple $(M_{Z_1},\ldots,
 M_{Z_n})$ is in the set of norm-convergence (or radial
 norm-convergence) for the $n$-tuple of formal power series
 $f=(f_1,\ldots, f_n)$. This  implies that the convergence above is
 in the operator topology.
  Applying the noncommutative Poisson transform $P_X$, we deduce that
 $X_i=f_i(g_1(X),\ldots, g_n(X))$, $i=1,\ldots,n$. This implies  that
 $f(Y)=f(g(X))=X$ and $g(f(Y))=g(X)=Y$, which  shows that $Y\in \BB_{f}^<(\cH)$.
Therefore, $\BB^<_{f}(\cH)= g([B(\cH)^n]_1)$, the function $g$ is one-to-one on
$[B(\cH)^n]_1$ and $f$ is its inverse on  $\BB^<_{f}(\cH)$.
Now consider the case when $\cH=\CC$ and assume that $f$ has  the model property.
Since $\BB_{f}^<(\CC)\subseteq g(\BB_n)$, we prove the reverse inclusion.
Let $\mu=g(\lambda)$ for some $\lambda\in \BB_n$ and assume that one of the relations \eqref{sifi} or \eqref{sifi2} holds, say the latter. Setting $z_\lambda:=\sum_{\alpha\in \FF_n^+} \overline{\lambda}_\alpha e_\alpha\in F^2(H_n)$, we deduce that
\begin{equation*}
\begin{split}
\lambda_j&=\left<S_j(1), z_\lambda\right>=\lim_{r\to 1} \left< f_j(r\widetilde{g}_1,\ldots, r\widetilde{g}_n)(1), z_\lambda\right>\\
&=\lim_{r\to 1} f_j(rg_1(\lambda),\ldots, rg_n(\lambda))=f_j(g(\lambda)).
\end{split}
\end{equation*}
This implies  that
 $f(\mu)=f(g(\lambda))=\lambda$ and $g(f(\mu))=g(\lambda)=\mu$, which  shows that $\mu\in \BB_{f}^<(\CC)$.
Therefore, $\BB^<_{f}(\CC)= g(\BB_n)$, the function $g$ is one-to-one on
$\BB_n$ and $f$ is its inverse on  $\BB^<_{f}(\BB_n)$.  Similarly, one can assume that relation \eqref{sifi} holds and reach  the same conclusion.

To prove  item (ii),  set $[B(\cH)^n]_1^{pure}:=\{ X\in [B(\cH)^n]_1^-:\ X \text{ is  a pure row contraction}\}$
and note that $\BB^{pure}_f(\cH)\subseteq \{g(X): \ X\in [B(\cH)^n]_1^{pure} \}$. The reversed inclusion follows similarly to the proof of item (i) using the noncommutative Poison transform $P_X$, where $X$ is a pure row contraction. In this case, we also show  that
 $f(g(X))=X$  and deduce that $ g:[B(\cH)^n]_1^{pure}\to
\BB^{pure}_f(\cH) $ is a bijection with inverse
$f:\BB^{pure}_f(\cH)\to [B(\cH)^n]_1^{pure}$.
Now we prove part (iii).
 Since $f$ has nonzero radius of convergence and $f(0)=0$, the Schwartz lemma for free holomorphic functions implies that there is $\gamma>0$ such that
$\|f(X)\|<1$ for any $X\in [B(\cH)^n]_\gamma$. On the other hand,
using Theorem 1.2 from \cite{Po-automorphism}, the composition
$g\circ f$ is a free holomorphic function on $[B(\cH)^n]_\gamma$.
Due to the uniqueness theorem for free holomorphic functions  and
the fact  that $g\circ f=id$,  as formal power series, we deduce
that $g(f(X))=X$ for any $X\in [B(\cH)^n]_\gamma$.

If   $X\in B(\cH)^n$  is  a nilpotent  $n$-tuple  with $\|f(X)\|\leq 1$, then taking into account that  $f(0)=0$, we deduce that
$[f_1(X),\ldots, f_n(X)]$ is a nilpotent $n$-tuple. Hence and using that $g\circ f=id$, we deduce that $g(f(X))=X$, which completes the proof.
\end{proof}

\begin{lemma}\label{gb} If $f=(f_1,\ldots, f_n)$ is   an $n$-tuple of formal power
 series in the class $\cM_{rad}^{||}$, then
 $$\BB_f(\cH)=g\left([B(\cH)^n]_1^-\right),
 $$
 where $g=(g_1,\ldots,g_n)$ is the inverse   of $f$ with respect to the composition of power series. Moreover, the function  $g:[B(\cH)^n]_1^-\to \BB_f(\cH)$ is a bijection with inverse
 $f:\BB_f(\cH)\to [B(\cH)^n]_1^-$.  When $\cH=\CC$, the result holds true  when $f$ has only the radial approximation  property.
\end{lemma}
\begin{proof}
First, note that $\BB_f(\cH)\subseteq g\left([B(\cH)^n]_1^-\right)$.
To prove the reverse inclusion, let $Y:=g(X)$ and
$X=(X_1,\ldots,X_n)\in [B(\cH)^n]_1^-$. Since $f$   has the
 radial approximation property, $g=\sum_{\alpha\in \FF_n^+} a_\alpha^{(i)} Z_\alpha$
 is a free holomorphic function on $[B(\cH)^n]_\gamma$ for some $\gamma>1$. Moreover,
  according to Lemma \ref{rad},
 there is \ $\delta\in (0,1)$ with the property   that   for any $r\in (\delta,1]$, the series  $g_i(\frac{1}{r}S):=\sum_{k=0}^\infty\sum_{|\alpha|=k} \frac{a_\alpha^{(i)}}{r^{|\alpha|}} S_\alpha$ is convergent in the operator norm topology and represents   an element
    in  the noncommutative disc algebra $\cA_n$,
and
  \begin{equation}
  \label{rad-sifi}
  \frac{1}{r} S_j=f_j\left(g_1\left(\frac{1}{r} S\right),\ldots,
  g_n\left(\frac{1}{r} S\right)\right),\qquad j\in\{1,\ldots, n\}, \ r\in (\delta,
  1],
  \end{equation}
where  $g(\frac{1}{r} S)$ is in the norm-convergence (or radial
norm-convergence) of $f$. Applying  now the noncommutative Poisson
transform $P_{rX}$,   we deduce that $X_j=f_j(g(X))$ for
$j=1,\ldots,n$. This also shows that $g$ is one-to-one on
$[B(\cH)^n]_1^-$. On the other hand, the relation above implies
$Y=g(X)=g(f(g(X)))=g(f(Y))$ and $\|f(Y)\|\leq 1$, which shows that
$Y\in \BB_f(\cH)$. Therefore,
$\BB_f(\cH)=g\left([B(\cH)^n]_1^-\right)$ and $f$ is one-to-one on
$\BB_f(\cH)$.

Now consider the case when $\cH=\CC$ and assume that $f$ has only the radial approximation property.
Since $\BB_{f}(\CC)\subseteq g(\overline{\BB}_n)$, we prove the reverse inclusion.
Let $\mu=g(\lambda)$ for some $\lambda\in \overline{\BB}_n$ and assume that  relation \eqref{rad-sifi} holds, where  $g(\frac{1}{r} S)$ is either in the convergence set \
$\cC_f^{SOT}(F^2(H_n))$ or $\cC_f^{rad}(F^2(H_n))$.  For example, assume  that $g(\frac{1}{r} S)\in \cC_f^{SOT}(F^2(H_n))$.
For each $r\in (\delta, 1)$, consider
 $z_{r\lambda}:=\sum_{\alpha\in \FF_n^+} \overline{\lambda}_\alpha r^{|\alpha|} e_\alpha\in F^2(H_n)$, and note  that
\begin{equation*}
\begin{split}
\lambda_j&=\left<\frac{1}{r}S_j(1), z_{r\lambda}\right>= \left< f_j\left({g_1(\frac{1}{r} S)},\ldots, {g}_n(\frac{1}{r}S)\right)(1), z_{r\lambda}\right>\\
&= f_j(g_1(\lambda),\ldots, g_n(\lambda))=f_j(g(\lambda)).
\end{split}
\end{equation*}
This implies  that
 $f(\mu)=f(g(\lambda))=\lambda$ and $g(f(\mu))=g(\lambda)=\mu$, which  shows that $\mu\in \BB_{f}(\CC)$.
Therefore, $\BB_{f}(\CC)= g(\overline{\BB}_n)$, the function $g$ is one-to-one on
$\overline{\BB}_n$ and $f$ is its inverse on  $\BB_{f}(\CC)$.  Similarly, one can treat the case when  $g(\frac{1}{r} S)\in \cC_f^{rad}(F^2(H_n))$..
The proof is complete.
\end{proof}

In what follows,  we identify the characters of the noncommutative
  domain
algebra $\cA(\BB_f)$. Let $\lambda=(\lambda_1,\dots,\lambda_n)$ be
in $\BB_f(\CC)$
  and define the
evaluation functional
$$
\Phi_\lambda:\cP(M_{Z_1},\ldots, M_{Z_n})\to
\CC,\quad\Phi_\lambda(p(M_{Z}))=p(\lambda),
$$
where $\cP(M_{Z_1},\ldots, M_{Z_n})$  denotes the algebra of all
polynomials in $M_{Z_1},\ldots, M_{Z_n}$ and the identity.
According to  Theorem \ref{Poisson-C*},   we have
$
|p(\lambda)|=\|p(\lambda I_{\CC})\|\leq
\|p(M_{Z})\|.
$
Hence, $\Phi_\lambda$ has a unique extension to the domain algebra
$\cA(\BB_f)$. Therefore
$\Phi_\lambda$ is a character of $\cA(\BB_f)$.
\begin{theorem} Let $f=(f_1,\ldots, f_n)$ be an  $n$-tuple of formal power series with the radial approximation property.
 and
let $M_{\cA(\BB_f)}$ be the set of all characters of $\cA(\BB_f)$.
Then the map
$$\Psi: \BB_f(\CC)\to M_{\cA(\BB_f)},\quad
\Psi(\lambda):=\Phi_\lambda, $$
 is a homeomorphism   and $\BB_f(\CC)$ is homeomorphic to the closed
 unit
ball $\overline{\BB}_n$.
\end{theorem}

\begin{proof} First, notice  that $\Psi$ is injective.
To prove that $\Psi$ is surjective, assume that $\Phi:\cA(\BB_f)\to\CC$ is a character. Setting
$\lambda_i:=\Phi(M_{Z_i})$,\ $i=1,\ldots,n$, we  deduce that
$
\Phi(p(M_{Z}))=p(\lambda)
$
for any polynomial $p(M_{Z_1},\ldots,M_{Z_n})$ in $\cA(\BB_f)$.
Since $\Phi$ is a character it follows that it is completely
contractive. Applying Theorem \ref{cc} in the particular case
when
$A_i:=\lambda_i I_{\CC}$,\ $i=1,\ldots,n$,  it follows that
$(\lambda_1 I_{\Bbb{C}},\ldots,\lambda_n I_{\CC})\in \BB_f(\CC)$.
Moreover, since
$$
\Phi(p(M_{Z}))=p(\lambda)=
\Phi_\lambda(p(M_{Z}))
$$
for any polynomial $p(M_{Z})$ in $\cA(\BB_f)$,
   we must have $\Phi=\Phi_\lambda$.
Suppose now that
$\lambda^\alpha:=(\lambda_1^\alpha,\dots,\lambda_n^\alpha);\
\alpha\in J$, is
a net in $\BB_f(\CC)$ such that
$\lim_{\alpha\in
J}\lambda^\alpha=\lambda:=(\lambda_1,\dots,\lambda_n)$.
It is clear  that
$$
\lim_{\alpha\in J}\Phi_{\lambda^\alpha}(p(M_{Z}))=
\lim_{\alpha\in J} p(\lambda^\alpha)=
p(\lambda)=\Phi_\lambda(p(M_{Z}))
$$
for every  polynomial  $p(M_{Z})$. Since the set of all
polynomials $\cP(M_{Z_1},\ldots,M_{Z_n})$  is  dense
in $\cA(\BB_f)$  and $\sup_{\alpha\in J}
\|\Phi_{\lambda^\alpha}\|\leq1$,
 it follows that $\Psi$
is continuous. According to Lemma \ref{gb},
$\BB_f(\CC)=g(\overline{\BB}_n)$ is a  compact subset of $\CC^n$ and
$g:\overline{\BB}_n\to \BB_f(\CC)$  is  a bijection. Since both $
\BB_f(\CC)$ and $M_{\cA(\BB_f)}$ are compact Hausdorff
spaces and $\Psi$ is also one-to-one and onto,   we deduce that
$\Psi$  is a homeomorphism. On the other hand, since the map
$\lambda\mapsto g(\lambda)$ is holomorphic on a ball
$(\CC^n)_\gamma$  for some  $\gamma>1$, one can see that
$\BB_f(\CC)$ is homeomorphic to the closed
 unit
ball $\overline{\BB}_n$. The proof is complete.
\end{proof}

\bigskip

\section{ The invariant subspaces under $M_{Z_1},\ldots, M_{Z_n}$}
\label{inv-sub}

In this section we obtain a Beurling type characterization of the
joint invariant subspaces under the    multiplication operators
$M_{Z_1},\ldots, M_{Z_n}$ associated with the noncommutative domain
$\BB_f$ and a minimal dilation theorem for pure $n$-tuples of
operators in $\BB_f(\cH)$.

  An operator
$A:\HH^2(f)\otimes \cH\to \HH^2(f)\otimes \cK$ is called
multi-analytic  with respect to $M_{Z_1},\ldots, M_{Z_n}$ if
$A(M_{Z_i}\otimes I_\cH)=(M_{Z_i}\otimes I_\cK)A$ \, for any
$i=1,\ldots,n$. If, in addition,  $A$ is a partial isometry, we call
it inner.

\begin{theorem}\label{Beurling}
Let $f=(f_1,\ldots, f_n)$ be  $n$-tuple of formal power series with
the model
 property
 and let $M_Z:=(M_{Z_1},\ldots, M_{Z_n})$ be the universal model associated with $\BB_f$.
 If $Y\in
B(\HH^2(f)\otimes \cH)$, then the following statements are
equivalent.
\begin{enumerate}
\item[(i)]
There is a Hilbert space $\cE$ and a multi-analytic operator
$\Psi:\HH^2(f)\otimes \cE\to \HH^2(f)\otimes \cH$ with respect to
the     multiplication  operators $M_{Z_1},\ldots, M_{Z_n}$ such
that $Y=\Psi \Psi^*$.
\item[(ii)]
$\Phi_{f,M_Z\otimes I}(Y)\leq Y$, where the positive linear mapping $\Phi_{f,M_Z\otimes
I}:B(\HH^2(f)\otimes\cH)\to B(\HH^2(f)\otimes \cH)$  is defined by
$$
\Phi_{f,M_Z\otimes I}(Y):=\sum\limits_{i=1}^n (f_i(M_Z)\otimes
I_\cH) Y(f_i(M_Z)\otimes I_\cH)^*.
$$
\end{enumerate}
\end{theorem}

\begin{proof}
First, assume that condition (ii) holds and note that
$Y-\Phi_{f,M_Z\otimes I}^m(Y)\geq 0$ for any $m=1,2,\ldots$. Since
$(M_{Z_1},\ldots, M_{Z_n})$ is a pure $n$-tuple  with respect to the
noncommutative domain $\BB_f(\HH^2(f))$, we deduce that
SOT-$\lim_{m\to\infty}\Phi_{f,M_Z\otimes I}^m(Y)=0$, which implies
$Y\geq0$. Denote $\cM:=\overline{\text{\rm range} Y^{1/2}}$ and
define
\begin{equation}\label{ai}
Q_i(Y^{1/2} x):=Y^{1/2} (f_i(M_Z)^*\otimes I_\cH)x,\qquad x\in
\HH^2(f)\otimes \cH,
\end{equation}
 for any $i=1,\ldots, n$.  We have
\begin{equation*}
\begin{split}
\sum_{i=1}^n \|Q_i(Y^{1/2}x)\|^2&\leq \sum_{i=1}^n \|Y^{1/2}
(f_i(M_Z)^*\otimes I_\cH)x\|^2\\
&=\left< \Phi_{f,M_Z\otimes I}(Y)x,x\right>\leq \|Y^{1/2} x\|^2
\end{split}
\end{equation*}
for any $x\in \HH^2(f)\otimes \cH$, which implies $ \|Q_iY^{1/2}
x\|^2\leq \|Y^{1/2} x\|^2$, for any $x\in \HH^2(f)\otimes \cH$.
 Consequently, $Q_i$ can be uniquely be extended to a
bounded operator (also denoted by $Q_i$) on the subspace $\cM$.
Setting $A_i:=Q_i^*$, $i=1,\ldots, n$, we deduce that $ \sum_{i=1}^n
A_iA_i^*\leq I_\cM.$ Denoting $\varphi_A(Y):=\sum_{i=1}^n A_iYA_i^*$
and using relation \eqref{ai}, we have
\begin{equation*}
\begin{split}
\left<\varphi_A^m(I) Y^{1/2}x, Y^{1/2} x\right> &= \left< \Phi_{f,
M_Z\otimes I}^m(Y)x,  x\right>\\
&\leq \|Y\| \left< \Phi_{f, M_Z\otimes I}^m(I)x,  x\right>
\end{split}
\end{equation*}
for any $x\in \HH^2(f)\otimes \cH$. Since \
SOT-$\lim\limits_{m\to\infty}\Phi_{f,M_Z\otimes I}^m(I)=0$, we have
\ SOT-$\lim\limits_{m\to\infty}\varphi_A^m(I)=0$. Therefore
$A:=(A_1,\ldots, A_n)$ is a pure row contraction. According to
\cite{Po-poisson},
 the Poisson kernel $K_{A}:\cM\to
\HH^2(f)\otimes \cE$ ($\cE$ is an appropriate Hilbert space) defined
by $$K_Ah:= \sum_{\alpha\in \FF_n^+} f_\alpha\otimes
\Delta_AA_\alpha^*h, \qquad h\in \cM, $$ where
$\Delta_A:=(I-A_1A_1^*-\ldots, A_nA_n^*)^{1/2}$ is an isometry with
the property that
\begin{equation}
\label{int-KT} A_iK_{A}^*=K_{A}^* (M_{f_i}\otimes I_\cE),\qquad
i=1,\ldots,n.
\end{equation}
Let $\Gamma:=Y^{1/2} K_{A}^*:\HH^2(f)\otimes \cE\to \HH^2(f)\otimes
\cH$ and note that, due to the fact that $f$ has the model property,
  $M_{f_i}=f_i(M_{Z_1},\ldots,
M_{Z_n})$ for $i=1,\ldots,n$. Consequently, we have
\begin{equation*}
\begin{split}
\Gamma(M_{f_i}\otimes I_\cE)&=Y^{1/2} K_{A}^*(M_{f_i}\otimes
I_\cE)=Y^{1/2}
A_i K_{A}^*\\
&=(f_i(M_Z)\otimes I_\cH) Y^{1/2} K_{A}^* =(M_{f_i}\otimes I_\cH)
\Psi
\end{split}
\end{equation*}
for any $i=1,\ldots, n$. Now, let $g=(g_1,\ldots,g_n)$ be the
inverse of $f=(f_1,\ldots, f_n)$ with respect to the composition  of
power series. In the proof of Theorem \ref{model}, we showed that
  $g_i(M_{f_1},\ldots, M_{f_n})=M_{Z_i}$ for all $i=1,\ldots,n$. Hence, we deduce that the
operator $M_{Z_i}$ is in the SOT-closure of all polynomials in
$M_{f_1},\ldots, M_{f_n}$ and the identity. Consequently, the
relation $\Gamma(M_{f_i}\otimes I_\cE)=(M_{f_i}\otimes I_\cH)\Gamma$
implies $\Gamma (M_{Z_i}\otimes I_\cH)=(M_{Z_i}\otimes I_\cH)
\Gamma$  \ for $i=1,\ldots, n$, which shows  that $\Gamma$ is a
multi-analytic  with respect to $M_{Z_1},\ldots, M_{Z_n}$. Note that
we  also  have $\Gamma \Gamma^*=Y^{1/2} K_{A}^* K_{A} Y^{1/2} =Y$.
The proof is complete.
\end{proof}

The next result is  a Beurling \cite{Be} type characterization  of
the invariant subspaces  under the    multiplication operators
$M_{Z_1},\ldots, M_{Z_n}$ associated with the noncommutative domain
$\BB_f$.

\begin{theorem} \label{inv-subs} Let $f=(f_1,\ldots, f_n)$ be  an  $n$-tuple of formal power series with
the model property  and  let $(M_{Z_1},\ldots, M_{Z_n})$ be the
multiplication operators associated  with the noncommutative domain
$\BB_f$. A subspace $\cN\subseteq \HH^2(f)\otimes \cH$ is invariant
under each operator $M_{Z_1}\otimes I_\cH,\ldots, M_{Z_n}\otimes
I_\cH$  if and only if there exists an inner multi-analytic operator
$\Psi:\HH^2(f)\otimes \cE\to \HH^2(f)\otimes \cH$ with respect to
$M_{Z_1},\ldots, M_{Z_n}$ such that
$$
\cN=\Psi[\HH^2(f)\otimes \cE].
$$
\end{theorem}
\begin{proof}
Assume that $\cN\subseteq \HH^2(f)\otimes \cH$ is invariant under
each operator $M_{Z_1}\otimes I_\cH,\ldots, M_{Z_n}\otimes I_\cH$.
Since $P_\cN(M_{Z_i}\otimes I_\cH)P_\cN=(M_{Z_i}\otimes I_\cH)P_\cN$
for any $i=1,\ldots, n$, and  $M_Z:=(M_{Z_1},\ldots, M_{Z_n})\in
\BB_f(\HH^2(f))$, we have
\begin{equation*}
\begin{split}
\Phi_{f, M_Z\otimes I_\cH}(P_\cN)&=P_\cN\left[ \sum_{i=1}^n
(f_i(M_{Z})\otimes   I_\cH) P_\cN
(f_i(M_{Z})^*\otimes   I_\cH)\right] P_\cN\\
&\leq P_\cN\left[ \sum_{ i= 1}^n (f_i(M_{Z})\otimes I_\cH)
(f_i(M_{Z})^*\otimes   I_\cH)\right] P_\cN\\
&=P_\cN \left(\sum_{i=1}^n M_{f_i} M_{f_i}^*\otimes I_\cH\right)
P_\cN\leq P_\cN.
\end{split}
\end{equation*}
Here, we also used the fact  that $M_{f_i}=f_i(M_{Z_1},\ldots,
M_{Z_n})$.  Applying now  Theorem \ref{Beurling}, we find a
multi-analytic operator $\Psi:\HH^2(f)\otimes \cE\to \HH^2(f)\otimes
\cH$ with respect to the   operators $M_{Z_1},\ldots, M_{Z_n}$ such
that $P_\cN=\Psi \Psi^*$. Since $P_\cN$ is an orthogonal projection,
we deduce that $\Psi$ is a partial isometry and
$\cN=\Psi[\HH^2(f)\otimes \cE]$. Since the converse is obvious, the
proof is complete.
\end{proof}

\begin{theorem}\label{co-inv} Let $f=(f_1,\ldots, f_n)$ be  an  $n$-tuple of formal power series with
the model property  and  let $(M_{Z_1},\ldots, M_{Z_n})$ be the
universal model   associated  with the noncommutative domain
$\BB_f$. If $\cN\subseteq \HH^2(f)\otimes \cH$ is a coinvariant
subspace under $M_{Z_1}\otimes I_{\cH},\ldots, M_{Z_n}\otimes
I_{\cH}$, then there is  a subspace $\cE\subseteq \cH$ such that
$$
\overline{\text{\rm span}}\,\left\{(M_{Z_\alpha}\otimes I_\cH)\cN:\
\alpha\in \FF_n^+\right\}=\HH^2(f)\otimes \cE.
$$
In particular, $\cN$ is cyclic for  the operators $M_{Z_1}\otimes
I_{\cH},\ldots, M_{Z_n}\otimes I_{\cH}$ if and only if \
$(P_\CC\otimes I_\cH)\cN=\cH$, where $P_\CC$ is the orthogonal
projection on $\CC$.
\end{theorem}
\begin{proof}
Let $\cE:=(P_\CC\otimes I_\cH)\cN\subset \cH$, where $1\otimes \cH$
is identified with $\cH$, and let $h\in \cN$ be a nonzero vector
with representation $h=\sum_{\alpha\in \FF_n^+} f_\alpha \otimes
h_\alpha$, $h_\alpha\in \cH$. Choose $\beta\in \FF_n^+$ with
$h_\beta\neq 0$. Since  $\cN$ is a co-invariant subspace under
$M_{Z_1}\otimes I_{\cH},\ldots, M_{Z_n}\otimes I_{\cH}$, and $f$ has
the model property, we have $M_{f_i}=f_i(M_{Z_1},\ldots, M_{Z_n})$
for $i=1,\ldots,n$,  and  deduce that
$$
(P_\CC\otimes I_\cH)([f(M_Z)]_\alpha\otimes I_\cH)h=(P_\CC M_{f_\alpha}\otimes I_\cH)h=h_\beta\in \cE.
$$
This implies $(M_{f_\beta}\otimes I_\cH)(1\otimes
h_\beta)=f_\beta\otimes h_\beta\in \HH^2(f)\otimes \cE$ for any
$\beta\in \FF_n^+$. Hence, we deduce that $h=\sum_{\alpha\in
\FF_n^+} f_\alpha\otimes h_\alpha\in \HH^2(f)\otimes \cE$.
Therefore, $\cN\subset \HH^2(f)\otimes \cE$, which implies
$$
\cG:=\overline{\text{\rm span}}\,\left\{(M_\alpha\otimes I_\cH)\cN:\
\alpha\in \FF_n^+\right\}\subseteq\HH^2(f)\otimes \cE.
$$
Now, we prove the reverse inclusion. Let $h_0\in \cE$, $h_0\neq 0$. Due to the definition of the subspace $\cE$, there exists $x\in \cM$ such that $x=1\otimes h_0+
\sum_{|\alpha|\geq 1}f_\alpha\otimes h_\alpha$. Hence, we obtain
$$
h_0=(P_\CC\otimes I_\cH)x=\left(I-\sum_{i=1}^n M_{f_i}M_{f_i}^*\otimes I_\cH\right)x.
$$
Since $M_{f_i}$ is a SOT-limit of polynomials in $M_{Z_1},\ldots,
M_{Z_n}$, and $\cN$   is a co-invariant subspace under
$M_{Z_1}\otimes I_{\cH},\ldots, M_{Z_n}\otimes I_{\cH}$, we deduce
that $h_0\in \cG$. Therefore, $\cE\subset \cG$ and
$(M_{Z_\alpha}\otimes I_\cH)(1\otimes \cE)\subset \cG$ for
$\alpha\in \FF_n^+$. Since, due to Proposition \ref{dense},
$\CC[Z_1,\ldots, Z_n]$ is dense in $\HH^2(f)$, we deduce that
$\HH^2(f)\otimes \cE\subseteq \cG$.  The last part of the theorem is
now obvious. The proof is complete.
\end{proof}

A simple consequence of Theorem \ref{co-inv} is the following result.

\begin{corollary}\label{co-inv2}
A subspace  $\cN\subseteq \HH^2(f)\otimes \cH$ is reducing  under
each operator  $M_{Z_i}\otimes I_{\cH}$, $i=1,\ldots,n$,   if and
only if  there is a subspace $\cE\subseteq \cH$ such that
$\cN=\HH^2(f)\otimes \cE$.
\end{corollary}

We remark that, in Theorem \ref{inv-subs}, the inner multi-analytic
operator $\Psi:\HH^2(f)\otimes \cE\to \HH^2(f)\otimes \cH$ with
respect to $M_{Z_1},\ldots, M_{Z_n}$ and  with the property  that $
\cN=\Psi[\HH^2(f)\otimes \cE] $ can be chosen to be an isometry.
Indeed, let $\cM:=\{x\in \HH^2(f)\otimes \cE: \
\|\Psi(x)\|=\|x\|\}$. Since $f$ has the model property, we deduce
that $f_i(M_{Z_1},\ldots, M_{Z_n})=M_{f_i}$ is an isometry for each
$i=1,\ldots,n$. Consequently, we have
\begin{equation*}
\begin{split}
\|\Psi (M_{f_i}\otimes I_\cE)x\|&=\|\Psi f_i(M_{Z_1}\otimes I_\cE,\ldots, M_{Z_n}\otimes I_\cE) x\|
=\|f_i(M_{Z_1}\otimes I_\cH,\ldots, M_{Z_n}\otimes I_\cH)\Psi(x)\|\\
&=\|\Psi(x)\|=\|x\|=\|f_i(M_{Z_1}\otimes I_\cE,\ldots, M_{Z_n}\otimes I_\cE) x\|=\|(M_{f_i}\otimes I_\cE)x\|
\end{split}
\end{equation*}
for any $x\in \cM$ and $i=1,\ldots,n$. This implies that $\cM$ is an invariant subspace under $M_{f_i}\otimes I_\cE$, $i=1,\ldots,n$. Using the fact that $M_{Z_i}=g_i(M_{f_1},\ldots, M_{f_n})$, $i=1,\ldots,n$, where $g=(g_1,\ldots, g_n)$ is the inverse of $f=(f_1,\ldots, f_n)$, we deduce that $\cM$ is invariant
under $M_{Z_1}\otimes I_\cE,\ldots, M_{Z_n}\otimes I_\cE$. On the other hand, since $\cM^\perp=\ker \Psi$ and $\Psi(M_{Z_i}\otimes I_\cE)=(M_{Z_i}\otimes I_\cH)\Psi$,
it is clear that $\cM^\perp$ is also invariant under $M_{Z_1}\otimes I_\cE,\ldots, M_{Z_n}\otimes I_\cE$, which shows that $\cM$ is a reducing subspace for $M_{Z_1}\otimes I_\cE,\ldots, M_{Z_n}\otimes I_\cE$. Now, due to Corollary \ref{co-inv2}, $\cM=\HH^2(f)\otimes \cG$ for some subspace $\cG\subseteq \cE$.
Therefore, we have
$$
\cN=\Psi[\HH^2(f)\otimes \cE]=\Psi(\cM)=\Psi[\HH^2(f)\otimes \cG]
$$
and the restriction of $\Psi$ to $\HH^2(f)\otimes \cG$ is  an isometric multi-analytic operator, which proves our assertion.

The next result can be viewed as a continuation of Theorem
\ref{model}.

\begin{theorem}\label{pure-dilation} Let $T=(T_1,\ldots, T_n)\in
\BB_f(\cH)$ be a pure $n$-tuple of operators and let $f=(f_1,\ldots,
f_n)$ have the model theory. Then the noncommutative Poisson kernel
$K_{f,T}:\cH\to \HH^2(f)\otimes \cD_{f,T}$ defined by relation
\eqref{KfT} is an isometry, the subspace $K_{f,T}(\cH)$ is
co-invariant under $M_{Z_1}\otimes I_{\cH},\ldots, M_{Z_n}\otimes
I_{\cH}$, and
\begin{equation*}
  T_i=K_{f,T}^* (M_{Z_i}\otimes
I_{\cD_{f,T}})K_{f,T},\qquad i=1,\ldots,n.
\end{equation*}
Moreover, the dilation  above is minimal, i.e.,
$$
\HH^2(f)\otimes \cD_{f,T}=\bigvee_{\alpha\in
\FF_n^+}(M_{Z_\alpha}\otimes I_{\cD_{f,T}})K_{f,T}(\cH),
$$
and unique up to an isomorphism.
\end{theorem}
\begin{proof}
The first part of the theorem was proved in Theorem \ref{model}. Due
to the definition of the noncommutative Poisson kernel  $K_{f,T}$,
we have $(P_\CC\otimes I_{\cD_{f,T}})K_{f,T}(\cH)=\cD_{f,T}$.
Applying Theorem \ref{co-inv}, we deduce the minimality of the
dilation. To prove the uniqueness, consider another minimal dilation
of $(T_1,\ldots, T_n)$, that is,
\begin{equation}
\label{another} T_i=V^*(M_{Z_i}\otimes I_\cE)V,\qquad i=1,\ldots,n,
\end{equation}
where $V:\cH\to \HH^2(f)\otimes \cE$ is an isometry, $V(\cH)$ is co-invariant under $M_{Z_i}\otimes I_\cE$, $i=1,\ldots,n$, and
$$
\HH^2(f)\otimes \cE=\bigvee_{\alpha\in
\FF_n^+}(M_{Z_\alpha}\otimes I_{\cE})V(\cH).
$$
According to  Theorem \ref{irreducible} and Theorem \ref{model}, we have
$$
 \overline{\text{\rm span}} \{M_{Z_\alpha} M_{Z_\beta}^*:\ \alpha,\beta\in
 \FF_n^+\}=C^*(M_{Z_1},\ldots, M_{Z_n})
 $$
 and
 there is a completely positive linear map $\Phi:C^*(M_{Z_1},\ldots, M_{Z_n})\to B(\cH)$
  such that $\Phi(M_{Z_\alpha}M_{Z_\beta}^*)=T_\alpha T_\beta^*$, $\alpha,\beta\in \FF_n^+$.
 Note that relation \eqref{another} and the fact that $V(\cH)$ is co-invariant under
 $M_{Z_i}\otimes I_\cE$, $i=1,\ldots,n$, imply that
 $$
 \Phi(X)=K_{f,T}^*\pi_1(X)K_{f,T}=V^*\pi_2(X)V,\qquad X\in C^*(M_{Z_1},\ldots, M_{Z_n}),
 $$
where $\pi_1$,$\pi_2$ are the $*$-representations of $C^*(M_{Z_1},\ldots, M_{Z_n})$ on
 $\HH^2(f)\otimes \cD_{f,T}$ and $\HH^2(f)\otimes \cE$ given by $\pi_1(X):=X\otimes I_{\cD_{f,T}}$
  and $\pi_2(X):=X\otimes I_{\cE}$, respectively.
Since $\pi_1$, $\pi_2$ are minimal Stinespring dilations of $\Phi$,
due to the uniqueness \cite{St},
 there exists a unitary operator $W:\HH^2(f)\otimes \cD_{f,T} \to \HH^2(f)\otimes \cE$ such that
$$
W(M_{Z_i}\otimes I_{\cD_{f,T}})=(M_{Z_i}\otimes I_{\cE})W,\qquad i=1,\ldots,n,
$$
and $WK_{f,T}=V$. Hence, we also deduce that $ W(M_{Z_i}^*\otimes
I_{\cD_{f,T}})=(M_{Z_i}^*\otimes I_{\cE})W$ for $ i=1,\ldots,n$.
Since, due to Theorem \ref{irreducible}, the $C^*$-algebra
$C^*(M_{Z_1},\ldots, M_{Z_n})$ is irreducible, we must have
$W=I_{\HH^2(f)}\otimes \Gamma$, where $\Gamma\in B(\cD_{f,T},\cE)$
is a unitary operator. Consequently, we have $\dim \cD_{f,T}=\dim
\cE$ and $WK_{f,T}V(\cH)=V(\cH)$, which proves that the two minimal
dilations are unitarily equivalent. The proof is complete.
\end{proof}

\begin{corollary} Let $(M_{Z_1},\ldots, M_{Z_n})$ be the universal model associated with the noncommutative domain
$\BB_f$. The $n$-tuples $(M_{Z_1}\otimes I_\cH,\ldots,
M_{Z_n}\otimes I_\cH) $ and $(M_{Z_1}\otimes I_\cK,\ldots,
M_{Z_n}\otimes I_\cK)$ are unitarily equivalent if and only if $\dim
\cH=\dim \cK$.
\end{corollary}
\begin{proof}

Let $W:\HH^2(f)\otimes \cH\to \HH^2(f)\otimes \cK$ be  a unitary
operator such that $W(M_{Z_i}\otimes  I_\cH)=(M_{Z_i}\otimes
I_\cK)W$  for $i=1,\ldots,n$. Since $W$ is unitary,
 we have $W(M_{Z_i}^*\otimes  I_\cH)=(M_{Z_i}^*\otimes  I_\cK)W$, $i=1,\ldots,n$. Using
 the fact that $C^*(M_{Z_1},\ldots, M_{Z_n})$ is ireducible, we deduce that
 $W=I_{\HH^2(f)}\otimes \Gamma$ for a unitary operator $\Gamma\in B(\cH,\cK)$, which shows
  that $\dim\cH=\dim \cK$. The converse is obvious, so the proof is complete.
\end{proof}

\bigskip

\section{
The Hardy algebra $H^\infty(\BB_f)$ and the eigenvectors of
$M_{Z_1}^*,\ldots, M_{Z_n}^*$ }

Let $f=(f_1,\ldots, f_n)$ be an $n$-tuple with the model property.
 We define the noncommutative Hardy algebra $H^\infty(\BB_f)$ to be the WOT-closure of all
  noncommutative polynomials in $M_{Z_1},\ldots, M_{Z_n}$ and the
  identity.
  Assume that $f\in \cM^{||}$. We say that $F:\BB_f^<(\cH)\to B(\cH)$ is  a free holomorphic  function on $\BB_f^<(\cH)$ if
  there are some coefficients $c_\alpha\in \CC$ such that
 $$
 F(Y)=\sum_{k=0}^\infty\sum_{|\alpha|=k} c_\alpha   [f(Y)]_\alpha,\qquad Y\in \BB_f^<(\cH),
 $$
 where the convergence of the series is in the operator norm topology. Since, according to
 Lemma
 \ref{ball}, we have $\BB_f^<(\cH)=g([B(\cH)^n]_1)$  and $f(g(X))=X$, $X\in [B(\cH)^n]_1$,  the uniqueness of the representation  of $F$ follows
  from the uniqueness
  of the representation of free holomorphic functions on $[B(\cH)^n]_1$.

 \begin{theorem}\label{algebra} Let $f=(f_1,\ldots, f_n)$ be  an $n$-tuple of formal power series
  with the model property and let $\BB_f$ be the corresponding  noncommutative domain. Then the
   following statements hold.
 \begin{enumerate}
 \item[(i)]  $H^\infty(\BB_f)$ coincides with the algebra of bounded left multipliers of $\HH^2(f)$.

 \item[(ii)] If $f\in \cM^{||}$, then $H^\infty(\BB_f)$  can be identified with the algebra $\HH^\infty(\BB^<_f)$ of all bounded free holomorphic
   functions on the noncommutative domain $\BB_f^<(\cH)$,  which coincides with
     $$
     \left\{\varphi\circ f:\BB_f^<(\cH)\to B(\cH): \  \varphi\in H^\infty_{\bf
     ball}\right\}.
 $$
 \item[(iii)] If $\psi\in H^\infty(\BB_f)$, then  there is a unique
 $\varphi=\sum_{\alpha} c_\alpha S_\alpha$ in the noncommutative analytic Toeplitz algebra $F_n^\infty$ such that
 $$
 \psi=\text{\rm SOT-}\lim_{r\to 1}\sum_{k=0}^\infty\sum_{|\alpha|=k} c_\alpha  r^{|\alpha|}
 [f(M_Z)]_\alpha,\qquad c_\alpha\in \CC,
 $$
 where   $M_Z:=(M_{Z_1},\ldots, M_{Z_n})$  and  the convergence of the series is in the
  operator norm topology.
 \end{enumerate}

 \end{theorem}
 \begin{proof}
 According to the proof of Lemma \ref{M2}, $M_{f_j}=U^{-1} S_j U$,
 $j=1,\ldots, n$, where $S_1,\ldots, S_n$ are the left creation
 operators on the full Fock space $F^2(H_n)$, and $M_{Z_j}=U^{-1}
 \varphi_j(S_1,\ldots, S_n)U$, where $\varphi_j(S_1,\ldots, S_n)$ is
 in the noncommutative Hardy algebra $F_n^\infty$. We recall that
 $F_n^\infty$ is the WOT   closure of the
 noncommutative polynomials in $S_1,\ldots, S_n$ and the identity.
 Since $H^\infty(\BB_f)$ is the WOT-closure of all noncommutative
 polynomials in $M_{Z_1},\ldots, M_{Z_n}$ and the identity, we
 deduce that $H^\infty(\BB_f)\subseteq U^{-1} F_n^\infty U$. On the
 other hand, using again Lemma \ref{M2}, the creation operator $S_j$
 is in the WOT-closure of polynomials in $\varphi_1(S_1,\ldots,
 S_n),\ldots, \varphi_n(S_1,\ldots, S_n)$ and the identity.
 Consequently, we have $U^{-1} S_jU \in H^\infty(\BB_f)$, $j=1,\ldots, n$, which
 implies $ U^{-1} F_n^\infty U\subseteq H^\infty(\BB_f)$. Thus, we
 have proved that
 \begin{equation}
 \label{FH}
H^\infty(\BB_f)= U^{-1} F_n^\infty U.
\end{equation}
Taking into account that $U(\HH^2(f))=F^2(H_n)$ and that the algebra
of bounded left multipliers on $F^2(H_n)$ coincides with
$F_n^\infty$, we deduce item (i).

To prove (ii), we recall (see \cite{Po-holomorphic}) that if
$\varphi\in H^\infty_{\bf ball}$, then $\varphi(X)=\sum_{k=0}^\infty
\sum_{|\alpha=k} a_\alpha X_\alpha$, $X\in [B(\cH)^n]_1$, where the
convergence is in the operator norm topology. Moreover, $\sup_{X\in
[B(\cH)^n]_1} \|\varphi(X)\|<\infty$, and the model boundary
function $\widetilde \varphi:=\text{\rm SOT-}\lim_{r\to
1}\sum_{k=0}^\infty \sum_{|\alpha=k} a_\alpha  r^{|\alpha|}
S_\alpha$ exists  in $F_n^\infty$. Since, according to Lemma
\ref{ball}, $\BB_f^<(\cH)=g([B(\cH)^n]_1)$ and $f(g(X))=X$ for $X\in
[B(\cH)^n]_1$, the map $F:\BB_f^<(\cH)\to B(\cH)$  defined by
$$
 F(Y)=\sum_{k=0}^\infty\sum_{|\alpha|=k} c_\alpha   [f(Y)]_\alpha,\qquad Y\in \BB_f^<(\cH),
 $$
 is well-defined with the convergence  in the operator norm
 topology. Consequently, $F=\varphi\circ f$ is a bounded  free
 holomorphic function on $\BB_f^<(\cH)$. Now, let $G\in
 \HH^\infty(\BB^<_f)$. Then there are  coefficients  $c_\alpha\in \CC$ such that
$$
 G(Y)=\sum_{k=0}^\infty\sum_{|\alpha|=k} c_\alpha   [f(Y)]_\alpha,\qquad Y\in \BB_f^<(\cH),
 $$
 where  the convergence  is in the   norm
 topology and $\sup_{Y\in \BB_f^<(\cH)}\|G(Y)\|<\infty$.
 Taking $Y=g(rS_1,\ldots, rS_n)$, we deduce that
 $\sup_{r\in [0,1)}\left\|\sum_{k=0}^\infty\sum_{|\alpha|=k} c_\alpha
 r^{|\alpha|} S_\alpha\right\|<\infty$, which shows that the map
 $\varphi:[B(\cH)^n]_1\to B(\cH)$, defined by $\varphi(X)=\sum_{k=0}^\infty\sum_{|\alpha|=k}
 c_\alpha X_\alpha$ is in $H^\infty_{\bf ball}$, and $G=\varphi\circ
 f$. This shows that
 $$
\HH^\infty(\BB^<_f)= \left\{\varphi\circ f:\BB_f^<(\cH)\to B(\cH): \
\varphi\in H^\infty_{\bf
     ball}\right\}.
$$
Hence, using  relation \eqref{FH}  and the fact that  $F_n^\infty$
can be identified with $H^\infty_{\bf ball}$, we deduce item (ii).

To prove part (iii), let $\psi\in H^\infty(\BB_f)$. Due  to relation
\eqref{FH}, the operator $U\psi U^{-1}$ is in the Hardy algebra
$F_n^\infty$ and, therefore,    there are coefficients $c_\alpha\in
\CC$ such that
$$U\psi U^{-1}=\text{\rm SOT-}\lim_{r\to 1}\sum_{k=0}^\infty
\sum_{|\alpha=k} a_\alpha  r^{|\alpha|} S_\alpha,
$$
where the convergence of the series is in norm.  Since $f$ has the
model property, we have  $M_{f_j}=f_j(M_{Z_1},\ldots, M_{Z_n})$,
$j=1,\ldots,n$.  Using now relation $M_{f_j}=U^{-1} S_j U$, we
deduce that item (iii) holds. The proof is complete.
\end{proof}

\begin{theorem}\label{eigenvalues}
Let $f=(f_1,\ldots, f_n)$ be  an  $n$-tuple of formal power series
with the model
 property and  let $M_Z:=(M_{Z_1},\ldots, M_{Z_n})$ be
the universal model associated  with the noncommutative domain
$\BB_f$. The eigenvectors for $M_{Z_1}^*,\ldots, M_{Z_n}^*$ are
precisely the  noncommutative Poisson kernels
$$
\Gamma_\lambda:= \left(1-\sum_{i=1}^n
|f_i(\lambda)|^2\right)^{1/2}\sum_{\alpha\in \FF_n}
[\overline{f(\lambda)}]_\alpha\
  f_\alpha,\qquad \lambda=(\lambda_1,\ldots, \lambda_n)\in \BB_{f}^<(\CC).
$$
They satisfy the equations
$$
M_{Z_i}^* \Gamma_\lambda=\overline{\lambda}_i \Gamma_\lambda,\qquad i=1,\ldots, n.
$$
If $\lambda\in \BB_{f}^<(\CC)$ and $\varphi(M_{Z})$ is in $
H^\infty(\BB_f)$, then
  the
map
$$\Phi_\lambda:H^\infty(\BB_f)\to \CC,\qquad   \Phi_\lambda(\varphi(M_{Z})):=\varphi(\lambda),
$$ is WOT-continuous and multiplicative and
$\varphi(\lambda)=\left< \varphi(M_{Z}) \Gamma_\lambda,
\Gamma_\lambda\right>$. Moreover, $
\varphi(M_{Z})^*\Gamma_\lambda=\overline{\varphi(\lambda)}
\Gamma_\lambda$ and $\lambda\mapsto
\varphi(\lambda)$ is a bounded holomorphic function on
$\BB_{f}^<(\CC)\subset \CC^n$.

\end{theorem}
\begin{proof} Assume that $\lambda=(\lambda_1,\ldots, \lambda_n)\in \BB_{f}^<(\CC)$. According to
Theorem \ref{model},
 the  noncommutative Poisson kernel  associated with the noncommutative domain $\BB_f$  at $\lambda$,
 which is a pure element, is the operator
 $K_{f,\lambda}:\CC\to \HH^2(f)\otimes \CC$ defined by
 $$
 K_{f,\lambda}(z)=\sum_{\alpha\in \FF_n^+} f_\alpha\otimes \left(1-\sum_{i=1}^n
|f_i(\lambda)|^2\right)^{1/2}[\overline{f(\lambda)}]_\alpha z,\qquad
z\in \CC,
$$
which satisfies  the equation $(M_{Z_i}^*\otimes
I_\CC)K_{f,\lambda}=K_{f,\lambda} (\overline{\lambda}_i I_\CC)$\,
for $i=1,\ldots, n$. Under the natural identification of
$\HH^2(f)\otimes \CC$ with $\HH^2(f)$, we deduce that
$\Gamma_\lambda=K_{f,\lambda}$ and
$$
M_{Z_i}^* \Gamma_\lambda=\overline{\lambda}_i \Gamma_\lambda,\qquad
i=1,\ldots, n.
$$

Conversely, let  $\xi:=\sum_{\beta\in \FF_n^+} c_\beta f_\beta(Z)$
be a formal power series in $\HH^2(f)$ such that $\xi\neq 0$  and
assume that $ M_{Z_i}^* \xi=\overline{\lambda}_i \xi$, $ i=1,\ldots,
n$, for some $\lambda=(\lambda_1,\ldots, \lambda_n)\in \CC^n$. Let
$f_i$ have the representation $f_i=\sum_{\alpha\in \FF_n^+}
a_\alpha^{(i)} Z_\alpha$. Since $f=(f_1,\ldots, f_n)$ has the model
property, we have $M_{f_i}=f_i(M_{Z_1},\ldots,
M_{Z_n})=\sum_{k=1}^\infty\sum_{|\alpha|=k}
a_\alpha^{(i)}M_{Z_\alpha}$, where $(M_{Z_1},\ldots, M_{Z_n})$ is
either in the convergence set $\cC_f^{SOT}(\HH^2(f))$ or
$\cC_f^{rad}(\HH^2(f))$. We shall consider just one case since the
other can be treated similarly. For example, assume that
$(M_{Z_1},\ldots, M_{Z_n})\in \cC_f^{SOT}(\HH^2(f))$ and let
$\eta\in \HH^2(f)$. Then we have
\begin{equation*}
\begin{split}
\left<f_i(M_{Z_1},\ldots, M_{Z_n})^*\xi,
\eta\right>&=\lim_{m\to\infty}\left<\xi,
\sum_{k=1}^m\sum_{|\alpha|=k} a_\alpha^{(i)}M_{Z_\alpha}\eta\right>
= \lim_{m\to\infty}\left<\sum_{k=1}^m\sum_{|\alpha|=k}
\overline{a_\alpha^{(i)}}M_{Z_\alpha}^*\xi,
\eta\right>\\
&= \lim_{m\to\infty}\left<\sum_{k=1}^m\sum_{|\alpha|=k}
\overline{a_\alpha^{(i)}}\overline{\lambda}_\alpha \xi, \eta\right>
= \lim_{m\to\infty}\sum_{k=1}^m\sum_{|\alpha|=k}
\overline{a_\alpha^{(i)}}\overline{\lambda}_\alpha\left< \xi,
\eta\right>\\
&=\left<\overline{f_i(\lambda )}\xi, \eta\right>,
\end{split}
\end{equation*}
which shows that
\begin{equation}\label{xi} f_i(M_{Z_1},\ldots,
M_{Z_n})^*\xi= \overline{f_i(\lambda )}\xi,\qquad i=1,\ldots,n.
\end{equation}
Hence, and using the fact that $M_{f_i}=f_i(M_{Z_1},\ldots,
M_{Z_n})$, $i=1,\ldots,n$, we deduce that
\begin{equation*}
\begin{split}
c_\beta&=\left<\xi, M_{f_\beta} 1\right>= \left<\xi,
[f(M_{Z_1},\ldots, M_{Z_n})]_\beta 1\right>\\
&=\left<[f(M_{Z_1},\ldots, M_{Z_n})]_\beta^*\xi, 1\right>
=\overline{[f(\lambda )]}_\beta\left<\xi,
1\right>\\
&=c_0 \overline{[f(\lambda )]}_\beta
\end{split}
\end{equation*}
for any $\beta\in \FF_n^+$. Therefore, we have
$$
\xi=c_0\sum_{\beta\in \FF_n^+}   \overline{[f(\lambda )]}_\beta\  f_\beta.
$$
Since $\xi\in \HH^2(f)$, we must have
$$
\sum_{k=0}^\infty (|f_1(\lambda)|^2+\cdots +|f_n(\lambda)|^2)^k =
\sum_{\beta\in \FF_n^+} |[f(\lambda )]_\beta|^2<\infty.
$$
Hence, we deduce that $|f_1(\lambda)|^2+\cdots +|f_n(\lambda)|^2<1$.

Now, due to relation \eqref{xi} and using again that
$M_{f_i}=f_i(M_{Z_1},\ldots, M_{Z_n})$,   $i=1,\ldots,n$, we deduce
that $M_{f_i}^* \xi=\overline{f_i(\lambda)}\xi$. On the other,
according to the proof of Theorem \ref{model} (see relation
\eqref{gi}), we have
$$M_{Z_i}=g_i(M_{f_1},\ldots, M_{f_n})=\text{\rm SOT-}\lim_{r\to
1}g_i(rM_{f_1},\ldots, rM_{f_n}).
$$
As above, one can show that $g_i(M_{f_1},\ldots,
M_{f_n})^*\xi=\overline{g_i(f(\lambda))} \xi$ for   $i=1,\ldots,n$.
Combining this relation with  the fact that $ M_{Z_i}^*
\xi=\overline{\lambda}_i \xi$, $ i=1,\ldots, n$, we conclude that
$\lambda=g(f(\lambda))$. Therefore, $\lambda\in \BB_{f}^<(\CC)$.

According to Theorem \ref{algebra}, part (iii), we have
$\varphi(M_Z)=\text{\rm SOT-}\lim_{r\to
1}\sum_{k=0}^\infty\sum_{|\alpha|=k} c_\alpha r^{|\alpha|}
 [f(M_Z)]_\alpha
 $
 for some coefficients $ c_\alpha\in \CC$. Using relation \eqref{xi}, we deduce
 that
 \begin{equation*}
 \begin{split}
\left< \varphi(M_{Z}) \Gamma_\lambda, \Gamma_\lambda\right> &=
\text{\rm SOT-}\lim_{r\to 1}\sum_{k=0}^\infty\sum_{|\alpha|=k}
c_\alpha r^{|\alpha|}\left< \Gamma_\lambda, [f(M_Z)]_\alpha^*
\Gamma_\lambda\right>\\
&=\text{\rm SOT-}\lim_{r\to 1}\sum_{k=0}^\infty\sum_{|\alpha|=k}
c_\alpha r^{|\alpha|}\left< \Gamma_\lambda,
\overline{[f(\lambda)]}_\alpha
\Gamma_\lambda\right>\\
  &=  \text{\rm SOT-}\lim_{r\to 1}\|\Gamma_\lambda\|^2\,\sum_{k=0}^\infty\sum_{|\alpha|=k}
c_\alpha r^{|\alpha|}[f(\lambda)]_\alpha=\varphi(\lambda).
 \end{split}
 \end{equation*}
Similarly, one can show that $
\varphi(M_{Z})^*\Gamma_\lambda=\overline{\varphi(\lambda)}
\Gamma_\lambda.$
According to Lemma \ref{ball} part (i), the mapping
$f|_{\BB_f^<(\CC)}:\BB_f^<(\CC)\to \BB_n$ is the inverse of
$g|_{\BB_n}:\BB_n\to \BB_f^<(\CC)$. Since $g$ is a bounded free
holomorphic function on $[B(\cH)^n]_1$, the map $\BB_n\ni
\lambda\mapsto g(\lambda)\in \BB_f^<(\CC)$  is holomorphic on
$\BB_n$ and its inverse $\BB_f^<(\CC)\ni \lambda\mapsto
f(\lambda)\in \BB_n$ is also holomorphic. On the other hand,
according  to Theorem \ref{algebra}, part (iii),  there is $\psi\in H^\infty_{\bf
ball}$ such that $\varphi(\lambda)=\psi(f(\lambda))$ for $\lambda\in
\BB_f^<(\CC)$. Hence, we deduce that $\lambda\mapsto
\varphi(\lambda)$ is a bounded holomorphic function on
$\BB_f^<(\CC)$. This completes the proof.
\end{proof}

 Theorem \ref{eigenvalues} can be used to prove the following result. Since the
  proof   is similar to  the corresponding result from \cite{DP1},
       we shall omit it.
\begin{corollary}\label{w*-funct} A map $\Phi: H^\infty(\BB_f)\to
\CC$ is a WOT-continuous multiplicative linear functional  if and
only if there exists $\lambda\in \BB_{f}^<(\CC)$  such that
$$
\Phi(A)=\Phi_\lambda(A):=\left<A\Gamma_\lambda,
\Gamma_\lambda\right>,\qquad A\in H^\infty(\BB_f),
$$
where $\Gamma_\lambda$ is the  noncommutative Poisson kernel  associated with the
 domain $\BB_f$ at $\lambda$.
\end{corollary}

Assume that $f=(f_1,\ldots, f_n)$ is an $n$-tuple of formal power series with the model property.
Using Theorem \ref{algebra},
 one can prove that  $J$  is a WOT-closed two-sided ideal of
$H^\infty(\BB_f)$  if and only if there is a WOT-closed two-sided
ideal $\cI$ of $F_n^\infty$ such that
$$
J=\{\varphi(f(M_Z)): \  \varphi\in \cI\}.
$$
We mention that if $\varphi(S_1,\ldots,S_n)\in F_n^\infty$ has the
Fourier representation $\varphi(S_1,\ldots,S_n)=\sum_{\alpha\in
\FF_n^+} c_\alpha S_\alpha$, then $$\varphi(f(M_Z))=\text{\rm
SOT-}\lim_{r\to 1}\sum_{k=0}^\infty \sum_{|\alpha|=k} c_\alpha
r^{|\alpha|} [f(M_Z)]_\alpha
$$
exists.
Denote by $H^\infty(\cV_{f,J})$   the WOT-closed algebra generated
 by the operators
 $B_i:=P_{\cN_J} M_{Z_i} |\cN_J$,  for  $i=1,\ldots, n$, and the identity, where
 $$
 \cN_J:= \HH^2(f)\ominus \cM_J\quad \text{ and }\quad
 \cM_J:=\overline{ J \HH^2(f)}.
 $$
The following result is a consequence   of  Theorem 4.1 from
\cite{ArPo2} and the above-mentioned  remarks.

\begin{theorem}
\label{F/J} Let $J$ be a WOT-closed  two-sided ideal of the Hardy
algebra $H^\infty(\BB_f)$. Then  the map
$$\Gamma:H^\infty(\BB_f)/J\to
B(\cN_J) \quad \text{ defined by } \quad \Gamma(\varphi+J)=P_{\cN_J}
\varphi|_{\cN_J}
$$
is a   completely isometric representation.
\end{theorem}

Since the set of all polynomials in $M_{Z_1},\ldots, M_{Z_n}$  and
the identity is WOT-dense in $H^\infty(\BB_f)$, Theorem \ref{F/J}
implies that $P_{\cN_J}H^\infty(\BB_f)|_{\cN_J}$ is  a WOT-closed
subalgebra  of $B(\cN_J)$ and, moreover,
$H^\infty(\cV_{f,J})=P_{\cN_J}H^\infty(\BB_f)|_{\cN_J}$.

We need a few more definitions. For each $\lambda=(\lambda_1,\ldots,
\lambda_n)$ and each $n$-tuple ${\bf k}:=(k_1,\ldots, k_n)\in
\NN_0^n$, where $\NN_0:=\{0,1,\ldots \}$, let $\lambda^{\bf
k}:=\lambda_1^{k_1}\cdots \lambda_n^{k_n}$. If ${\bf k}\in \NN_0$,
we denote
$$
\Lambda_{\bf k}:=\{\alpha\in \FF_n^+: \ \lambda_\alpha =\lambda^{\bf
k} \text{ for all } \lambda\in \CC^n\}.
$$
For each ${\bf k}\in \NN_0^n$, define the formal power series
$$
\omega^{(\bf k)}:=\frac{1}{\gamma_{\bf k}} \sum_{\alpha \in
\Lambda_{\bf k} } f_\alpha\in \HH^2(f), \quad  \text{ where } \
\gamma_{\bf k}:=\text{\rm card}\,  \Lambda_{\bf
k}=\left(\begin{matrix} |{\bf k}|!\\
k_1!\cdots k_n!\end{matrix}\right).
$$
   Note that the set  $\{\omega^{(\bf k)}:\ {\bf
k}\in \NN_0^n\}$ consists  of orthogonal power series in $\HH^2(f)$
and
$\|\omega^{(\bf k)}\|=\frac{1}{\sqrt{\gamma_{\bf k}}}$. We denote by
$\HH_s^2(f)$ the closed span of these formal power series, and call
it the symmetric Hardy space associated with the noncommutative
domain $\BB_f$.

\begin{theorem}\label{symm-Hardy} Let $f=(f_1,\ldots, f_n)$ be  an  $n$-tuple of formal power series with
 the model property  and  let $(M_{Z_1},\ldots, M_{Z_n})$ be
the universal model associated  with the noncommutative domain
$\BB_f$.  Let $J_c$ be the WOT-closed two-sided ideal of the Hardy
algebra $H^\infty(\BB_f)$ generated by the commutators
$$ M_{Z_i}M_{Z_j}-M_{Z_j}M_{Z_i},\qquad i,j=1,\ldots, n.
$$
 Then the following statements hold.
 \begin{enumerate}
 \item[(i)]
 $
 \HH_s^2(f)=\overline{\text{\rm span}}\{\Gamma_\lambda: \ \lambda\in
 \BB_{f}^<(\CC)\}=\cN_{J_c}:=\HH^2(f)\ominus \overline{J_c(1)}. $
\item[(ii)]  The symmetric Hardy  space $\HH_s^2(f)$ can be
identified with the Hilbert space $H^2(\BB_{f}^<(\CC))$ of all
holomorphic functions $\psi:\BB_{f}^<(\CC)\to \CC$ which admit a
series representation $\psi(\lambda)=\sum_{{\bf k}\in \NN_0} c_{\bf
k} f(\lambda)^{\bf k}$ with
$$
\|\psi\|_2=\sum_{{\bf k}\in \NN_0}|c_{\bf
k}|^2\frac{1}{\gamma_{\bf k}}<\infty.
$$
More precisely, every  element  $\psi=\sum_{{\bf k}\in \NN_0} c_{\bf
k} \omega^{(\bf k)}$ in $\HH_s^2(f)$  has a functional
representation on $\BB_{f}^<(\CC)$ given by
$$
\psi(\lambda):=\left<\psi, \Omega_\lambda\right>=\sum_{{\bf k}\in
\NN_0} c_{\bf k} f(\lambda)^{\bf k}, \qquad
\lambda=(\lambda_1,\ldots, \lambda_n)\in \BB_{f}^<(\CC),
$$
where $\Omega_\lambda:=\frac{1}{\sqrt{1- \sum_{ i=1}^n
 |f_i(\lambda) |^2}} \Gamma_\lambda $ and
$$
|\psi(\lambda)|\leq \frac{\|\psi\|_2}{\sqrt{1-\sum_{ i=1}^n
 |f_i(\lambda) |^2}},\qquad \lambda=(\lambda_1,\ldots,
\lambda_n)\in \BB_{f}^<(\CC).
$$
\item[(iii)]
 The mapping $\Lambda_f:\BB_{f}^<(\CC)\times
\BB_{f}^<(\CC)\to \CC$ defined by
$$
\Lambda_f(\mu,\lambda):= \left<\Omega_\lambda,
\Omega_\mu\right>=\frac{1}{1-\sum_{ i=1}^n   f_i(\mu)
\overline{f_i(\lambda)}},\qquad \lambda,\mu\in \BB_{f}^<(\CC),
$$
is positive definite.
\end{enumerate}
\end{theorem}

\begin{proof}
First, note that $\Omega_\lambda=\sum_{{\bf k}\in \NN_0^n}
\overline{f(\lambda)}^{\bf k} \gamma_{\bf k} \omega^{(\bf k)}$,
$\lambda \in \BB_{f}^<(\CC)$, and, therefore,
 $$\overline{\text{\rm span}}\{\Gamma_\lambda: \ \lambda\in
\BB_{f}^<(\CC)\}\subseteq \HH_s^2(f).
$$
Now, we prove  that $\omega^{(\bf k)}\in \cN_{J_c}:=F^2(H_n)\ominus
\overline{J_c(1)}$. First,  we show that   $J_c$ coincides with the
WOT-closed
 commutator ideal  of $H^\infty(\BB_f)$.
 Indeed, since $M_{Z_i}M_{Z_j}-M_{Z_j}M_{Z_i}\in J_c$ and every
permutation of $k$ objects is a product of transpositions, it is
clear that $M_{Z_\alpha} M_{Z_\beta}-M_{Z_\beta} M_{Z_\alpha}\in
J_c$ for any $\alpha,\beta\in \FF_n^+$. Consequently,
$M_{Z_\gamma}(M_{Z_\alpha} M_{Z_\beta}-M_{Z_\beta}
M_{Z_\alpha})M_{Z_\omega}\in J_c$ for any $\alpha,\beta,
\gamma,\omega\in \FF_n^+$. Since the polynomials in $M_{Z_1},\ldots,
M_{Z_n}$ are WOT dense in $H^\infty(\BB_f)$, the result follows.
Note also that $\overline{J_c(1)}\subset \HH^2(f)$ coincides with
 $$\overline{\text{\rm
span}}\left\{ Z_{\gamma g_jg_i\beta}-  Z_{\gamma g_ig_j\beta}:\
\gamma, \beta\in \FF_n^+, i,j=1,\ldots,n\right\}.
$$
Similarly, one can prove that the WOT-closed two-sided ideal
generated by the commutators $M_{f_j}M_{f_i}-M_{f_i}M_{f_j}$,
$i,j\in \{1,\ldots,n\}$ coincides with   the WOT-closed
 commutator ideal  of $H^\infty(\BB_f)$. Combining these results, we
 deduce that $J_c$ coincides with the WOT-closed two-sided ideal
 generated by the commutators $M_{f_j}M_{f_i}-M_{f_i}M_{f_j}$,
 $i,j\in \{1,\ldots, n\}$ and
 $$
\overline{J_c(1)} =\overline{\text{\rm span}}\left\{ f_{\gamma
g_jg_i\beta}- f_{\gamma g_ig_j\beta}:\ \gamma,\beta\in \FF_n^+,
i,j=1,\ldots,n\right\}.
$$
Consequently, since
$$
\left<\sum_{\alpha \in \Lambda_{\bf k}} f_\alpha,M_{f_\gamma}(M_{f_j}M_{f_i}-M_{f_i}M_{f_j})M_{f_\beta}(1)\right>=0
$$
for any ${\bf k}\in \NN_0^n$, we deduce  that $\omega^{(\bf k)}\in
\cN_{J_c}$. Hence, we have $\HH_s^2(f)\subseteq \cN_{J_c}$. To
complete the proof of part (i), it is enough to show that
$$
\overline{\text{\rm span}}\{\Gamma_\lambda: \ \lambda\in
\BB_{f}^<(\CC)\}=\cN_{J_c}.
$$
Assume that there is a vector $ x:=\sum_{\beta\in \FF_n^+} c_\beta
f_\beta\in \cN_{J_c}$ and $x\perp \Gamma_\lambda$ for all
$\lambda\in \BB_{f}^<(\CC)$. Then
\begin{equation*}
\begin{split}
\left<\sum_{\beta\in \FF_n^+} c_\beta f_\beta,
\Omega_\lambda\right>&=
\sum_{\beta\in \FF_n^+} c_\beta [f(\lambda)]_\beta
=\sum_{{\bf k}\in \NN_0^n}\left(\sum_{\beta\in \Lambda_{\bf k}} c_\beta
\right)f(\lambda)^{\bf k}=0
\end{split}
\end{equation*}
for any $\lambda\in \BB_{f}^<(\CC)$. Since $\BB_{f}^<(\CC)$ contains
an open ball in $\CC^n$, we deduce that
\begin{equation}\label{sigma=0}
\sum_{\beta\in \Lambda_{\bf k}} c_\beta =0 \quad
\text{ for all } \ {\bf k}\in \NN_0^n.
\end{equation}
Fix $\beta_0\in \Lambda_{\bf k}$ and let $\beta\in \Lambda_{\bf k}$
be such that $\beta$ is obtained from $\beta_0$ by transposing just
two generators. So we can assume that $\beta_0=\gamma g_j g_i\omega$
and $\beta=\gamma g_i g_j\omega$ for some $\gamma,\omega\in \FF_n^+$
and $i\neq j$, $i,j=1,\ldots,n$. Since $x\in
\cN_{J_c}=\HH^2(f)\ominus\overline{J_c(1)}$, we must have
$$
\left<x,M_{f_\gamma}(M_{f_j}M_{f_i}-M_{f_i}M_{f_j})M_{f_\omega}(1)\right>=0,
$$
which implies $c_{\beta_0}=c_\beta$. Since any element $\gamma\in
\Lambda_{\bf k}$ can be obtained from $\beta_0$  by successive
transpositions, repeating the above argument, we deduce that $
c_{\beta_0}=c_\gamma \quad \text{ for all }\ \gamma\in \Lambda_{\bf
k}. $ Now relation  \eqref{sigma=0} implies   $c_\gamma=0$ for any
$\gamma\in \Lambda_{\bf k}$ and ${\bf k}\in \NN_0^n$, so $x=0$.
Consequently, we have $\overline{\text{\rm span}}\{\Gamma_\lambda: \
\lambda\in \BB_{f}^<(\CC)\}=\cN_{J_c}. $

Now, let us prove part (ii) of the theorem. Note that
\begin{equation*}
\begin{split}
\left<\omega^{(\bf k)},\Omega_\lambda\right>&=\frac{1}{\gamma_{\bf
k}} \left< \sum_{\beta\in\Lambda_{\bf k}} f_\beta,
\Omega_\lambda\right>
=\frac{1}{\gamma_{\bf k}}\sum_{\beta\in\Lambda_{\bf k}}
[f(\lambda)]_\beta =f(\lambda)^{\bf k}
\end{split}
\end{equation*}
for any $\lambda\in \BB_{f}^<(\CC)$ and ${\bf k}\in \NN_0^n$. Hence,
every  element  $\psi=\sum_{{\bf k}\in \NN_0} c_{\bf k} \omega^{(\bf
k)}$ in $\HH_s^2(f)$  has a functional representation on
$\BB_{f}^<(\CC)$ given by
$$
\psi(\lambda):=\left<\psi, \Omega_\lambda\right>=\sum_{{\bf k}\in
\NN_0} c_{\bf k} f(\lambda)^{\bf k}, \qquad
\lambda=(\lambda_1,\ldots, \lambda_n)\in \BB_{f}^<(\CC),
$$
and
$$|\psi(\lambda)|\leq \|\psi\|_2 \|\Omega_\lambda\|=
\frac{\|\psi\|_2}{\sqrt{1-\sum_{i=1 }^n |f_i(\lambda)|^2}}.
$$
  The identification of $\HH_s^2(f)$
with  $H^2(\BB_{f}^<(\CC))$  is now clear. If $
 (\lambda_1,\ldots, \lambda_n)$ and $(\mu_1,\ldots, \mu_n)$ are in
$\BB_{f}^<(\CC)$,  then we have
$$
\Lambda_f(\mu,\lambda):= \left<\Omega_\lambda,
\Omega_\mu\right>=\sum_{\beta\in \FF_n^+}
[f(\mu)]_\beta\overline{[f(\lambda)]}_\beta,
$$
which implies item (iii).
 The proof is complete.
\end{proof}

  If $~A\in B(\cH)~$ then we  denote by $\text{\rm Lat~}A$ the set of all
invariant subspaces of $~A~$. When
 $~\cU\subset B(\cH)~$, we define
$
\text{\rm Lat~}\cU=\bigcap_{A\in\cU}\text{\rm Lat~}A.
$
 Given  any collection  $~\cS~$  of subspaces of $~\cH$,
then we  set $$\text{\rm
Alg~}\cS:=\{A\in B(\cH):\ \cS\subset\text{\rm Lat~}A\}.$$ We recall
that the algebra $~\cU\subset B(\cH)~$ is reflexive if
$\cU=\text{\rm Alg Lat~}\cU.$

\begin{theorem}\label{W(L)} Let $f=(f_1,\ldots, f_n)$ be  an  $n$-tuple of formal power series with
 the model property  and  let $(M_{Z_1},\ldots, M_{Z_n})$ be
the universal model associated  with the noncommutative domain
$\BB_f$. If $H^\infty(\cV_{f,J_c})$ is the WOT-closed algebra
generated by the operators
$$L_i:=P_{\HH_s^2(f)} M_{Z_i}|_{\HH_s^2(f)}, \qquad i=1,\ldots, n,
$$
and the identity, then the following statements hold.
\begin{enumerate}
\item[(i)]  $H^\infty(\cV_{f,J_c})$ can be identified
 with the algebra of  all multipliers of the  Hilbert space $H^2(\BB_{f}^<(\CC))$.
 \item[(ii)]The algebra $H^\infty(\cV_{f,J_c})$ is reflexive.
\end{enumerate}
\end{theorem}

\begin{proof} According to the remarks following Theorem \ref{F/J},
we have
$H^\infty(\cV_{f,J_c})=P_{\HH_s^2(f)}H^\infty(\BB_f)|_{\HH_s^2(f)}$.
 Let $\varphi(M_Z)\in H^\infty(\BB_f)$ and
$\varphi(L)=P_{\HH_s^2(f)} \varphi(M_Z)|_{\HH_s^2(f)}$.  Due to
Theorem \ref{symm-Hardy}, since $\Omega_\lambda\in \HH_s^2(f)$ for
$\lambda\in \BB_{f}^<(\CC)$, and $\varphi(M_Z)^*\Omega_\lambda=
\overline{\varphi(\lambda)} \Omega_\lambda$ (see Theorem
\ref{eigenvalues}), we have
\begin{equation*}
\begin{split}
[\varphi(L) \psi](\lambda)&=\left<\varphi(L)\psi,
\Omega_\lambda\right>
=\left<\varphi(M_Z)\psi, \Omega_\lambda\right>\\
&=\left< \psi,\varphi(M_Z)^*\Omega_\lambda\right>
=\left< \psi,\overline{\varphi(\lambda)} \Omega_\lambda\right>\\
&=\varphi(\lambda)\psi(\lambda)
\end{split}
\end{equation*}
for any $\psi\in \HH_s^2(f)$ and $\lambda\in \BB_{f}^<(\CC)$.
Therefore, the operators in $\HH^\infty(\cV_{f,J_c})$ are
``analytic'' multipliers of $\HH_s^2(f)$. Moreover,
$$
\|\varphi(L)\|=\sup\{\|\varphi \chi\|_2:\ \chi\in \HH_s^2(f), \,
\|\chi\|\leq 1\}.
$$
Conversely, suppose that $\psi=\sum_{{\bf k}\in \NN_0} c_{\bf k}
\omega^{\bf (k)}$ is a bounded multiplier, i.e., $M_\psi\in
B(\HH_s^2(f))$. As in \cite{DP1} (see Lemma 1.1), using Cesaro
means,  one can find a sequence  $q_m=\sum  c_{\bf k}^{(m)}
\omega^{(\bf k)}$ such that $M_{q_m}$ converges to $M_\psi$ in the
strong operator topology and, consequently, in  the $w^*$-topology.
Since $M_{q_m}$ is a polynomial in $L_1,\ldots, L_n$, we conclude
that $M_\psi\in H^\infty(\cV_{f,J_c})$. In particular $L_i$ is the
multiplier $M_{\lambda_i}$  by the coordinate function.

Now, we prove part (ii).
 Let $Y\in B(\HH_s^2(f))$ be an operator   that leaves
invariant all the invariant subspaces under each operator
$L_1,\ldots, L_n$. According  to Theorem \ref{eigenvalues}, we have
$L_i^*\Gamma_\lambda=\overline{\lambda}_i \Gamma_\lambda$ for any
$\lambda\in \BB_{f}^<(\CC)$ and $i=1,\ldots,n$. Since $Y^*$ leaves
invariant all the invariant subspaces under $L_1^*,\ldots, L_n^*$,
the vector $\Omega_\lambda$ must be an eigenvector for $Y^*$.
Consequently, there is a function $\varphi:\BB_{f}^<(\CC)\to \CC$
such that $Y^*\Omega_\lambda=\overline{\varphi(\lambda)}
\Omega_\lambda$ for any $\lambda\in \BB_{f}^<(\CC)$. Due to Theorem
\ref{symm-Hardy},  if $f\in H_s^2(f)$, then  $Yf$ has the functional
representation
$$
(Yf)(\lambda)=\left<Yf,\Omega_\lambda\right>=\left<f,Y^*\Omega_\lambda\right>=
\varphi(\lambda)f(\lambda)\quad \text{ for all }\ \lambda\in
\BB_{f}^<(\CC).
$$
In particular, if $f=1$, then  the   functional representation  of
$Y(1)$  coincide with $\varphi$. Therefore, $\varphi$ admits a
representation  $\sum_{{\bf k}\in \NN_0} c_{\bf k} f(\lambda)^{\bf
k}$ on $\BB_{f}^<(\CC)$ and can be identified with $X(1)\in
\HH_s^2(f)$. Moreover, the equality above shows that $\varphi f\in
H^2(\BB_{f}^<(\CC))$ for any $f\in \HH_s^2(f)$. Applying the first
part of this theorem, we deduce that $Y=M_\varphi\in
H^\infty(\cV_{f,J_c})$. The proof is complete.
\end{proof}

We remark that, in  the particular case when $f=(Z_1,\ldots, Z_n)$,
we recover some of the results obtained by Arias and the author,
Davidson and Pitts, and Arveson (see  \cite{Po-disc}, \cite{ArPo},
\cite{ArPo2}, \cite{DP}, \cite{DP1}, and \cite{Arv}).

\section{Characteristic functions  and functional models  }

 In this section,
we introduce
  the characteristic  function  of  an  $n$-tuple $T=(T_1,\ldots, T_n)\in
  \BB_f(\cH)$,
   present a model for  pure $n$-tuples of operators  in the
noncommutative domain $\BB_f(\cH)$  in terms of characteristic
functions, and show that the   characteristic function is a complete
unitary invariant for pure $n$-tuples of operators  in $\BB_f(\cH)$.

 Let $f=(f_1,\ldots, f_n)$ be  an  $n$-tuple of formal power series with
 the model property  and  let $(M_{Z_1},\ldots, M_{Z_n})$ be
the universal model associated  with the noncommutative domain
$\BB_f$. We introduce
  the characteristic  function  of  an  $n$-tuple $T=(T_1,\ldots, T_n)\in
  \BB_f(\cH)$ to be
  the multi-analytic operator, with respect to $M_{Z_1},\ldots, M_{Z_n}$,
$$
\Theta_{f,T}:\HH^2(f)\otimes \cD_{f,T^*}\to \HH^2(f)\otimes \cD_{f,T}
$$
having  the formal Fourier representation
\begin{equation*}
\begin{split}
  -I_{\HH^2(f)}\otimes f(T)+
\left(I_{\HH^2(f)}\otimes \Delta_{f,T}\right)&\left(I_{\HH^2(f)\otimes \cH}-\sum_{i=1}^n \Lambda_i\otimes f_i(T)^*\right)^{-1}\\
&\left[\Lambda_1\otimes I_\cH,\ldots, \Lambda_n\otimes I_\cH
\right] \left(I_{\HH^2(f)}\otimes \Delta_{f,T^*}\right),
\end{split}
\end{equation*}
where $\Lambda_1,\ldots, \Lambda_n$ are the right multiplication
operators  by the power series $f_1,\ldots, f_n$, respectively,  on
the Hardy space $\HH^2(f)$, and the defect operators  associated
with $T:=(T_1,\ldots, T_n)\in \BB_f(\cH)$ are
\begin{equation*}
\Delta_{f,T}:=\left( I_\cH-\sum_{i=1}^n f_i(T)f_i(T)^*\right)^{1/2}\in B(\cH) \quad \text{ and }\quad \Delta_{f,T^*}:=(I-f(T)^*f(T))^{1/2}\in B(\cH^{(n)}),
\end{equation*}
while the defect spaces are $\cD_{f,T}:=\overline{\Delta_{f,T}\cH}$ and
$\cD_{f,T^*}:=\overline{\Delta_{f,T^*}\cH^{(n)}}$, where
$\cH^{(n)}$ denotes the direct sum of $n$ copies of $\cH$.
We remark that when $f=(f_1,\ldots, f_n)=(Z_1,\ldots, Z_n)$,
 we recover the characteristic function for row contractions.
We recall that the characteristic  function associated with an
arbitrary row contraction $T:=[T_1,\ldots, T_n]$, \ $T_i\in B(\cH)$,
was introduce in \cite{Po-charact} (see \cite{SzF-book} for the
classical case $n=1$) and it was proved to be  a complete unitary
invariant for completely non-coisometric (c.n.c.) row contractions.
Related to our setting, we remark  that
\begin{equation}
\label{TT}
 \Theta_{f,T}=(U^*\otimes
I_{\cD_{f,T}})\Theta_{f(T)}(U\otimes I_{\cD_{f,T^*}}),
\end{equation}
where $\Theta_{f(T)}$ is the characteristic function of the row contraction
 $f(T)=[f_1(T),\ldots, f_n(T)]$ and $U:\HH^2(f)\to F^2(H_n)$ is the
  canonical unitary operator defined by $Uf_\alpha=e_\alpha$, $\alpha\in \FF_n^+$.
   Consequently, due to Theorem 3.2 from \cite{Po-varieties}, we deduce the following  result.

\begin{theorem}\label{factor} Let $f=(f_1,\ldots, f_n)$ be  an  $n$-tuple of formal power series with
the model  property  and let  $T=(T_1,\ldots, T_n)\in \BB_f(H)$.
Then
\begin{equation}\label{fa}
I_{\HH^2(f)\otimes
\cD_{f,T}}-\Theta_{f,T}\Theta_{f,T}^*=K_{f,T}K_{f,T}^*,
\end{equation}
where
$\Theta_{f,T}$ is  the characteristic function  of $T$ and $K_{f,T}$  is the corresponding Poisson kernel.
\end{theorem}

Now we present a model for  pure $n$-tuples of operators  in the
noncommutative domain $\BB_f(\cH)$  in terms of characteristic
functions.

\begin{theorem}\label{funct-model} Let $f=(f_1,\ldots, f_n)$ be  an  $n$-tuple of formal power series with
 the model property  and  let $(M_{Z_1},\ldots, M_{Z_n})$ be
the universal model associated  with the noncommutative domain
$\BB_f$. If  $T=(T_1,\ldots, T_n)$ is a pure $n$-tuple  of operators
in $ \BB_f(\cH)$,   then  the characteristic function $\Theta_{f,T}
$ is an  isometry and $T$ is unitarily equivalent to the $n$-tuple
\begin{equation}\label{HH}
\left(P_ {{\bf H}_{f,T}} (M_{Z_1}\otimes I_{\cD_{f,T}})|_{{\bf
H}_{f,T}},\ldots, P_{{\bf H}_{f,T}} (M_{Z_n}\otimes
I_{\cD_{f,T}})|_{{\bf H}_{f,T}}\right),
\end{equation}
where $P_{{\bf H}_{J,T}}$ is the orthogonal projection of $\HH^2(f)\otimes \cD_{f,T}$ on the Hilbert space
$${\bf H}_{f,T}:=\left(\HH^2(f)\otimes \cD_{f,T}\right)\ominus
\Theta_{f,T}(\HH^2(f)\otimes \cD_{f,T^*}).
$$
 \end{theorem}
\begin{proof}
According to  Theorem \ref{pure-dilation},  the noncommutative
Poisson kernel $K_{f,T}:\cH\to \HH^2(f)\otimes {\cD_{f,T}}$
 is an isometry, $K_{f,T}\cH$ is  a co-invariant  subspace under
$M_{Z_i}\otimes  I_{\cD_{f,T}}$, $i=1,\dots, n$,   and
\begin{equation}\label{Pois2}
T_i=K_{f,T}^*(M_{Z_i}\otimes  I_{{\cD_{f,T}}})K_{f,T},\quad i=1,\ldots, n.
\end{equation}
Hence, $K_{f,T}K_{f,T}^*$ is the orthogonal projection of
$\HH^2(f)\otimes {\cD_{f,T}}$ onto $K_{f,T}\cH$. Using relation
\eqref{Pois2}, we deduce  that $K_{f,T}K_{f,T}^*$ and
$\Theta_{f,T}\Theta_{f,T}^*$ are mutually orthogonal projections
such that
$$
K_{f,T}K_{f,T}^*+\Theta_{f,T}\Theta_{f,T}^*=
I_{\HH^2(f)\otimes {\cD_{f,T}}}.
$$
This implies
$$
K_{f,T}\cH=(\HH^2(f)\otimes \cD_{f,T})\ominus
\Theta_{f,T}(\HH^2(f)\otimes \cD_{f,T^*}).
$$
Taking into account that $K_{f,T}$ is an isometry, we identify the
Hilbert space $\cH$ with ${\bf H}_{f,T}:=K_{f,T}\cH$.  Using again
relation \eqref{Pois2}, we deduce that $T$ is unitarily equivalent
to the $n$-tuple given by  relation \eqref{HH}. That $\Theta_{f,T}$
is an isometry follows from relation \eqref{TT} and the fact that
the characteristic function of a pure row contraction is an isometry
\cite{Po-charact}.
 The proof is complete.
\end{proof}
Let  $\Phi:\HH^2(f)\otimes \cK_1 \to \HH^2(f)\otimes \cK_2$ and
$\Phi':\HH^2(f)\otimes \cK_1' \to \HH^2(f)\otimes \cK_2'$ be two
multi-analytic
 operators with respect to $M_{Z_1},\ldots, M_{Z_n}$. We say that $\Phi$ and $\Phi'$ coincide
if there are two  unitary multi-analytic operators $W_j:
\HH^2(f)\otimes \cK_j\to \HH^2(f)\otimes\cK_j'$, $j=1,2$,  with
respect to $M_{Z_1},\ldots, M_{Z_n}$ such that $\Phi' W_1=W_2\Phi$.
Since $W_j(M_{Z_i}\otimes I_{\cK_j})=(M_{Z_i}\otimes
I_{\cK_j'})W_j$, $ i=1,\ldots, n, $ we also have
$W_j(M_{Z_i}^*\otimes I_{\cK_j})=(M_{Z_i}^*\otimes I_{\cK_j'})W_j$,
$ i=1,\ldots, n$. Taking into account that $C^*(M_{Z_1},\ldots,
M_{Z_n})$ is irreducible (see Theorem \ref{irreducible}), we
conclude that $
 W_j =I_{\HH^2(f)}\otimes \tau_j$, $ j=1,2,
$
for some unitary operators $\tau_j\in B(\cK_j, \cK_j')$.

 The next result shows that the   characteristic function is a
complete unitary invariant for pure $n$-tuple of operators  in
$\BB_f(\cH)$.

\begin{theorem}\label{u-inv}  Let $f=(f_1,\ldots, f_n)$ be  an  $n$-tuple of formal power series with
 the model property  and
let  $T=(T_1,\ldots, T_n)\in \BB_f(\cH)$  and  $T'=(T_1',\ldots,
T_n')\in \BB_f(\cH')$ be two pure $n$-tuples  of operators. Then $T$
and $T'$ are  unitarily equivalent if and only if their
characteristic functions $\Theta_{f,T}$  and $\Theta_{f,T'}$
coincide.
\end{theorem}
\begin{proof}
Assume that $T$ and $T'$ are unitarily equivalent and let $W:\cH\to
\cH'$ be a unitary operator such that $T_i=W^*T_i'W$ for
$i=1,\ldots, n$.  Note that
$$
W\Delta_{f,T}=\Delta_{f,T'}W \quad \text{ and }\quad
(\oplus_{i=1}^n W)\Delta_{f,T^*}=\Delta_{f,T'^*}(\oplus_{i=1}^n W).
$$
Consider the unitary operators $\tau$ and $\tau'$ defined by
$$\tau:=W|_{\cD_{f,T}}:\cD_{f,T}\to \cD_{f,T'} \quad \text{ and }\quad
\tau':=(\oplus_{i=1}^n W)|_{\cD_{f,T^*}}:\cD_{f,T*}\to \cD_{f,T'^*}.
$$
Using the definition of the characteristic function,  we deduce that
$ (I_{\HH^2(f)}\otimes
\tau)\Theta_{f,T}=\Theta_{f,T'}(I_{\HH^2(f)}\otimes \tau'). $

Conversely, assume that the characteristic functions  of $T$ and
$T'$ coincide. Then there exist unitary operators $\tau:\cD_{f,T}\to
\cD_{f,T'}$ and $\tau_*:\cD_{f,T^*}\to \cD_{{f,T'}^*}$ such that
\begin{equation}\label{com}
(I_{\HH^2(f)}\otimes
\tau)\Theta_{f,T}=\Theta_{f,T'}(I_{\HH^2(f)}\otimes \tau_*).
\end{equation}
Hence, we deduce that $ V:=(I_{\HH^2(f)}\otimes
\tau)|_{\HH_{f,T}}:\HH_{f,T}\to \HH_{f,T'} $ is a unitary operator,
where $\HH_{f,T}$ and $ \HH_{f,T'}$ are the model spaces for  the
$n$-tuples $T$ and $T'$, respectively, as defined in  Theorem
\ref{funct-model}. Since
$$
(M_{Z_i}^*\otimes I_{\cD_{f,T}})(I_{\HH^2(f)}\otimes \tau^*)=
(I_{\HH^2(f)}\otimes \tau^*)(M_{Z_i}^*\otimes I_{\cD_{f,T'}}),\qquad
i=1,\ldots, n,
$$
and $\HH_{f,T}$ (resp. $ \HH_{f,T'}$) is a co-invariant subspace under $M_{Z_i}\otimes I_{\cD_{f,T}}$ (resp. $M_{Z_i}\otimes I_{\cD_{f,T'}}$), \ $i=1,\ldots, n$, we deduce that
$$
\left[(M_{Z_i}^*\otimes I_{\cD_{f,T}})|_{\HH_{f,T}}\right] V^*=
V^*\left[(M_{Z_i}^*\otimes
I_{\cD_{f,T'}})|_{\HH_{f,T'}}\right],\qquad i=1,\ldots, n.
$$
Consequently, we obtain
$$
V\left[P_{\HH_{f,T}}\left(M_{Z_i}\otimes
I_{\cD_{f,T}}\right)|_{\HH_{f,T}}\right]=
\left[P_{\HH_{f,T'}}\left(M_{Z_i}\otimes
I_{\cD_{f,T'}}\right)|_{\HH_{f,T'}}\right] V,\qquad i=1,\ldots, n.
$$
Now, using Theorem \ref{funct-model}, we conclude that $T$ and $T'$
are unitarily equivalent. The proof is complete.
\end{proof}

\begin{theorem}\label{new-inv} Let $f=(f_1,\ldots, f_n)$ be  an  $n$-tuple of formal power series with
 the model property   and  let $(M_{Z_1},\ldots, M_{Z_n})$ be the universal
 model associated with  the noncommutative domain $\BB_f$.
 Then the following statements hold.
 \begin{enumerate}
\item[(i)]
If $\cM_1, \cM_2\subset \HH^2(f)$ are invariant subspaces under the
operators $M_{Z_1},\ldots, M_{Z_n}$, then  the $n$-tuple
$(P_{\cM_1^\perp}M_{Z_1} |\cM_1^\perp,\ldots, P_{\cM_1^\perp}M_{Z_n}
|\cM_1^\perp)$ is equivalent to $(P_{\cM_2^\perp}M_{Z_1}
|\cM_2^\perp,\ldots, P_{\cM_2^\perp}M_{Z_n} |\cM_2^\perp)$ if and
only if $\cM_1=\cM_2$.
\item[(ii)]
  If $\cM\subseteq  \HH^2(f)$ is an invariant subspace under $M_{Z_1},\ldots, M_{Z_n}$, and
$$
T:=(T_1,\ldots, T_n),\qquad T_i:=P_{\cM^\perp}M_{Z_i} |\cM^\perp,
\qquad i=1,\ldots, n,
$$
then
$
\cM=\Theta_{f,T}\left(\HH^2(f)\otimes \cD_{f,T^*}\right),
$
where $\Theta_{f,T}$ is the  characteristic function of $T$.
 \end{enumerate}
 \end{theorem}
\begin{proof}
Assume the hypotheses of item (ii). Since $f=(f_1,\ldots, f_n)$ is
an $n$-tuple of formal power series with
 the model property, $M_{f_i}=f_i(M_{Z_1},\ldots, M_{Z_n})$, where
$(M_{Z_1},\ldots, M_{Z_n})$ is either in the convergence set
$\cC_f^{SOT}(\HH^2(f))$ or $\cC_f^{rad}(\HH^2(f))$. Since
$\cM^\perp$ is invariant under  $M_{Z_1}^*,\ldots, M_{Z_n}^*$, we
deduce that
\begin{equation*}
\begin{split}
\Delta_{f,T}&=I_{\cM^\perp}-\sum_{i=1}^nf_i(T_1,\ldots, T_n)f_i(T_1,\ldots, T_n)^*\\
&=P_{\cM^\perp}(I_{\HH^2(f)}-M_{f_i}M_{f_i}^*)|_{\cM^\perp}
=P_{\cM^\perp} P_\CC|_{\cM^\perp}.
\end{split}
\end{equation*}
Hence, $\rank \Delta_{f,T}\leq 1$. On the other hand, since $[M_{f_1},\ldots, M_{f_n}]$ is
 a pure row contraction, so is $[f_1(T),\ldots, f_n(T)]$. Therefore, $T$ is pure
 $n$-tuple in $\BB_f(\cM^\perp)$ and $\rank \Delta_{f,T}\neq 0$, which implies
$\rank \Delta_{f,T}= 1$. Therefore, we can identify  the subspace
$\cD_{f,T}$ with $\CC$. The Poisson kernel $K_{f,T}:\cM^\perp\to
\HH^2(f)\otimes \cD_{f,T}$ can be identified  with the injection of
$\cM^\perp$ into $\HH^2(f)$, via a unitary operator from
$\HH^2(f)\otimes \cD_{f,T}$ to $\HH^2(f)$. Indeed, note that if
$\sum_{\alpha}c_\alpha f_\alpha\in \cM^\perp\subset \HH^2(f)$, then,
taking into account that $\Delta_{f,T}=P_{\cM^\perp}
P_\CC|_{\cM^\perp}$ and $[f(T)]_\alpha=P_{\cM^\perp}
M_{f_\alpha}|_{\cM^\perp}$, we have
\begin{equation*}
\begin{split}
K_{f,T}\left( \sum_{\alpha}c_\alpha f_\alpha\right) &=\sum_{\beta\in
\FF_n^+}f_\beta\otimes P_{\cM^\perp} P_\CC|_{\cM^\perp}
M_{f_\beta}^*\left( \sum_{\alpha}c_\alpha
f_\alpha\right)=\sum_{\beta\in \FF_n^+}c_\beta f_\beta \otimes
P_{\cM^\perp}(1),
\end{split}
\end{equation*}
 which implies our assertion. As a consequence, we deduce that  the $n$-tuple $(T_1,\ldots, T_n)$ is unitarily equivalent to $(K_{f,T}^* M_{Z_1} K_{f,T},\ldots, K_{f,T}^* M_{Z_n} K_{f,T})$. Due to Theorem \ref{pure-dilation}, the $n$-tuple
 $(M_{Z_1},\ldots, M_{Z_n})$ is  the minimal dilation of $(T_1,\ldots, T_n)$.

 Now, using this result under the hypotheses of item (i) and the uniqueness of the minimal dilation
  (see Theorem \ref{pure-dilation}), we obtain that
 the $n$-tuple
$(P_{\cM_1^\perp}M_{Z_1} |\cM_1^\perp,\ldots, P_{\cM_1^\perp}M_{Z_n}
|\cM_1^\perp)$ is equivalent to $(P_{\cM_2^\perp}M_{Z_1}
|\cM_2^\perp,\ldots, P_{\cM_2^\perp}M_{Z_n} |\cM_2^\perp)$ if and
only if
there exists a unitary operator $W:\HH^2(f)\to \HH^2(f)$ such that
$WM_{Z_i}=M_{Z_i}W$, $i=1,\ldots,n$, and $W(\cM_1^\perp)=\cM_2^\perp$.
 Hence we deduce that  $WM_{Z_i}^*=M_{Z_i}^*W$, $i=1,\ldots,n$.
  Since $C^*(M_{Z_1},\ldots, M_{Z_n})$ is irreducible (see Theorem \ref{irreducible}), $W$ is a scalar multiple of the identity. Therefore, we must have
$\cM_1=\cM_2$, which proves part (i).

To prove part (ii), note that, due to Theorem \ref{funct-model}, we
have
$$
{\bf H}_{f,T}=\HH^2(f)\ominus \Theta_{f,T}\left(\HH^2(f)\otimes \cD_{f,T^*}\right)
$$ and  $T=(T_1,\ldots, T_n)$ is unitarily equivalent to
$
\left(P_{{\bf H}_{f,T}}M_{Z_1}|_{{\bf H}_{f,T}},\ldots, P_{{\bf
H}_{f,T}}M_{Z_n}|_{{\bf H}_{f,T}}\right).
$
Using part (i), we deduce that $ {\bf H}_{f,T}=\cM^\perp$ and therefore
$\cM=\Theta_{f,T}\left(\HH^2(f)\otimes \cD_{f,T^*}\right)$.
This completes the proof.
\end{proof}

{\bf The commutative case.}  Assume that   $f=(f_1,\ldots, f_n)$ has the model property. According to Theorem \ref{symm-Hardy}
and Theorem \ref{W(L)}, if $J_c$ is the WOT-closed two-sided ideal
of the Hardy algebra $H^\infty(\BB_f)$ generated by the commutators
$$ M_{Z_i}M_{Z_j}-M_{Z_j}M_{Z_i},\qquad i,j=1,\ldots, n,
$$
then $\cN_{J_c}=\HH^2_s(f)$, the symmetric Hardy  space associated
with $\BB_f$. Moreover, $\HH^2_s(f)$   can be identified with the
Hilbert space $H^2(\BB_{f}^<(\CC))$ of  holomorphic functions on
$\BB_{f}^<(\CC)$,
 namely, the reproducing kernel Hilbert space
with reproducing kernel $\Lambda_f:\BB_{f}^<(\CC)\times
\BB_{f}^<(\CC)\to \CC$ defined by
$$
\Lambda_f(\mu,\lambda):=\frac{1}{1-\sum_{ i=1}^n   f_i(\mu)
\overline{f_i(\lambda)}},\qquad \lambda,\mu\in \BB_{f}^<(\CC).
$$
 The algebra $P_{H_s^2(f)} \HH^\infty(\BB_f)|_{\HH_s^2(f)}$
  coincides with the WOT-closed algebra  generated by
    the operators $L_i:=P_{\HH_s^2(f)} M_{Z_i}|_{\HH_s^2(f)}$, $i=1,\dots, n$, and
 can be identified
 with the algebra of  all multipliers of the  Hilbert space $H^2(\BB_{f}^<(\CC))$.
 Under this identification the  operators $L_1,\ldots, L_n$ become the
  multiplication operators $M_{z_1},\ldots, M_{z_n}$ by the coordinate
   functions $z_1,\ldots, z_n$, respectively.
Now, let $T:=(T_1,\ldots, T_n)\in \BB_f(\cH)$ be  such that $
T_iT_j=T_jT_i, \quad i,j=1,\ldots, n. $ Under the above-mentioned
identifications, we define the characteristic function of $T$ to be
  the multiplier
$\Theta_{f,J_c,T}:H^2(\BB_{f}^<(\CC))\otimes \cD_{f,T^*}\to
H^2(\BB_{f}^<(\CC))\otimes \cD_{f,T} $ given by    the
operator-valued analytic function on  $\BB_{f}^<(\CC)$
\begin{equation*}
\begin{split}
\Theta_{f,J_c,T}(z):=
  - f(T)+
 \Delta_{f,T}\left(I-\sum_{i=1}^n f_i(z) f_i(T)^*\right)^{-1}
\left[f_1(z)I_\cH,\ldots, f_n(z) I_\cH \right] \Delta_{f,T^*}
\end{split}
\end{equation*}
for  $z=(z_1,\ldots, z_n)\in \BB_{f}^<(\CC)$. All the results  of
this section can be written in this commutative setting.

\section{Curvature invariant on $\BB_f(\cH)$}

In this section,  we
     introduce a curvature invariant on  the noncommutative domain $\BB_f(\cH)$ and show that it is
         a complete numerical
      invariant for the finite rank submodules of the free $\BB_f$-Hilbert module $\HH^2(f)\otimes \cK$,
      where $\cK$ is finite dimensional. We also provide an index type formula for the curvature in terms
of the characteristic function.

Let $f=(f_1,\ldots, f_n)$ be  an  $n$-tuple of formal power series
with
 the model property   and  let $(M_{Z_1},\ldots, M_{Z_n})$ be the universal
 model associated with  the noncommutative domain $\BB_f$.
Let $T=(T_1,\ldots, T_n)\in \BB_f(\cH)$  be such that
$$\rank_f(T):= \rank\left(I-\sum_{i=1}^n f_i(T)f_i(T)^*\right)^{1/2}<\infty.
$$
We define the  curvature of $T$    by setting
$$
\text{\rm curv}_f(T):=\lim_{m\to \infty} \frac{\text{\rm
trace}~[K_{f,T}^*(Q_{\leq m}\otimes I_{\cD_{f,T}})K_{f,T}]} {
\text{\rm trace}~[K_{f,M_Z}^*(Q_{\leq m})K_{f,M_Z}]},
$$
where $Q_{\leq m}$, $m=0,1,\ldots,$ is  the orthogonal projection of
$\HH^2(f)$ on the linear span of the formal power series $f_\alpha$,
$\alpha\in \FF_n^+$ with $|\alpha|\leq m$. In what follows, we show
that the limit exists and we provide a formula for the curvature in
terms of the characteristic function. We denote by $Q_{ m}$,
$m=0,1,\ldots,$  the orthogonal projection of $\HH^2(f)$ on the
linear span of the formal power series $f_\alpha$, $\alpha\in
\FF_n^+$ with $|\alpha|= m$.

\begin{theorem}\label{curve} Let $f=(f_1,\ldots, f_n)$ be  an  $n$-tuple of formal power series
with
 the model property  and let $T=(T_1,\ldots, T_n)\in \BB_f(\cH)$  be such that
$\rank_f(T)<\infty$. Then
$$
\text{\rm curv}_f(T)=\rank_f(T)-\text{\rm trace}~[\Theta_{f,T}
 (Q_0\otimes I_{\cD_{f,T^*}}) \Theta_{f,T}^*N],
$$
where $\Theta_{f,T}$ is the characteristic function of $T$ and
$$
N:=\sum_{k=0}^\infty\frac{1}{n^k} Q_k\otimes I_{\cD_{f,T}}.
$$
\end{theorem}
\begin{proof} Since
$$\text{\rm trace}~[K_{f,M_Z}^*(Q_{\leq m})K_{f,M_Z}]=
\text{\rm trace}~[Q_{\leq m}]=1+n+\cdots+n^m,
$$
 we  can use Theorem \ref{factor} to deduce that
 \begin{equation*}
 \begin{split}
 \text{\rm curv}_f(T)&=\lim_{m\to\infty} \frac{\sum_{k=0}^m\text{\rm trace}~[K_{f,T}^*
 (Q_{ k}\otimes I_{\cD_{f,T}})K_{f,T}]}{1+n+\cdots+n^m}\\
 &=\lim_{m\to\infty} \frac{\text{\rm trace}~[K_{f,T}^*(Q_{ m}\otimes I_{\cD_{f,T}})K_{f,T}]}{n^m}\\
 &=\lim_{m\to\infty}
  \frac{\text{\rm trace}~[(Q_{m}\otimes I_{\cD_{f,T}})K_{f,T}K_{f,T}^*](Q_{m}\otimes I_{\cD_{f,T}})}{n^m}\\
  &=\rank_f(T)-\lim_{m\to\infty}
  \frac{\text{\rm trace}~[(Q_{m}\otimes I_{\cD_{f,T}})\Theta_{f,T}\Theta_{f,T}^*(Q_{m}\otimes I_{\cD_{f,T}})]}{n^m},
 \end{split}
 \end{equation*}
 provided the latter limit exists, which we should prove now.
 Since $\Theta_{f,T}$ is a multi-analytic operator with respect to $M_{Z_1},\ldots, M_{Z_n}$ and
 $$
 \sum_{k=0}^\infty \sum_{|\alpha|=k} M_{f_\alpha} Q_0 M_{f_\alpha}^*=I_{\HH^2(f)},
 $$
 we deduce that
 $$
 (Q_m\otimes I_{\cD_{f,T}})\Theta_{f,T}\Theta_{f,T}^*(Q_m\otimes I_{\cD_{f,T}})
 = \sum_{k=0}^m \sum_{|\alpha|=k} (Q_m M_{f_\alpha}\otimes I) \Theta_{f,T}(Q_0\otimes I_{\cD_{f,T^*}})
  \Theta_{f,T}^*( M_{f_\alpha}^*Q_m\otimes I).
 $$
 Hence, and taking into account that $\sum_{|\alpha|\leq m}M_{f_\alpha}^*Q_m M_{f_\alpha}=\sum_{k=0}^m n^kQ_{m-k}$, we obtain
 \begin{equation*}
 \begin{split}
 \frac{\text{\rm trace}~[(Q_{m}\otimes I_{\cD_{f,T}})\Theta_{f,T}\Theta_{f,T}^*
 (Q_{m}\otimes I_{\cD_{f,T}})]}{n^m}
 &=
  \frac{\text{\rm trace}~ [(\Theta_{f,T}(Q_0\otimes I_{\cD_{f,T^*}}) \Theta_{f,T}^*)( \sum_{|\alpha|\leq m}
  M_{f_\alpha}^*Q_m M_{f_\alpha}\otimes I_{\cD_{f,T}})]}{n^m}\\
  &=\text{\rm trace}~ [\Theta_{f,T} (Q_0\otimes I_{\cD_{f,T^*}}) \Theta_{f,T}^*N_m],
 \end{split}
 \end{equation*}
 where $
N_m:=\sum_{k=0}^m\frac{1}{n^k} Q_k\otimes I_{\cD_{f,T}}.
$
Consequently, we have
\begin{equation*}
 \begin{split}
0&\leq \text{\rm trace}~ [\Theta_{f,T} (Q_0\otimes I_{\cD_{f,T^*}}) \Theta_{f,T}^*N_m]
\leq \frac{\text{\rm trace}~[(Q_{m}\otimes I_{\cD_{f,T}})\Theta_{f,T}\Theta_{f,T}^*(Q_{m}\otimes I_{\cD_{f,T}})]}{n^m}\\
&\leq \|\Theta_{f,T}\|^2\dim \cD_{f,T}=\dim \cD_{f,T}<\infty.
\end{split}
 \end{equation*}
 Since $\{N_m\}$ is an increasing sequence of positive operators convergent to $N$, we deduce that
 \begin{equation*}
 \begin{split}
 \text{\rm trace}~[\Theta_{f,T} (Q_0\otimes I_{\cD_{f,T^*}}) \Theta_{f,T}^*N]
 &=\lim_{m\to\infty} \text{\rm trace}~[\Theta_{f,T} (Q_0\otimes I_{\cD_{f,T^*}}) \Theta_{f,T}^*N_m].
 \end{split}
 \end{equation*}
 Combining this result with the relations above, we complete the proof.
 \end{proof}
We remark that the proof of Theorem \ref{curve} is simpler than that
of the corresponding result from \cite{Po-curvature}, in the
particular case when $f=(Z_1,\ldots, Z_n)$.

\begin{corollary}\label{curv-row}  Let $f=(f_1,\ldots, f_n)$ be  an  $n$-tuple of formal power series
with
 the model property.
If $T=(T_1,\ldots, T_n)\in \BB_f(\cH)$   and $\rank_f(T)<\infty$,
then
$$
\text{\rm curv}_f(T)=\lim_{m \to \infty}\frac{\text{\rm trace}~[I-\Phi_{f,T}^{m+1}(I)]}
{1+n+\cdots+n^{m}}=\text{\rm curv}(f(T)),
$$
where the  $\Phi_{f,T}(Y):=\sum_{i=1}^n f_i(T)Yf_i(T)^*$ and $\text{\rm curv}(f(T))$ is the curvature of the row contraction $f(T)$.
\end{corollary}
\begin{proof} Due to the properties of the noncommutative Poisson Kernel $K_{f,T}$, we have
\begin{equation*}
\begin{split}
K_{f,T}^*\left(\sum_{|\alpha|=k} M_{f_\alpha} M_{f_\alpha}^*\otimes I\right)K_{f,T}
&=\sum_{|\alpha|=k} [f(T)]_\alpha K_{f,T}^*K_{f,T} [f(T)]_\alpha^*\\
&=\sum_{|\alpha|=k} [f(T)]_\alpha [f(T)]_\alpha^* -\Phi_{f,T}^\infty(I),
\end{split}
\end{equation*}
where $\Phi_{f,T}^\infty(I):=\text{\rm SOT-}\lim_{k\to\infty}\Phi_{f,T}^k(I)$. Consequently, we obtain
$$
K_{f,T}^* (Q_m\otimes I)K_{f,T}=\Phi_{f,T}^m\left(I-\sum_{i=1}^n f_i(T)f_i(T)^*\right).
$$
Now, using  the equalities from the  proof of Theorem \ref{curve}, the result follows.
\end{proof}

\begin{theorem}  Let $f=(f_1,\ldots, f_n)$ be  an  $n$-tuple of formal power series
with
 the model property   and  let $(M_{Z_1},\ldots, M_{Z_n})$ be
the universal model associated  with the noncommutative domain
$\BB_f$.  If   an $n$-tuple $T=(T_1,\ldots, T_n)\in \BB_f(H)$ is
such that $\rank_f(T)<\infty$, then $T$ is unitarily equivalent to
the $n$-tuple $(M_{Z_1}\otimes I_\cK,\ldots, M_{Z_n}\otimes I_\cK)$
with $\dim \cK<\infty$  if and only if $T$ is pure and
$$\text{\rm
Curv}_f(T)=\rank_f(T).
$$
\end{theorem}
\begin{proof}
 Assume  that $T:=(T_1,\dots, T_n)\in \BB_f(H)$ is
unitarily equivalent to $( M_{Z_1}\otimes I_{ \cK}, \dots,
M_{Z_n}\otimes I_{ \cK})$, where  $\dim \cK<\infty$. Note that due
to the fact that $f=(f_1,\ldots, f_n)$  has
 the model property, we have
\begin{equation*}
\begin{split}
 \rank_f(T)&= \rank\left(I-\sum_{i=1}^n f_i(M_{Z_1}\otimes I_\cK,\ldots, M_{Z_n}\otimes I_\cK)
 f_i(M_{Z_1}\otimes I_\cK,\ldots,
 M_{Z_n}\otimes I_\cK)^*\right)^{1/2}\\
 &=\rank \left(I-\sum_{i=1}^n (M_{f_i}\otimes
 I_\cK)(M_{f_i}\otimes
 I_\cK)^*\right)
  =\dim \cK.
\end{split}
\end{equation*}
On the other hand,  according to the definition of the curvature, we have
\begin{equation*}
\begin{split}
\text{\rm Curv}_f(T) = \lim_{m\to \infty}
\frac{\text{\rm trace}~[K_{f,M_Z\otimes I_\cK}^*(Q_{\leq m}\otimes I_\cK)K_{f,M_Z\otimes I_\cK}]} {
\text{\rm trace}~[K_{f,M_Z}^*(Q_{\leq m})K_{f,M_Z}]}
 =\dim \cK.
 \end{split}
\end{equation*}

Conversely, assume that $T$ is pure and $\text{\rm
Curv}_f(T)=\rank_f(T).$  According to Theorem \ref{factor},
$$
K_{f,T} K_{f,T}^* =I_{\HH^2(f)\otimes \cD_{f,T}}-\Theta_{f,T}
\Theta_{f,T}^*
$$
where $\Theta_{f,T}$ is the characteristic function associated with
$T$. Since the noncommutative Poison kernel $K_{f,T}$  is an
isometry,  $\Theta_{f,T}$ is an inner multi-analytic operator.
 On the other hand, Theorem \ref{curve} implies
$$
\text{\rm Curv}_f(T)=\rank_f(\cH)- \text{\rm
trace}~[\Theta_{f,T} (Q_0\otimes I_{\cD_{f,T^*}}) \Theta_{f,T}^* N],
$$
where $N$ is the number operator. Therefore,
$\text{\rm trace}[\Theta_{f,T} (Q_0\otimes I_{\cD_{f,T^*}})
\Theta_{f,T}^* N]=0$.
 Since trace is faithful, we obtain  $\Theta_{f,T} (Q_0\otimes I_{\cD_{f,T^*}})\Theta_{f,T}^*Q_j=0$
 for any $j=0,1,\dots.$
This implies
 $\Theta_{f,T} (Q_0\otimes I_{\cD_{f,T^*}}) \Theta_{f,T}^*=0$. Taking into account that $\Theta_{f,T} $
 is an isometry, we infer that
 $\Theta_{f,T} (Q_0\otimes I_{\cD_{f,T^*}}) =0$. Since  $\Theta_{f,T} $ is multi-analytic
  with respect to $M_{Z_1},\ldots M_{Z_n}$, and $\CC[Z_1,\ldots,
  Z_n]$ is dense in $\HH^2(f)$,
  we deduce $\Theta_{f,T}=0$. Using again  the fact that $
K_{f,T}K_{f,T}^*+\Theta_{f,T}\Theta_{f,T}^*= I_{\HH^2(f)\otimes
{\cD_{f,T}}}$, we   deduce that  $K_{f,T}:\cH\to \HH^2(f)\otimes
{\cD_{f,T}}$ is a unitary operator. According to the properties of
the Poisson kernel, we have
 $$K_{f,T}^* (M_{Z_i}\otimes I_{  {\cD_{f,T}}}) K_{f,T}=T_i, \qquad
i=1,\dots, n. $$ This shows that   the $n$-tuple $(T_1,\ldots, T_n)$
is unitarily equivalent to $(M_{Z_1}\otimes I_{\cD_{f,T})},\ldots,
M_{Z_n}\otimes I_{\cD_{f,T}})$ and $\dim \cD_{f,T}<\infty$. This
completes the proof.
\end{proof}

In what follows we show that
the curvature  on $\BB_f(\cH)$ is  a complete numerical
      invariant for the finite rank submodules of the   $\BB_f$-Hilbert module $\HH^2(f)\otimes \cK$,
      where $\cK$ is finite dimensional.

\begin{theorem}  Let $f=(f_1,\ldots, f_n)$ be  an  $n$-tuple of formal power series
with
 the model property   and  let $M_Z:=(M_{Z_1},\ldots, M_{Z_n})$ be
the universal model associated  with the noncommutative domain
$\BB_f$.
 Given  $\cM,
\cN\subseteq \HH^2(f)$ two invariant subspaces under
$M_{Z_1},\ldots, M_{Z_n}$,  the following statements hold.
\begin{enumerate}
\item[(i)] If $\rank_f(M_Z|_\cM)<\infty$, then $\text{\rm
Curv}_f(M_Z|_\cM)=\rank_f(M_Z|_\cM)$.
\item[(ii)] If $\rank_f(M_Z|_\cM)<\infty$ and $\rank_f(M_Z|_\cN)<\infty$,
then $M_Z|_\cM$ is unitarily equivalent to $M_Z|_\cN$ if and only if
$$
\text{\rm Curv}_f(M_Z|_\cM)=\text{\rm Curv}_f(M_Z|_\cN).
$$
\end{enumerate}
\end{theorem}
\begin{proof}
Let $g=(g_1,\ldots, g_n)$ be the inverse of $f=(f_1,\ldots,f_n)$ with respect to the composition of formal power series. Since $f$ has the model property, we have
$$
f_i(M_{Z_1},\ldots, M_{Z_n})=M_{f_i}\quad \text{ and }\quad g_i(M_{f_1},\ldots, M_{f_n})=M_{Z_i}
$$
for any $i=1,\ldots,n$. Hence, we deduce that a subspace $\cM$ is invariant under
$M_{Z_1},\ldots, M_{Z_n}$ if and only if it is invariant under $M_{f_1},\ldots, M_{f_n}$. We recall that $M_{f_i}=U^{-1} S_i U$ , $i=1,\ldots,n$, where
$U:\HH^2(f)\to F^2(H_n)$ is the unitary operator defined by $U(f_\alpha)=e_\alpha$, $\alpha\in \FF_n^+$, and $S_1,\ldots, S_n$ are the left creation operators. Now, one can easily see that $\cM$ is an invariant subspace  under
$M_{Z_1},\ldots, M_{Z_n}$ if and only if $U\cM$ is invariant under $S_1,\ldots, S_n$.
Hence, using  Corollary \ref{curv-row} and  the fact that $UP_\cM U^{-1}=P_{U\cM}$, we have
\begin{equation*}
\begin{split}
\rank_f(M_Z|_\cM)&=\rank \left(f_1(M_Z|_\cM),\ldots, f_n(M_Z|_\cM)\right)\\
&=\rank\left( U^{-1}S_1 U|_\cM,\ldots,U^{-1}S_n U|_\cM\right)\\
&=\rank \left(S_1|_{U\cM},\ldots, S_n|_{U\cM}\right)
\end{split}
\end{equation*}
and
\begin{equation*}
\begin{split}
\text{\rm Curv}_f(M_Z|_\cM)&=\text{\rm Curv} \left(f_1(M_Z|_\cM),\ldots, f_n(M_Z|_\cM)\right)\\
&=\text{\rm Curv}\left( U^{-1}S_1 U|_\cM,\ldots,U^{-1}S_n U|_\cM\right)\\
&=\text{\rm Curv} \left(S_1|_{U\cM},\ldots, S_n|_{U\cM}\right)
\end{split}
\end{equation*}
According to  Theorem 3.2 from
\cite{Po-curvature}, we have
$$
\rank \left(S_1|_{U\cM},\ldots, S_n|_{U\cM}\right)
=
\text{\rm Curv} \left(S_1|_{U\cM},\ldots, S_n|_{U\cM}\right).
$$
Combining the results above, we deduce item (i).
To prove part (ii), note that the direct implication is due to the fact that, for any
$T=(T_1,\ldots, T_n)\in \BB_f(\cH)$ and $T'=(T_1',\ldots, T_n')\in \BB_f(\cH')$,
if $T$ is unitarily equivalent to $T'$, then  $\text{\rm Curv}_f(T)=\text{\rm Curv}_f(T')$.
Conversely, assume that
 $
\text{\rm Curv}_f(M_Z|_\cM)=\text{\rm Curv}_f(M_Z|_\cN).
$
As shown above, the latter equality is equivalent to
$$
\text{\rm Curv} \left(S_1|_{U\cM},\ldots, S_n|_{U\cM}\right)=\text{\rm Curv} \left(S_1|_{U\cN},\ldots, S_n|_{U\cN}\right).
$$
Applying again Theorem 3.2 from
\cite{Po-curvature}, we find a unitary operator $W:U\cM\to U\cN$ such that
$$
W(S_i|_{U\cM})=(S_i|_{U\cN}) W,\qquad i=1,\ldots,n.
$$
Consequently, we have
$$
(U^{-1}WU|_\cM)(U^{-1}S_iU|_\cM)=(U^{-1}S_iU|_\cN) (U^{-1}WU|_\cM),\qquad i=1,\ldots,n,
$$
which implies
$$
(U^{-1}WU|_\cM)(M_{f_i}|_\cM)=(M_{f_i}|_\cN) (U^{-1}WU|_\cM),\qquad i=1,\ldots,n.
$$
Using now relation $g_i(M_{f_1},\ldots, M_{f_n})=M_{Z_i}$, $i=1,\ldots,n$, we obtain
$$
(U^{-1}WU|_\cM)(M_{Z_i}|_\cM)=(M_{Z_i}|_\cN) (U^{-1}WU|_\cM),\qquad i=1,\ldots,n.
$$
Since $U^{-1}WU|_\cM:\cM\to \cN$ is a unitary operator, we conclude that the $n$-tuples $(M_{Z_1}|_\cM,\ldots, M_{Z_n}|_\cM)$ and $(M_{Z_1}|_\cN,\ldots, M_{Z_n}|_\cN)$ are unitarily equivalent.
The proof is complete.
\end{proof}

We remark that all the results  of this section have  commutative
versions when $T=(T_1,\ldots, T_n)\in \BB_f(\cH)$, $T_iT_j=T_jT_i$,
and the universal model $(M_{Z_1},\ldots, M_{Z_n})$ is replaced by the
$n$-tuple $(L_1,\ldots,L_n)$, where $L_i:=P_{\HH_s^2(f)}
M_{Z_i}|_{\HH_s^2(f)}$, $i=1,\dots, n$, and $\HH_s^2(f)$ is the
symmetric  Hardy space associated with the noncommutative domain
$\BB_f$. In this case, we  obtain  analogues of   Arveson's results
\cite{Arv2}  concerning the curvature for commuting row
contractions,  for  the set of commuting $n$-tuples in the domain
$\BB_f(\cH)$.

\bigskip

\section{Commutant lifting and interpolation}

In this section, to provide a commutant lifting theorem for the pure
$n$-tuples of operators in the noncommutative domain $\BB_f(\cH)$
and  solve the  Nevanlinna Pick  interpolation problem  for the
noncommutative Hardy algebra $\HH^\infty(\BB_f^<)$.

First, we present a Sarason  \cite{S} type commutant lifting result
in our setting.

\begin{theorem}\label{CLT} Let $f=(f_1,\ldots, f_n)$ be  an  $n$-tuple of formal power series with
the model  property  and  let $(M_{Z_1},\ldots, M_{Z_n})$ be the
universal model associated with $\BB_f$. Let $\cE_j\subset
\HH^2(f)\otimes \cK_j$, $j=1,2$, be a co-invariant subspace under
each operator $M_{Z_i}\otimes I_{\cK_j}$, $i=1,\ldots,n$.

 If
$X:\cE_1\to \cE_2$ is a bounded operator such that
$$
X[P_{\cE_1}(M_{Z_i}\otimes
I_{\cK_1})|_{\cE_1}]=[P_{\cE_2}(M_{Z_i}\otimes
I_{\cK_2})|_{\cE_2}]X,\qquad i=1,\ldots,n,
$$
then there exists a bounded operator  $Y:\HH^2(f)\otimes \cK_1\to
\HH^2(f)\otimes \cK_2$ with the property
$$
Y( M_{Z_i}\otimes I_{\cK_1})= (M_{Z_i}\otimes I_{\cK_2}) Y,\qquad
i=1,\ldots,n,
$$
and   such that $Y^*\cE_2\subseteq \cE_1$, $ Y^*|\cE_2=X^*$, and
 $ \|Y\|=\|X\|$.
\end{theorem}
\begin{proof} Setting $A_i:=P_{\cE_1}(M_{Z_i}\otimes
I_{\cK_1})|_{\cE_1}$ and $B_i:=P_{\cE_2}(M_{Z_i}\otimes
I_{\cK_2})|_{\cE_2}$, we have $XA_i=B_iX$, $i=1,\ldots,n$. Since $f$
has the model property   and $\cE_1$ is a co-invariant subspace
under each operator $M_{Z_i}\otimes I_{\cK_j}$, $i=1,\ldots,n$, we
deduce
 that $\cE_1$ is a co-invariant subspace under
$f_i(M_{Z_1},\ldots,M_{Z_n})\otimes I_{\cK_1}=M_{f_i}\otimes
I_{\cK_1}$ and
$$
f_i(A_1,\ldots, A_n)=P_{\cE_1}[f_i(M_{Z_1},\ldots,M_{Z_n})\otimes
I_{\cK_1}]|_{\cE_1}=P_{\cE_1}( M_{f_i}\otimes
I_{\cK_1})|_{\cE_1},\qquad i=1,\ldots,n.
$$
Similarly, $\cE_2$ is a co-invariant subspace under
$f_i(M_{Z_1},\ldots,M_{Z_n})\otimes I_{\cK_2}=M_{f_i}\otimes
I_{\cK_2}$ and $$f_i(B_1,\ldots,
B_n)=P_{\cE_2}[f_i(M_{Z_1},\ldots,M_{Z_n})\otimes
I_{\cK_2}]|_{\cE_2}=P_{\cE_2}( M_{f_i}\otimes
I_{\cK_2})|_{\cE_2},\qquad i=1,\ldots,n.$$ Using the canonical
unitary operator $U:\HH^2(f)\to F^2(H_n)$, defined by
$U(f_\alpha)=e_\alpha$, $\alpha\in \FF_n^+$, we have
$M_{f_i}=U^*S_iU$ and the subspace  $U(\cE_1)$ is co-invariant under
 $S_1\otimes I_{\cK_1},\ldots, S_n\otimes I_{\cK_1}$,
where $S_1,\ldots, S_n$ are the left creation operators on
$F^2(H_n)$. Similarly,  we have that $U(\cE_2)$ is co-invariant
under each operator $S_1\otimes I_{\cK_2},\ldots, S_n\otimes
I_{\cK_2}$. Now, since $XA_i=B_iX$, we deduce that $
Xf_i(A_1,\ldots, A_n)=f_i(B_1,\ldots, B_n)X$, $ i=1,\ldots,n, $
which together with the considerations above imply
$$
\widetilde X\left[P_{U(\cE_1)}(S_i\otimes
I_{\cK_1})|_{U(\cE_1)}\right]=\left[P_{U(\cE_2)}(S_i\otimes
I_{\cK_2})|_{U(\cE_2)}\right]\widetilde X,\qquad i=1,\ldots,n,
$$
where $\widetilde X:U(\cE_1)\to U(\cE_2)$ is defined by $\widetilde
X:=UXU^*|_{U(\cE_1)}$. Note that $[S_1\otimes I_{\cK_1},\ldots,
S_n\otimes I_{\cK_1}]$ is an isometric dilation of the row
contraction $[P_{U(\cE_1)}(S_1\otimes I_{\cK_1})|_{U(\cE_1)},\ldots,
P_{U(\cE_1)}(S_n\otimes I_{\cK_1})|_{U(\cE_1)}]$. Applying the
noncommutative commutant lifting theorem  from \cite{Po-isometric},
we find a bounded operator $\widetilde Y:F^2(H_n)\otimes \cK_1\to
F^2(H_n)\otimes \cK_2$ with the properties $\widetilde Y(S_i\otimes
I_{\cK_1})=(S_i\otimes I_{\cK_2}) \widetilde Y$ for  $i=1,\ldots,n$,
$\widetilde Y^*(U(\cE_2))\subset U(\cE_1)$, $\widetilde
Y^*|_{U(\cE_2)}=\widetilde X^*$, and $\|\widetilde Y\|=\|\widetilde
X\|$. Now, setting $Y:=U\widetilde Y U^*$, we deduce that
$Y:\HH^2(f)\otimes \cK_1\to \HH^2(f)\otimes \cK_2$  has  the
  property
$$
Y( M_{f_i}\otimes I_{\cK_1})= (M_{f_i}\otimes I_{\cK_2}) Y,\qquad
i=1,\ldots,n,
$$
and   also satisfies the relations   $Y^*\cE_2\subseteq \cE_1$, $
Y^*|\cE_2=X^*$, and
 $ \|Y\|=\|X\|$. Once again, taking into account that  $f$ has the
 model
 property, we have
 $M_{Z_i}=g_i(M_{f_1},\ldots, M_{f_n})$ for $i=1,\ldots$, where
 $g=(g_1,\ldots, g_n)$ is the inverse of $f=(f_1,\ldots, f_n)$ with
 respect to the composition, and
$g_i(M_{f_1},\ldots, M_{f_n})$ is defined
 using  the
 radial SOT-convergence. Consequently, the above-mentioned
 intertwining relation implies
 $$
Y( M_{Z_i}\otimes I_{\cK_1})=Y[g_i(M_{f_1},\ldots, M_{f_n})\otimes
I_{\cK_1}]=[g_i(M_{f_1},\ldots, M_{f_n})\otimes
I_{\cK_2}]Y=(M_{Z_i}\otimes I_{\cK_2}) Y
$$
for $i=1,\ldots,n$. The proof is complete.
\end{proof}

 Recall that, due to Theorem \ref{symm-Hardy} and Theorem \ref{eigenvalues}, we  have
 $\HH_s^2(f)=\overline{\text{\rm span}}\{\Gamma_\lambda: \ \lambda\in
 \BB_{f}^<(\CC)\}$   and
$M_{Z_i}^* \Gamma_\lambda=\overline{\lambda}_i \Gamma_\lambda$ far all  $i=1,\ldots, n$.
 This shows that    $\HH_s^2(f)$ is  a co-invariant  subspace under each  operator
  $M_{Z_1},\ldots, M_{Z_n}$. We remark that this  observation can be used together with
 Theorem \ref{CLT} to obtain a commutative version  of the latter theorem, when
$\cE_j\subset \HH_s^2(f)\otimes \cK_j$, $j=1,2$, are co-invariant
subspaces under each operator $L_i\otimes I_{\cK_j}$,
$i=1,\ldots,n$.

 Now we can obtain  the following  Nevanlinna-Pick \cite{N} interpolation
 result  in our setting.

\begin{theorem}\label{Nev} Let $f=(f_1,\ldots, f_n)$ be  an  $n$-tuple of formal power series with
the model  property.
 If  $\lambda_1,\ldots, \lambda_m$ are  $m$ distinct
points in    $\BB_{f}^<(\CC)$ and   $A_1,\ldots, A_m\in B(\cK)$,
then there exists $\Phi \in H^\infty(\BB_f)\bar\otimes B(\cK)$ such
that
$$\|\Phi \|\leq 1\quad \text{and }\quad
\Phi(\lambda_j)=A_j,\ j=1,\ldots, m,
$$
if and only if the operator matrix
\begin{equation*}
\left[  \frac{I_\cK-A_i A_j^*}{1-\sum_{ k=1}^n   f_k(\lambda_i)
\overline{f_k(\lambda_j)}}\right]_{m\times m}
\end{equation*}
is positive semidefinite.
\end{theorem}

\begin{proof}

  Let $\lambda_j:=(\lambda_{j1},\dots,\lambda_{jn})$, $j=1,\dots,m$, be $m$ distinct
points in    $\BB_{f}^<(\CC)$. Consider the formal power series
$$
\Gamma_{\lambda_j}:= \left(1-\sum_{i=1}^n
|f_i(\lambda_j)|^2\right)^{1/2}\sum_{\alpha\in \FF_n}
[\overline{f(\lambda_j)}]_\alpha\
  f_\alpha,\qquad j=1,\dots,m,
$$
and set $\Omega_{\lambda_j}:=(1- \sum_{ i=1}^n
 |f_i(\lambda_j) )^{-1/2} \Gamma_{\lambda_j}$. According to Theorem \ref{eigenvalues}, they satisfy the
equations
\begin{equation}
\label{MFA}
 M_{Z_i}^* \Gamma_{\lambda_j}=\overline{\lambda}_{ji}
\Gamma_{\lambda_j},\qquad i=1,\ldots, n.
\end{equation}
Note that the subspace $ \cM:=\text{span} \{ \Gamma_{\lambda_j}:\
j=1,\dots,m\}\subset \HH^2(f) $
 is invariant under $M_{Z_i}^*$  for any $i=1,\ldots,
 n$.
  Define  the operators $X_i\in B(\cM\otimes\cK)$ by
   setting  $X_i:=P_\cM M_{Z_i}|_\cM\otimes I_\cK$,
   $i=1,\ldots, n$. Since $f$ is one-to-one on $\BB_{f}^<(\CC)$,
   we deduce that $f(\lambda_1),\ldots, f(\lambda_m)$ are distinct points in $\BB_n$.
    Consequently, the formal power series
     $\Gamma_{\lambda_1},\dots, \Gamma_{\lambda_m}$ are linearly independent
      and we can define an operator
    $T\in  B(\cM\otimes\cK)$ by setting
    \begin{equation}\label{def-T*}
    T^*(\Gamma_{\lambda_j}\otimes h)=
\Gamma_{\lambda_j}\otimes A_j^* h
    \end{equation}
    for any $h\in \cK$ and   $j=1,\dots, k$.
    A simple calculation using  relations \eqref{MFA} and \eqref{def-T*}
    shows that  $TX_i=X_iT$ for  $i=1,\dots, n $.
 Since $\cM$ is  a co-invariant subspace  under  each operator $M_{Z_i}$,
       $ i=1,\dots,n$,  we can apply  Theorem
\ref{CLT}
       and find
        a bounded operator  $Y:\HH^2(f)\otimes \cK\to
\HH^2(f)\otimes \cK$ with the property
\begin{equation}\label{YMZ}
Y( M_{Z_i}\otimes I_{\cK})= (M_{Z_i}\otimes I_{\cK}) Y,\qquad
i=1,\ldots,n,
\end{equation}
and   such that
\begin{equation}\label{PHI*}
Y^* (\cM\otimes \cK)\subset \cM\otimes \cK,\quad
Y^*|\cM\otimes \cK=T^*,
\end{equation}
 and  $ \|Y\|=\|T\|$.
Due to relation \eqref{YMZ} and the fact that
$M_{f_i}=f_i(M_{Z_1},\ldots, M_{Z_n})$, we  deduce that  $Y(
M_{f_i}\otimes I_{\cK})= (M_{f_i}\otimes I_{\cK}) Y$, $
i=1,\ldots,n$, which implies
$$
(U\otimes I_\cK)Y(U^*\otimes I_\cK)(S_i\otimes I_\cK)=(S_i\otimes
I_\cK)(U\otimes I_\cK)Y(U^*\otimes I_\cK), \qquad i=1,\ldots,n,
$$
where $U:\HH^2(f)\to F^2(H_n)$ is the canonical unitary operator
defined by $U(f_\alpha):=e_\alpha$. Using the characterization of
the commutant of $\{S_i\otimes I_\cK\}_{i=1}^n$ (see
\cite{Po-analytic}), we deduce that $(U\otimes I_\cK)Y(U^*\otimes
I_\cK)\in \cR_n^\infty\bar\otimes B(\cK)$ and has a unique Fourier
representation $\sum_{\alpha\in \FF_n^+} R_\alpha\otimes
C_{(\alpha)}$, $C_{(\alpha)}\in B(\cK)$, that is,
$$
(U\otimes I_\cK)Y(U^*\otimes I_\cK)=\text{\rm SOT-}\lim_{r\to 1}
\sum_{k=0}^\infty \sum_{|\alpha|=k}r^{|\alpha|} R_\alpha\otimes
C_{(\alpha)}.
$$
Using the flipping unitary operator $W:F^2(H_n)\to F^2(H_n)$,
defined by $W(e_\alpha):=e_{\widetilde \alpha}$, where
$\widetilde\alpha$ is the reverse of $\alpha\in \FF_n^+$, we define
$\Phi(M_Z)\in H^\infty(\BB_f)\bar\otimes B(\cK)$ by  setting
\begin{equation}\label{YPhi}
\Phi(M_Z):=(U^*W^*U\otimes I_\cK)Y(U^*WU\otimes I_\cK).
\end{equation}
Note that
$$
\Phi(M_Z)=\text{\rm SOT-}\lim_{r\to 1} \sum_{k=0}^\infty
\sum_{|\alpha|=k}r^{|\alpha|} [f(M_Z)]_\alpha\otimes C_{(\alpha)}.
$$
Hence,  and using the equations $ M_{Z_i}^*
\Gamma_{\lambda_j}=\overline{\lambda}_{ji} \Gamma_{\lambda_j}$,
$i=1,\ldots, n$, we deduce that
\begin{equation*}
\begin{split}
\left<\Phi(M_Z)^*(\Omega_\lambda\otimes h), y\otimes h'\right>&=
\left<\Omega_\lambda\otimes h, \lim_{r\to 1} \sum_{k=0}^\infty
\sum_{|\alpha|=k}r^{|\alpha|} [f(M_Z)]_\alpha y\otimes C_{(\alpha)}
h'\right>\\\
&=\lim_{r\to 1}  \sum_{k=0}^\infty
\sum_{|\alpha|=k}r^{|\alpha|}\left<\Omega_\lambda\otimes h,
[f(M_Z)]_\alpha y\otimes C_{(\alpha)} h'\right>\\
&=\lim_{r\to 1}  \sum_{k=0}^\infty
\sum_{|\alpha|=k}r^{|\alpha|}\overline{[f(\lambda)]}_\alpha
\left<\Omega_\lambda, y\right>\left< h,
  C_{(\alpha)} h'\right>\\
  &=\left<\Omega_\lambda, y\right>\left<h, \lim_{r\to 1}  \sum_{k=0}^\infty
\sum_{|\alpha|=k}r^{|\alpha|}\overline{[f(\lambda)]}_\alpha
C_{(\alpha)}h'\right>\\
&=\left<\Omega_\lambda, y\right>\left<h,\Phi(\lambda)h'\right>
=\left<\Omega_\lambda\otimes \Phi(\lambda)^*h, y\otimes h'\right>
\end{split}
\end{equation*}
for any $\lambda\in \BB_f^<(\CC)$, $y\in \HH^2(f)$, and $h,h'\in
\cK$. Therefore,
\begin{equation}\label{Phi*-eq}
  \Phi(M_Z)^*(\Omega_\lambda\otimes
h)=\Omega_\lambda\otimes \Phi(\lambda)^*h.
\end{equation}
Hence, and using   relation \eqref{YPhi}, we can   show that
\begin{equation}\label{YGA}
Y^*(\Gamma_\lambda\otimes h)=\Gamma_\lambda\otimes \Phi(\lambda)^*h,
\qquad \lambda\in \BB_f^<(\CC),\  h,h'\in \cK.
\end{equation}
 Now, we  prove   that $\Phi(\lambda_j)=A_j$,\
$j=1,\ldots, k$, if and only if $$P_{\cM\otimes \cK}Y|_{\cM\otimes \cK}=T.
$$
 Indeed,   due to
relation  \eqref{YGA}, we have
\begin{equation*}
\begin{split}
\left<Y^*(\Gamma_{\lambda_j}\otimes x),
\Gamma_{\lambda_j}\otimes y\right>&= \left<\Phi(M_Z)^*(\Gamma_{\lambda_j}\otimes x), \Gamma_{\lambda_j}\otimes y\right>\\
&=\left<\Gamma_{\lambda_j}\otimes \Phi(\lambda_j)^*x,
\Gamma_{\lambda_j}\otimes
y\right>\\
&= \left<\Gamma_{\lambda_j},
\Gamma_{\lambda_j}\right>\left<\Phi(\lambda_j)^*x,y\right>.
\end{split}
\end{equation*}
On the other hand, relation \eqref{def-T*} implies
$$\left<T^*(\Gamma_{\lambda_j}\otimes x), \Gamma_{\lambda_j}\otimes y\right>=
\left<\Gamma_{\lambda_j}, \Gamma_{\lambda_j}\right>\left<A_j^*x,y\right>.
$$
Due to Theorem \ref{eigenvalues}, we have
$$\left<\Omega_{\lambda_j}, \Omega_{\lambda_j}\right>=
\Lambda_f(\lambda_j, \lambda_i)= \frac{1}{1-\sum_{ k=1}^n
f_k(\lambda_i) \overline{f_k(\lambda_j)}}\neq 0$$
 for any $j=1,\ldots, k$.
Consequently,  the above relations imply  our assertion.

 Now, since $ \|Y\|=\|T\|$, it is clear that
 $\|Y\|\leq 1$ if and only if $TT^*\leq I_\cM$.
Note that, for any $h_1,\ldots, h_k\in \cK$, we have
 \begin{equation*}
 \begin{split}
 \left<\sum_{j=1}^k \Omega_{\lambda_j}\otimes h_j, \sum_{j=1}^k \Omega_{\lambda_j}\otimes
 h_j\right>&-\left<T^*\left(\sum_{j=1}^k \Omega_{\lambda_j}\otimes h_j\right),
  T^*\left(\sum_{j=1}^k \Omega_{\lambda_j}\otimes
 h_j\right)\right>\\
 &=\sum_{i,j=1}^k \left<\Omega_{\lambda_i},
\Omega_{\lambda_j}\right>\left< (I_\cK-A_jA_i^*)h_i, h_j\right>\\
&=\sum_{i,j=1}^k \Lambda_f(\lambda_j, \lambda_i)\left<
(I_\cK-A_jA_i^*)h_i, h_j\right>.
 \end{split}
 \end{equation*}
  Consequently, we
have $\|Y\|\leq 1$ if and only if
  the matrix $\left[  \frac{I_\cK-A_i A_j^*}{1-\sum_{ k=1}^n   f_k(\lambda_i)
\overline{f_k(\lambda_j)}}\right]_{m\times m}$ is positive
semidefinite. This completes the proof.
\end{proof}

\begin{corollary} Let $f=(f_1,\ldots, f_n)$ be  an  $n$-tuple of formal power
series  with the model property
 and let  $\lambda_1,\ldots, \lambda_m$ be $m$ distinct
points in    $\BB_{f}^<(\CC)$.  Given  $A_1,\ldots, A_m\in B(\cK)$,
the following statements are equivalent:
\begin{enumerate}\label{Nev-equi}
\item[(i)] there exists $\Psi \in  H^\infty(\BB_f)\bar\otimes B(\cK)$ such that
$$\|\Psi\|\leq 1\quad \text{and }\quad
\Psi(\lambda_j)=A_j,\ j=1,\ldots, m;
$$
\item[(ii)]
 there exists $\Phi \in  H^\infty(\BB_{f}^<(\CC))\bar\otimes B(\cK)$ such that
$$\|\Phi\|\leq 1\quad \text{and }\quad
\Phi(\lambda_j)=A_j,\ j=1,\ldots, m,
$$
where $H^\infty(\BB_{f}^<(\CC))$ is the algebra of multipliers of
$H^2(\BB_{f}^<(\CC))$;
 \item[(iii)] the operator matrix
\begin{equation*}
\left[  \frac{I_\cK-A_i A_j^*}{1-\sum_{ k=1}^n   f_k(\lambda_i)
\overline{f_k(\lambda_j)}}\right]_{m\times m}
\end{equation*}
is positive semidefinite.
\end{enumerate}
\end{corollary}

Using this corollary, we can obtain the the following result.

\begin{corollary}\label{cara-hol}
Let $f=(f_1,\ldots, f_n)$ be  an  $n$-tuple of formal power series
with the model property and let \,$\varphi $ be a
complex-valued function defined on
  $\BB_{f}^<(\CC)\subset \CC^n$.
 Then there exists $F\in H^\infty(\BB_f)$ with $\|F\|\leq 1$ such that
$$\varphi(z_1,\ldots, z_n)=F(z_1,\ldots, z_n)\quad \text{ for all }\ (z_1,\ldots,
z_n)\in \BB_{f}^<(\CC),
$$
if and only if for each $m$-tuple of  distinct points
$\lambda_1,\ldots, \lambda_m\in \BB_{f}^<(\CC)$, the matrix
\begin{equation*}
\left[  \frac{1-\varphi(\lambda_i)
\overline{\varphi(\lambda_j)}}{1-\sum_{ k=1}^n f_k(\lambda_i)
\overline{f_k(\lambda_j)}}\right]_{m\times m}
\end{equation*}
is positive semidefinite. In this case,   $\varphi$ is  a bounded
analytic function on  $\BB_{f}^<(\CC)$.
\end{corollary}

\begin{proof}
One implication follows from Corollary \ref{Nev-equi}. Conversely,
assume that $\varphi: \BB_{f}^<(\CC)\to \CC$ is such that the matrix
above is positive semidefinite for any  $m$-tuple of distinct points
$\lambda_1,\ldots, \lambda_m\in \BB_{f}^<(\CC)$. Let
$\{\lambda_j\}_{j=1}^\infty$ be a countable dense set in
 $\BB_{f}^<(\CC)$. Applying  Theorem \ref{Nev},
for each $m\in\NN$, we find  $F_m\in H^\infty(\BB_f)$ such that
$\|F_m\|\leq 1$ and
\begin{equation}
\label{F-varfi} F_m(\lambda_j)=\varphi(\lambda_j)\quad \text{ for }
\ j=1,\ldots,m.
\end{equation}
Since the Hardy algebra $H^\infty(\BB_f)$ is $w^*$-closed subalgebra
in $B(\HH^2(f))$ and $\|F_k\|\leq 1$ for any $m\in \NN$, we can use
Alaoglu's theorem to find a subsequence $\{F_{k_m}\}_{m=1}^\infty$
and $F\in H^\infty(\BB_f)$  such that
 $F_{k_m}\to F$, as $m\to\infty$, in the $w^*$-topology.
Since $\lambda_j:=(\lambda_{j1},\ldots, \lambda_{jn})\in
\BB_{f}^<(\CC)$, the $n$-tuple  is also of class $C_{\cdot 0}$. Due
to Theorem \ref{model} and Theorem \ref{pure-dilation}, the
$H^\infty(\BB_f)$-functional calculus for pure  $n$-tuples of
operators in $\BB_f(\cH)$  is $WOT$-continuous on bounded sets.
Consequently, we deduce that $F_{k_m}(\lambda_j)\to F(\lambda_j)$,
as $m\to\infty$, for any $j\in \NN$. Hence, and using relation
\eqref{F-varfi}, we obtain $\varphi(\lambda_j)=F(\lambda_j)$ for
$j\in \NN$. Given an arbitrary element $z\in \BB_{f}^<(\CC)$, we can
apply again the above argument to find $G\in H^\infty(\BB_f)$,
$\|G\|\leq 1$ such that
$$
G(z)=\varphi(z) \ \text{ and }\ G(\lambda_j)=\varphi(\lambda_j),\
j\in \NN.
$$
Due to Theorem \ref{eigenvalues}, the maps $\lambda\mapsto
G(\lambda)$ and $\lambda\mapsto F(\lambda)$  are analytic on
$\BB_{f}^<(\CC)$. Since they coincide on the set
$\{\lambda_j\}_{j=1}^\infty$, which is dense in $\BB_{f}^<(\CC))$,
we deduce that $G(\lambda)=F(\lambda)$ for any $\lambda\in
\BB_{f}^<(\CC)$. In particular, we have $F(z)=\varphi(z)$. Since $z$
is an arbitrary element in $\BB_{f}^<(\CC)$, the proof is complete.
\end{proof}

We remark that, in  the particular case when $f=(Z_1,\ldots, Z_n)$,
we recover some of the results obtained by Arias and the author and
Davidson and Pitts (see   \cite{Po-interpo}, \cite{ArPo2},
\cite{DP}).

\bigskip

       %

      \end{document}